\documentclass[11pt,reqno]{amsart}
\DeclareFontFamily{U}{rsfs}{} \DeclareFontShape{U}{rsfs}{n}{it}{<->
rsfs10}{} \DeclareSymbolFont{mscr}{U}{rsfs}{n}{it}
\DeclareSymbolFontAlphabet{\scr}{mscr}
\def\mathscr{\scr}
\usepackage{pgf}
\usepackage{pgfkeys}
\usepackage{url}
\usepackage{stmaryrd,
	amsmath,amsthm,amssymb,graphicx}
\usepackage{tikz}
\let\oldproofname=\proofname
\renewcommand{\proofname}{\rm\bf{\oldproofname}}

\usepackage{verbatim}
\usepackage{enumerate}
\usepackage[matrix,arrow,all,cmtip,2cell]{xy}
\UseTwocells

\usepackage{color}
\usepackage{amsmath,amssymb, amsfonts, amscd}
\usepackage[all]{xy}
\usepackage{verbatim}
\usepackage{enumerate}
\usepackage{yhmath}
\usepackage{xcolor}
\usepackage{hyperref}
\newtheorem{thm}{Theorem}[section]
\newtheorem{cor}[thm]{Corollary}

\newtheorem{con}[thm]{Conjecture}
\newtheorem{example}[thm]{Example}
\newtheorem{lem}[thm]{Lemma}

\newtheorem{prop}[thm]{Proposition}

\theoremstyle{definition}
\newtheorem{defn}[thm]{Definition}

\newtheorem{rem}[thm]{Remark}
\numberwithin{equation}{section}

\DeclareFontFamily{U}{rsf}{} \DeclareFontShape{U}{rsf}{m}{n}{
  <5> <6> rsfs5 <7> <8> <9> rsfs7 <10->  rsfs10}{}
\newcommand{\To}{\longrightarrow}

\newcommand{\bMap}{\mathbf{Map}}

\newcommand{\fkC}{\text{\bfseries\sf{C}}}

\renewcommand{\imath}{\sqrt{-1}}

\newcommand{\id}{\text{id}}

\DeclareMathOperator{\Hom}{Hom}

\def\PV{{\mathbin{\cal{PV}}}}

\def\e#1\e{\begin{equation}#1\end{equation}}
\def\ea#1\ea{\begin{align}#1\end{align}}

\theoremstyle{plain}% default
\def\vdim{\mathop{\rm vdim}\nolimits}
\def\dim{\mathop{\rm dim}\nolimits}

\def\id{\mathop{\rm id}\nolimits}

\def\Hom{\mathop{\rm Hom}\nolimits}

\def\bT{{\mathbin{\mathbb T}}}

\def\PV{\mathcal{P}}

\def\C{{\mathbin{\mathbb C}}}

\def\vp{\varphi}

\def\om{\omega}

\def\La{\Lambda}

\def\Tau{{\rm T}}

\def\sst{\scriptscriptstyle}

\def\ot{\otimes}

\def\cSz{{\mathbin{\cal S}\kern -0.1em}^{\kern .1em 0}}

\def\otL{{\kern .1em\mathop{\otimes}\limits^{\sst L}\kern .1em}}
\def\boxtL{{\kern .2em\mathop{\boxtimes}\limits^{\sst L}\kern .2em}}

\newcommand{\mbL}{\mathbb{L}}
\newcommand{\mbT}{\mathbb{T}}

\newcommand{\mS}{\mathcal{S}}
\newcommand{\mT}{\mathcal{T}}

\newcommand{\mP}{\mathcal{P}}

\newcommand{\Perf}{\textnormal{Perf}}

\newcommand{\fkLag}{{\text{\bfseries\sf{Lag}}}}
\newcommand{\fkLLag}{{\text{\bfseries\sf{LLag}}}}
\newcommand{\fkSymp}{{\text{\bfseries\sf{Symp}}}}
\newcommand{\fkLSymp}{{\text{\bfseries\sf{LSymp}}}}

\newcommand{\fkFill}{{\text{\bfseries\sf{Fill}}}}
\newcommand{\fkOr}{{\text{\bfseries\sf{Or}}}}
\newcommand{\fkCob}{{\text{\bfseries\sf{Cob}}}}
\numberwithin{equation}{section}

\begin{document}

\title[]{Perversely categorified Lagrangian correspondences}
\author{Lino Amorim and Oren Ben-Bassat}\thanks{}%
\address{Mathematical Institute, University of Oxford, ROQ, Woodstock Road, Oxford OX2 6GG, United Kingdom}
\address{Department of Mathematics, University of Haifa, Haifa, Israel}
\email{camposamorim@maths.ox.ac.uk, ben-bassat@math.haifa.ac.il}
\dedicatory{}
\subjclass{}%
\thanks{We would like to thank Dominic Joyce for sharing some of his insights on this topic and many helpful discussions. We would also like to thank Tony Pantev, Damien Calaque and Bertrand To\"{e}n for helpful conversations.
The first author was supported by EPSRC grant EP/J016950/1.
The second author acknowledges the support of the European Commission under the Marie Curie Programme for the IEF grant which enabled this research to take place. The contents of this article reflect the author's views and not the views of the European Commission. }
\keywords{}%

%\date{}%
%\dedicatory{}%
%\commby{}%
% ----------------------------------------------------------------
\begin{abstract}
In this article, we construct a $2$-category of Lagrangians in a fixed shifted symplectic derived stack S. The objects and morphisms are all given by Lagrangians living on various fiber products. A special case of this gives a $2$-category of $n$-shifted symplectic derived stacks $\fkSymp^n$. This is a $2$-category version of Weinstein's symplectic category in the setting of derived symplectic geometry. We introduce another $2$-category $\fkSymp^{or}$ of $0$-shifted symplectic derived stacks where the objects and morphisms in $\fkSymp^0$ are enhanced with orientation data. Using this, we define a partially linearized $2$-category $\fkLSymp$. Joyce and his collaborators defined a certain perverse sheaf on any oriented $(-1)$-shifted symplectic derived stack. In $\fkLSymp$, the $2$-morphisms in $\fkSymp^{or}$ are replaced by the hypercohomology of the perverse sheaf assigned to the $(-1)$-shifted symplectic derived Lagrangian intersections. To define the compositions in $\fkLSymp$ we use a conjecture by Joyce, that Lagrangians in $(-1)$-shifted symplectic stacks define canonical elements in the hypercohomology of the perverse sheaf over the Lagrangian. We refine and expand his conjecture and use it to construct $\fkLSymp$ and a $2$-functor from $\fkSymp^{or}$ to $\fkLSymp$. We prove Joyce's conjecture in the most general local model. Finally, we define a $2$-category of $d$-oriented derived stacks and fillings. Taking mapping stacks into a $n$-shifted symplectic stack defines a $2$-functor from this category to $\fkSymp^{n-d}$. 
\end{abstract}

\maketitle 

\tableofcontents

\section{Introduction}

Since the early stages of the development of Symplectic Geometry it is was clear the important role played by Lagrangian correspondences as natural generalizations of symplectomorphisms. Weinstein \cite{W} considered a symplectic ``category" where the set of morphisms between two symplectic manifolds $M_0$ and $M_1$ is the set of Lagrangian  correspondences, that is submanifolds of the product $M_0^- \times M_1$. Composition in this category should be defined as a fiber product, given Lagrangian correspondences $L_1 \to M_0^- \times M_1 $ and $L_2 \to M_1^- \times M_2 $, one considers the composition
$$L_1 \times_{M_1} L_2 \to M_0^- \times M_2.$$
If this fiber product is transversal then this is again a Lagrangian correspondence. Since we cannot guarantee transversality in general one is forced to work with ``categories" where the composition is only partially defined or to consider strings of correspondences as is done by Wehrheim and Woodward \cite{WW}. 
One then expects that symplectic invariants of symplectic manifolds can be made functorial with respect to Lagrangian correspondences. Weinstein's constructions were related to quantization where one associates to each symplectic manifold a linear space and to each Lagrangian a linear map. More recently Wehrheim and Woodward carried such a construction in the context of Floer theory (under some technical restrictions) \cite{WW}. Namely, they associated to each symplectic manifold its Donaldson--Fukaya category and to each Lagrangian correspondence a functor between those categories. The Donaldson--Fukaya category is a category whose objects are Lagrangian submanifolds and morphism spaces are the Floer cohomology groups. Moreover, they showed that this data can be assembled into a (weak) $2$-category, which they called the Weinstein--Floer $2$-category.

In this paper we explore similar ideas in the context of derived symplectic geometry recently introduced by Pantev, To\"en, Vaqui\'e and Vezzosi in \cite{PTVV}. It turns out that all the basic constructions involving Lagrangian correspondences have direct analogues in the derived world, with the advantage that all fiber products (in the homotopy sense) exist and so we don't have to worry with transversality conditions. However there are two new phenomena in the derived setting: 
\begin{itemize}
\item[1)]The intersection of two derived Lagrangians in a $n$-shifted symplectic derived  stack (or schemes) is naturally a $(n-1)$-shifted symplectic derived stack. This allows us to consider iterated Lagrangian correspondences.
\item[2)]Each (oriented) $(-1)$-shifted symplectic derived stack carries a natural perverse sheaf. Its hypercohomology will replace Floer cohomology in our context.
\end{itemize}

The first phenomenon is the starting point of our first main result. Since the (derived) intersection $X \cap Y$ of two $n$-shifted Lagrangians is $(n-1)$-shifted symplectic one can think of Lagrangians in $X \cap Y$ as ``relative" Lagrangian correspondences between $X$ and $Y$. It should be possible to iterate this construction and so define a symplectic $k$-category for each $k$. More precisely since derived stacks form an $\infty$-category one expects that there will be an $(\infty, k)$-category of ``relative" Lagrangian correspondences. This was already proposed by Calaque in \cite{Cal} and is subject of ongoing work by Haugseng \cite{ R, R2} , Li-Bland \cite{LB}, and others. Schreiber has written extensively on higher Lagrangian correspondences and their quantization in sections 1.2.10, 3.9.14, and 6 of \cite{Sch}.
In this paper, we only consider the case of $k=2$ and work with weak $2$-categories (also known as bicategories), as this is all we need for our main goal of constructing an analogue, in derived algebraic geometry of the Weinstein--Floer $2$-category. We prove the following
\begin{thm} 
Let $S$ be an $n$-shifted symplectic derived stack. There is a weak $2$-category $\fkLag(S)$ whose objects are (derived) Lagrangians in $S$, $1$-morphisms are ``relative" Lagrangian correspondences and $2$-morphisms are ``relative" Lagrangian correspondences between ``relative" Lagrangian correspondences.
\end{thm}

As observed by Calaque \cite{Cal}, the point is a $(n+1)$-shifted symplectic derived stack $\bullet_{n+1}$ and a Lagrangian in $\bullet_{n+1}$  is the same as a $n$-shifted symplectic derived stack. Therefore as a corollary of the above we obtain a $2$-category 
$$\fkSymp^n=\fkLag(\bullet_{n+1}),$$
with objects $n$-shifted derived symplectic stacks. Additionally, we show that this $2$-category is symmetric monoidal as defined in \cite{S-P}. 

These categories are highly non-linear and so not very manageable, our goal is to construct a linear version of the category $\fkSymp^n$, in the case of $n=0$. As a by-product we will also obtain a linearization of $\fkLag(S)$ when $S$ is a $1$-shifted symplectic derived stack. 
In order to do this, one first needs to chose some extra data on the Lagrangians, which goes by the general name of orientation data. This is closely related to Kontsevich--Soibelman orientation data \cite{KoSo}. Our version is partially inspired by the notion of relatively spin Lagrangian from Lagrangian Floer theory introduced in \cite{FOOO}. We describe in Theorem \ref{thm:sympor} a symmetric monoidal weak $2$-category  
\[\fkSymp_c^{or},
\]
whose objects are $0$-shifted symplectic derived stacks equipped with line bundles, $1$-morphisms are oriented $0$-shifted Lagrangians correspondences and $2$-morphisms are proper, oriented $(-1)$-shifted ``relative" Lagrangian correspondences. Also, there is a $2$-functor of symmetric monoidal $2$-categories $\fkSymp_c^{or} \to \fkSymp^{0}$ which forgets the orientation data. 

Our second main result is a linearization of $\fkSymp_c^{or}$. This is related to the second phenomenon mentioned above and can be thought as part of the programme by Joyce and his collaborators \cite{Joyc2,BBJ,BBDJS,BBBJ,JS}, on the categorification of Donaldson--Thomas invariants. One of the outcomes of this theory is that a $(-1)$-shifted symplectic derived stack $X$, together with some orientation data, carries a natural perverse sheaf $\mathcal{P}_X$. The idea of such perverse sheaf was due to Behrend who suggested it as a sort of categorification of the Behrend function, a function used to present Donaldson--Thomas invariants as a weighted Euler characteristic on the moduli space of sheaves \cite{Beh}.  For example in the case $X$ is the moduli space of coherent sheaves on a Calabi--Yau $3$-fold, the Euler characteristic of the hypercohomology of $\mathcal{P}_X$ is the Donaldson--Thomas invariant of the $3$-fold. In our setting we will be interested in the perverse sheaf in the intersection of two $0$-shifted Lagrangians, which is $(-1)$-shifted symplectic as proved in \cite{PTVV}. In fact, the intersection of any two Lagrangians carries such a perverse sheaf, even in the holomorphic case as proved by Bussi in \cite{Bu}.

Joyce conjectured that Lagrangians in a $(-1)$-shifted symplectic derived stack define sections of the sheaf $\mathcal{P}_X$. We reformulate this conjecture as a sort of ``quantization" of $(-1)$-shifted symplectic derived stacks. In other words, a $(-1)$-shifted Lagrangian correspondence $\phi: M \to X_0^- \times X_1$, together with some orientation data, determines a map 
$$\mu_M : \phi_0^* \mathcal{P}_{X_0}[\vdim M] \To \phi_1^! \mathcal{P}_{X_1},$$
in the derived category of constructible sheaves of $M$. We formulate a more detailed version of this in Conjecture \ref{con:OurConj}, namely we describe the behaviour of $\mu$ under composition of Lagrangian correspondences. We also give the local construction of the map $\mu_M$ in the case of derived schemes. 

Assuming this conjecture we prove the following 
\begin{thm}
There exists a symmetric monoidal weak $2$-category $\fkLSymp$, whose objects are $0$-shifted symplectic derived stacks with line bundles, $1$-morphisms are oriented Lagrangian correspondences and the space of $2$-morphisms between $X$ and $Y$ is the hypercohomology  $$\mathbb{H}^\bullet(\mathcal{P}_{X\cap Y}[-\vdim X]),$$
where $\vdim X$ is the virtual dimension of $X$.

Moreover there is a $2$-functor of symmetric monoidal $2$-categories $\fkSymp_c^{or} \to \fkLSymp$.
\end{thm}
Given two classical smooth algebraic Lagrangians (of  complex dimension $n$) in a smooth algebraic symplectic variety, these are examples of $0$-shifted Lagrangians, therefore their (derived) intersection $X\cap Y$ is $(-1)$-shifted symplectic. If this intersection is clean one can show that the hypercohomology $\mathbb{H}^\bullet(\mathcal{P}_{X\cap Y}[-\vdim X])$ is isomorphic to the Floer cohomology of the pair $X, Y$. In fact,  this should hold in general and therefore we expect that a version of the Weinstein--Floer $2$-category for holomorphic symplectic manifolds embeds in $\fkLSymp$. 

The method used to construct $\fkLSymp$ in fact applies to any $1$-shifted symplectic derived stack, not just to the point. Concretely let $S$ be a $1$-shifted symplectic derived stack, in Theorem \ref{thm:Lagor} we construct a $2$-category $\fkLag^{or}_c(S)$ whose objects are Lagrangians in $S$ equipped with line bundles. There is also a $2$-functor 
\[\fkLag_c^{or}(S) \to \fkLag(S),
\]
that forgets the orientation data. Then assuming Conjecture \ref{con:OurConj}, the above construction gives the following
\begin{thm}
For each $1$-shifted symplectic derived stack $S$ there is a weak $2$-category $\fkLLag(S)$ whose objects are Lagrangians in $S$ equipped with line bundles. The $1$-morphisms are ``relative" oriented Lagrangian correspondences and the space of $2$-morphisms between $X$ and $Y$ is the hypercohomology $\mathbb{H}^\bullet(\mathcal{P}_{X\cap Y}[-\vdim X])$.
Also there is a $2$-functor $\fkLag_c^{or}(S) \to \fkLLag(S)$.
\end{thm} 
It would be interesting to understand the information encoded by this $2$-category for some specific examples of $S$, such as the loop space of the classifying stack $BG$ of a reductive smooth group scheme $G$; or the moduli stack of perfect complexes in an elliptic curve. 

The Rozansky--Witten topological sigma-model \cite{KR} predicts the existence of $2$-category associated to each holomorphic symplectic manifold $S$ (in our context this is a $0$-shifted symplectic derived scheme), with objects holomorphic Lagrangians in $S$. The above theorem provides an (at least formally) analogous $2$-category in the $1$-shifted case.

\vspace{.2cm}
One of the main sources of examples of shifted symplectic derived stacks is the following construction. Let $S$ be a $n$-shifted symplectic derived Artin stack and let $X$ be a $\mathcal{O}$-compact derived stack equipped with an $\mathcal{O}$-orientation of dimension $d$, as defined in \cite{PTVV}. Rather informally this can be thought of as a volume form (of degree $d$) that allows us to ``integrate functions" on $X$.
Then the mapping stack $\bMap(X,S)$ inherits an $(n-d)$-shifted symplectic structure by a theorem of \cite{PTVV}. 

There is also a relative version of $\mathcal{O}$-orientation that we call {\it fillings} (or relative $\mathcal{O}$-orientations) of an $\mathcal{O}$-oriented derived stack $(X,[X])$. Heuristically, these are objects whose boundary is $X$. They were introduced by Calaque in \cite{Cal} that also proved that the mapping stack takes relative orientations to Lagrangians.

We elaborate on this constructions and show that there is a weak $2$-category $\fkFill(X)$ whose objects are fillings of $(X,[X])$. Moreover we show that there is a $2$-functor 
\[\mathcal{M}_S: \fkFill_{\fkC}(X) \To \fkLag\left(\bMap(X,S)\right),
\]
where $\fkFill_{\fkC}(X)$ is an appropriate subcategory of $\fkFill(X)$, dependent on $S$.

As a particular case when $X$ is the empty set thought of as a $(d-1)$-dimensional $\mathcal{O}$-oriented, derived stack then the bicategory $\fkOr_d:=\fkFill(\emptyset_{d-1})$ is a symmetric monoidal weak $2$-category of $\mathcal{O}$-compact derived stacks with $\mathcal{O}$-orientations of dimension $d$. We then have a symmetric monoidal $2$-functor 
\[\mathcal{M}_S: \fkOr^d_{\fkC} \To \fkSymp^{n-d}
\] 
determined by a $n$-shifted symplectic derived Artin stack $S$. 

Finally one can check that, similarly to \cite{Cal}, the Betti stack construction gives a $2$-functor from the $2$-category $\fkCob_d^{or}$ of cobordisms of $d$-dimensional manifolds, defined in \cite{S-P}, to $\fkOr^d_{\fkC}$. Therefore composing with $\mathcal{M}_S$ one obtains a symmetric monoidal $2$-functor
\[\mathcal{Z}_S: \fkCob_d^{or} \To \fkSymp^{n-d}.
\]
Thus providing examples of extended Topological Filed Theories (TFT) with values in our symplectic $2$-category.
Just as the shifted symplectic geometry of \cite{PTVV} can be thought of as a mathematically rigorous framework for understanding the AKSZ formalism \cite{AKSZ}, we hope this article will be used towards the understanding of classical BV theory (including boundaries) as in the article \cite{CMR}.

The above constructions, in the case that $n=d$, give rise to the following (in our opinion) very interesting question. Consider the following diagram of $2$-categories  
\[\xymatrix{ ? \ar[r] \ar[d] & \fkSymp_c^{or} \ar[d]^{}\ar[r]^(.6){} & \fkLSymp \\
\fkOr^d_{\fkC} \ar[r]^{\mathcal{M}_S} & \fkSymp^{0} &
}
\]
Is there a natural $2$-category that completes the above diagram? This should amount to finding, for each specific $S$, some geometric structure on $\mathcal{O}$-oriented (or relatively oriented) derived stacks that naturally induces orientations on the symplectic (or Lagrangian) derived stack $\bMap(-,S)$. There is work, related to this question, in the case of the moduli of coherent sheaves by Cao--Leung \cite{Cao} and Hua \cite{H}. We leave this problem for future work.

This paper is organized in the following way. In Section 2 we review some basics of shifted symplectic geometry and provide various constructions of symplectic and Lagrangian structures on derived intersections of Lagrangians. In Section 3 we define and study equivalences between shifted symplectic and Lagrangian derived stacks. In Section 4 we construct the $2$-category $\fkLag(S)$, using the results from the previous two sections. In Section 5 we study orientation data, define $\fkSymp^{or}$ and formulate Conjecture \ref{con:OurConj}. In Section 6 we construct $\fkLSymp$ and $\fkLLag(S)$ assuming Conjecture \ref{con:OurConj}. Finally, in Section 7 we construct the $2$-category of fillings $\fkFill(X)$ and construct the $2$-functors $\mathcal{M}_S$ and $\mathcal{Z}_S$.

\section{Derived Lagrangian Intersections}

\subsection{Review of shifted symplectic geometry}\

We will review some of the basics of shifted symplectic geometry following the work of Pantev, To\"{e}n, Vaqui\'{e} and Vezzosi \cite{PTVV}. 
We start by establishing some notation and conventions. We work relative to a fixed field $k$ of characteristic zero which we suppress from the notation with it being understood that everything is relative to this field. We often suppress pullbacks of (relative) tangent or cotangent complexes in order to simplify notation. Also, since all of our fiber products are homotopy fiber products, we denote them simply in the form $X\times_{Z} Y$ without any special emphasis on the fact that these are homotopy fiber products. The same goes for other derived functors. 

We assume that all the derived Artin stacks are locally of finite presentation. In particular given such a derived Artin stack $F$, its cotangent complex $\mathbb{L}_{F}$ is dualizable and hence we define its tangent complex $\mathbb{T}_F:=\mathbb{L}_{F}^\vee$.
We call a morphism of derived Artin stacks $f:X \longrightarrow Y$ {\it formally \'{e}tale} if the relative cotangent complex $\mathbb{L}_{f}$ vanishes. All the morphisms in this article are assumed to be homotopically finitely presentable and so we do not distinguish between formally \'{e}tale morphisms and \'{e}tale morphisms.  

Let $F$ be a derived stack. 
In \cite{PTVV}, the authors define a space $\mathcal{A}^{p}(F,n)$ of $n$-shifted $p$-forms on $F$ and a space of $n$-shifted closed $p$-forms $\mathcal{A}^{p,cl}(F,n)$. In order to define these one first considers the affine case. For a commutative (non-positively graded) dg-algebra $A$ one defines the simplicial sets
\[\mathcal{A}^{p}(A,n)= \arrowvert (\wedge^{p}\mathbb{L}_{A}[n],d)\arrowvert,
\]
\[\mathcal{A}^{p,cl}(A,n)= \arrowvert ( \prod_{i\geq 0} \wedge^{p+i}\mathbb{L}_{A}[-i+n], d+dR )\arrowvert,
\]
where $d$ is the differential induced by the differential in $\mathbb{L}_{A}$, $dR$ is the de Rham differential and $|-|$ is the realization functor. These define $\infty$-functors on the $\infty$-category of commutative dg-algebras that satisfies \'{e}tale descent, which can be then extended to two $\infty$-functors
\[\mathcal{A}^{p}(-,n), \ \mathcal{A}^{p,cl}(-,n): \mathbf{dSt}^{op} \To \mathbb{S},
\]
where $\mathbf{dSt}$ and $\mathbb{S}$ are the $\infty$-categories of derived stacks and simplicial sets.

The projection to the first component $\prod_{i\geq 0} \wedge^{p+i}\mathbb{L}_{A}[-i+n] \to \wedge^{p}\mathbb{L}_{A}[n]$ then induces a morphism 
\[\mathcal{A}^{p,cl}(F,n) \To \mathcal{A}^{p}(F,n).
\]
Given a closed $p$-form $\omega$ we call its image under this map the underlying $p$-form and denote it by $\omega^0$.
We would like to recall from \cite{PTVV}, that if $F$ is a derived Artin stack we have the following description of the space of $p$-forms, as the mapping space
\[\mathcal{A}^{p}(F,n) \cong Map_{L_{qcoh}(F)}(\mathcal{O}_{F},\wedge^{p}\mathbb{L}_{F}[n] )
\]
where  $L_{qcoh}(F)$ is the $\infty$-category of quasi-coherent sheaves on $F$.

\begin{defn}
Let $S$ be a derived Artin stack. An element $\omega \in \mathcal{A}^{2,cl}(S,n)$ is called an {\it $n$-shifted symplectic form} if the underlying $2$-form $\omega^0$ is {\it non-degenerate}. Non-degeneracy is the condition that the map induced by $\omega^0$: 
\[\Theta_{\omega}:\mathbb{T}_{S} \longrightarrow \mathbb{L}_{S}[n]
\]
is a quasi-isomorphism. We will denote by $Symp(S,n)$ the space of all symplectic forms in $S$. We will call a pair $(S,\omega)$ a $n$-symplectic derived stack.
\end{defn}

\begin{example}
A smooth symplectic scheme is an example of a $0$-symplectic derived stack. The point $Spec(k)$ admits an unique $n$-shifted symplectic form, for every $n$. We will denote this $n$-symplectic stack simply by $\bullet_{n}$. 
\end{example}

Suppose that $(S, \omega)$ is an $n$-symplectic derived stack and consider a morphism of derived Artin stacks $f:X \To S$. An isotropic structure on $f$ is an element $h \in P_{0, f^{*}\omega}(\mathcal{A}^{2,cl}(X,n))$, that is a path in $\mathcal{A}^{2,cl}(X,n)$ from $0$ to $f^{*}\omega$. This determines homotopy commutativity data for the diagram 
\begin{equation}
\xymatrix{\mbT_{X} \ar[d]\ar[ddr]^{0}& & &  \\ f^{*}\mbT_{S}   \ar[d]_{f^{*}\Theta_{\omega}} \\ f^{*}\mbL_{S}[n] \ar[r]& \mbL_{X}[n] \ultwocell\omit
 }
\end{equation}
Recall that we have an exact sequence $\mathbb{L}_{f}[n-1] \longrightarrow f^*\mathbb{L}_{S}[n] \longrightarrow \mathbb{L}_{X}[n] $, therefore $h$ induces a map 
\[\Theta_{h}:\mathbb{T}_{X} \longrightarrow \mathbb{L}_{f}[n-1].
\]
We say $h$ is non-degenerate, if this map is a quasi-isomorphism.

\begin{defn}\label{defn:Lagrangian}
Let $(S, \omega)$ be a $n$-symplectic derived stack. A Lagrangian structure on a morphism $f:X \to S$ is a non-degenerate isotropic structure $h$. We denote by $\mathcal{L}ag(f, \omega)$ as the set of Lagrangian structures on $f$. A {\it Lagrangian} in $(S, \omega)$ is a pair $(f,h)$ consisting of a morphism $f:X \to S$ and an element $h \in \mathcal{L}ag(f, \omega)$. The collection of Lagrangians in $(S, \omega)$  will be written $\mathcal{L}ag(S, \omega)$.
\end{defn}

A simple, but conceptually important observation from \cite{Cal} is the following description of Lagrangians in a point.

\begin{example}\label{point}
Let $\bullet_{n+1}$ be the point equipped with the canonical $(n+1)$-shifted symplectic structure, let $X$ be a derived Artin stack and let $X\stackrel{\pi}\longrightarrow \bullet_{n+1}$ denote the canonical morphism of derived Artin stacks. A Lagrangian structure on $\pi$ is equivalent to a $n$-shifted form on $X$. To see this note that, by definition, an isotropic structure $h$ on $\pi$ is a loop (based at $0$) in $\mathcal{A}^{2,cl}(X,n+1)$, thus $h$ determines a class in $\pi_1(\mathcal{A}^{2,cl}(X,n+1))\simeq \pi_0(\mathcal{A}^{2,cl}(X,n))$. Denote by $\omega$ this closed 2-form. It follows easily from the isomorphism $\mathbb{L}_{X}\simeq \mathbb{L}_{\pi}$ that non-degeneracy of $h$ implies non-degeneracy of $\omega$. Hence $\omega$ is a $n$-shifted symplectic structure on $X$.  
\end{example}

We end this subsection with the definition of the product and the opposite for $n$-symplectic derived stacks. It is straightforward to check that they are indeed $n$-symplectic derived stacks.

\begin{defn} The product of $n$-symplectic derived stacks $(S_0,\omega_0)$ and $(S_1,\omega_1)$ is given by $(S_0 \times S_1, \omega_0 \boxplus \omega_1:= p^{*}_0 \omega_0 +  p^{*}_1 \omega_1)$ where $p_0, p_1$ are the standard projections.

If $(S,\omega)$ is a $n$-symplectic derived stack we define its opposite as the $n$-symplectic stack $(S, -\omega)$. Often we will denote $(S,\omega)$ simply by $S$, in that case we use the notation $S^{-}$ for its opposite.
\end{defn}

\subsection{New Lagrangians out of old}\ 

In this subsection we will give several constructions of symplectic and Lagrangian structures obtained by considering various derived intersections of Lagrangians. The first construction of this type that serves as inspiration can be found in \cite[Theorem 2.9]{PTVV}. Calaque in \cite{Cal} proved that the classical result about composing Lagrangian correspondences holds in the shifted setting. 

Here we will show that all these constructions and a few new ones follow from one basic result, Proposition \ref{prop:Lag2Lag}, and two canonical Lagrangians: the diagonal and the triple intersection of Lagrangians \cite{B}.
We start with an elementary proposition.

\begin{prop}\label{product} 
Let $(S_0,\omega_0)$ be an $n$-symplectic derived stack and  $f: X \To S_0$ be a map of derived stacks. There is a canonical bijection
$$\mathcal{L}ag(f, \omega_0) \longrightarrow \mathcal{L}ag(f, -\omega_0).$$

Moreover given another $n$-symplectic derived stack $(S_1,\omega_1)$ and a map $g: Y \To S_1$ there is a canonical map
$$\mathcal{L}ag(f, \omega_0) \times \mathcal{L}ag(g, \omega_1)\longrightarrow \mathcal{L}ag(f\times g, \omega_0 \boxplus \omega_1).$$
\end{prop}

Next we show that, like in classical symplectic geometry, the diagonal map is Lagrangian.

\begin{prop}\label{prop:Diag}Let $(S,\omega)$ be an $n$-symplectic derived stack. Then the diagonal morphism
\[\Delta : S \To S^{-} \times S
\]
has a canonical Lagrangian structure.
\end{prop}
\begin{proof}
Denote by $p_0$ and $p_1$ the two natural projections $S^{-} \times S \To S$. By definition of $\Delta$, $p_0 \circ \Delta$ and $p_1 \circ \Delta$ are homotopic to $id_S$, therefore we have a natural path $l$ from $(p_0 \circ \Delta)^{*}\omega$ to $(p_1 \circ \Delta)^{*}\omega$ in 
$\mathcal{A}^{2,cl}(S,n)$. Translating $l$ by $(p_0 \circ \Delta)^{*}\omega$ we obtain a path from $0$ to $-(p_0 \circ \Delta)^{*}\omega+(p_1 \circ \Delta)^{*}\omega$ which we denote by $h$.  
Next we compute
\[\Delta^*(-\omega_0 \boxplus \omega_1)=\Delta^{*}(p_{0}^{*}(-\omega)+p_{1}^{*}(\omega) )= -(p_0 \circ \Delta)^{*}\omega + (p_1 \circ \Delta)^{*}\omega
\]
and conclude that $h$ is an isotropic structure on $\Delta$.

Next we check non-degeneracy of $h$. First note that $p_0 \circ \Delta \cong id_S$ gives the exact triangle 
\[\Delta^{*}\mbL_{p_0}\To \mbL_{\text{id}}\To \mbL_{\Delta}\To.
\]
This implies $\mathbb{L}_{\Delta}[-1] \simeq \Delta^*\mbL_{p_0}$, since $\mathbb{L}_{id}=0$. Therefore 
$$\mathbb{L}_{\Delta}[-1] \simeq \Delta^*\mbL_{p_0}\simeq \Delta^*p_1^*\mbL_{S}\simeq \mbL_{S},$$ 
since $p_1 \circ \Delta \cong id_S$ and $\mbL_{p_0}\simeq p_1^*\mbL_{S}$. Hence we obtain the following quasi-isomorphism 
\[ \mbT_S \stackrel{\Theta_\omega}\simeq \mbL{_S}[n] \simeq \mbL_{\Delta}[n-1],
\]
which we can easily see it is induced by $h$.
\end{proof}

From now on we will call a Lagrangian $f:X \longrightarrow S_0^-\times S_1$, a Lagrangian correspondence from $S_0$ to $S_1$.

The following proposition generalizes to the shifted setting a result in classical symplectic geometry: (under appropriate transversality assumptions) a Lagrangian correspondence induces a map from the set of Lagrangians in one factor to the other. This result follows from Theorem 4.4 in \cite{Cal}, but we prove it here for the sake of completeness.

\begin{prop}\label{prop:Lag2Lag}
Let $(S_0,\omega_0)$ and $(S_1,\omega_1)$ be $n$-symplectic derived stacks, let 
\[f=f_0 \times f_1: X \To S_{0}^{-} \times S_1
\] be a Lagrangian correspondence and let $g:N \To S_0$ be a morphism of derived stacks. There is a map 
\[\mathcal{L}ag(g, \omega_0) \longrightarrow \mathcal{L}ag(c_f(g), \omega_1)
\]
where $c_f(g)=f_1 \circ \pi_X: N \times_{g,S_0,f_0} X \To S_1 $.
\end{prop}
\begin{proof}
Let $h$ be the Lagrangian structure in $f$, that is a path from $0$ to 
\[f^{*}(-p_{0}^{*}\omega_{0} + p_{1}^{*}\omega_{1}) = -(p_0 \circ f)^{*}\omega_0 + (p_1 \circ f)^{*}\omega_1 = -f_{0}^{*}\omega_0 + f_{1}^{*}\omega_1. 
\]
As before, up to translations this is equivalent to a path from $f_{0}^{*}\omega_0$ to $f_{1}^{*}\omega_1$ (which we will still denote by $h$).
Consider the following (homotopy) commutative diagram

\[\xymatrix{ N\times_{S_0}X \ar[r]^(.6){\pi_X}\ar[d]^{\pi_N}& X  \ar[d]_{f_0}\\
N\ar[r]^{g} & S_0
}
\]
It gives us a path $l$ from $\pi_{N}^{*}g^{*}\omega_{0}$ to $\pi_{X}^{*}f_{0}^{*}\omega_{0}.$
Let $e$ be a Lagrangian structure on $g$. We define a path $H$ to be the concatenation $\pi_{N}^{*}e \bullet l \bullet \pi_{X}^{*} h$. This is a path from $0$ to $\pi_{X}^{*}f_{1}^{*}\omega_{1}$, in other words an isotropic structure on $c_f(g)$. 

We now need to check non-degeneracy of $H$. First observe that the map $c_f(g)$ is homotopic to the following composition
\[N\times_{S_0}X \xrightarrow{g\times_{S_0}f} S_{0} \times_{S_0}(S_0 \times S_1) \cong S_0 \times S_1 \stackrel{p}\longrightarrow S_1.
\]
This gives the exact triangle
\begin{equation*}
(g\times_{S_0}f)^{*}\mbL_{p} \To \mbL_{c_f(g)} \To \mbL_{(g,f)}\To,
\end{equation*}
which can be rewritten as 
\begin{equation*}
(f_{0}\circ \pi_{X})^{*}\mbL_{S_0} \To \mbL_{c_f(g)} \To \mbL_{g} \boxplus \mbL_{f}\To,
\end{equation*}
since $\mbL_{p} \simeq p^{*}_0 \mbL_{S_0}$ and $\mbL_{(g,f)}\simeq \mbL_{g} \boxplus \mbL_{f}.$
Rotating and shifting, we get the exact triangle
\begin{equation}\label{eqn:Tri3} 
\mbL_{c_f(g)}[n-1] \To \mbL_{g}[n-1] \boxplus \mbL_{f}[n-1] \To (f_{0}\circ \pi_{X})^{*}\mbL_{S_0}[n].
\end{equation}

Next we recall that, since $e$ and $h$ are Lagrangian structures we have the commutative squares
\[\xymatrix{\mbT_{N}\ar[r]\ar[d]_{\Theta_{e}} & g^{*}\mbT_{S_0}\ar[d]^{g^{*}\Theta_{\omega_0}} \\
\mbL_{g}[n-1]\ar[r] & g^{*}\mbL_{S_0}[n]
}
\]
and 
\[\xymatrix{\mbT_{X}\ar[r]\ar[d]_{\Theta_{h}} & f^{*}(\mbT_{S_0} \boxplus \mbT_{S_1})\ar[d]^{\Theta_{-\omega_0}\boxplus \Theta_{\omega_1}} \\
\mbL_{f}[n-1]\ar[r] & f^{*}(\mbL_{S_0} \boxplus \mbL_{S_1})[n]
}
\]
We can pull back both diagrams to $N \times_{S_0} X$ and assemble them into the following homotopy commutative square
\[\xymatrix{\mbT_{N}\boxplus \mbT_{X}\ar[r]\ar[d]_{\Theta_{e} \boxplus\Theta_{h}} & (f_0 \circ \pi_{X})^{*}\mbT_{S_0}\ar[d]^{\Theta_{\omega_0}} \\
(\mbL_{g} \boxplus \mbL_{f})[n-1] \ar[r] & (f_0 \circ \pi_{X})^{*}\mbL_{S_0}[n]
}
\]
and hence we get the commutative diagram 
\begin{equation*}
\xymatrix@=6em{\mathbb{T}_{N \times_{S_0} X}\ar[r] \ar@{.>}[d]_{\Theta_{H}}  &  \mathbb{T}_{N} \boxplus \mathbb{T}_{X} \ar[r] \ar[d]_{\Theta_{e} \boxplus\Theta_{h}} & (f_0 \circ \pi_X)^{*} \mathbb{T}_{S_0} \ar[d]_{\Theta_{\omega_0}}\\
\mathbb{L}_{c_{f}(g)}[n-1]\ar[r]   &  \mathbb{L}_{g}[n-1] \boxplus  \mathbb{L}_{f}[n-1]\ar[r] & (f_0 \circ \pi_X)^{*} \mathbb{L}_{S_0}[n].}
\end{equation*}
The top row is exact by general properties of homotopy fiber products and the bottom row is exact by (\ref{eqn:Tri3}). Therefore we conclude that $\Theta_H$ is a quasi-isomorphism, since $\Theta_{e}$, $\Theta_{h}$ and $\Theta_{\omega_0}$ are quasi-isomorphisms. This completes the proof that $H$ is a Lagrangian structure on $c_f(g)$.\end{proof}

\begin{defn}\label{def:bullet}
Let $(S_0,\omega_0)$ and $(S_1,\omega_1)$ be $n$-symplectic derived stacks and let $f=f_0 \times f_1: X \To S_{0}^{-} \times S_1$ be a Lagrangian correspondence. Given $g:N \To S_0$ a map of derived stacks, we define the map
\[C_{f}:\mathcal{L}ag(g, \omega_0) \longrightarrow \mathcal{L}ag(c_f(g), \omega_1)
\]
given by Proposition \ref{prop:Lag2Lag}, where $c_f(g)=f_1 \circ \pi_X: N \times_{g,S_0,f} X \To S_1 $.

We will sometimes use the notation $C_X$ instead of $C_f$. Also when the map $g$ and a particular Lagrangian structure $h$ are fixed we write $C_X(N)$ for the Lagrangian $C_{f}(h)$ on the map $c_f(g)$.
\end{defn}

We will now use the map $C_f$ (for several different Lagrangian structures $f$) to give several constructions of new Lagrangians out of old ones, which we will later use. The first one was already proved in \cite[Theorem 2.9]{PTVV}.

\begin{cor}\label{cor:Const}
Let $(S,\omega)$ be a $n$-symplectic derived stack and let $f: X \To S$ and $g:Y \To S$ be maps of derived stacks. There is a map
\[\mathcal{L}ag(f,\omega) \times \mathcal{L}ag(g,\omega) \To Symp(X \times_S Y, n-1) 
\]
\end{cor}
\begin{proof}
It follows from Proposition \ref{product} that there is a mapping
\begin{equation}\label{eqn:SymLagInt}\mathcal{L}ag(f,\omega) \times \mathcal{L}ag(g,\omega) \To \mathcal{L}ag(f\times g,-\omega \boxplus \omega).
\end{equation}
Proposition \ref{prop:Diag} determines a Lagrangian structure on the diagonal morphism $\Delta: S \longrightarrow S^- \times S$ which can be interpreted as a Lagrangian structure on the map $\Delta: S \longrightarrow (S^- \times S)^- \times \bullet_{n}$. Now by Proposition \ref{prop:Lag2Lag} we get a map
\[C_{\Delta}: \mathcal{L}ag(f \times g, -\omega \boxplus \omega) \To \mathcal{L}ag(c_\Delta(f \times g), \bullet_n).
\]
where $c_\Delta(f \times g)$ is the canonical map $(X\times Y) \times_{f\times g, S\times S, \Delta} S \To \bullet_n$. Now recall from Example \ref{point}, that a Lagrangian structure in the canonical map to the point is equivalent to a $(n-1)$-shifted symplectic structure on the domain. Therefore composing the above two maps we obtain a map 
$$\mathcal{L}ag(f,\omega) \times \mathcal{L}ag(g,\omega) \To Symp((X\times Y) \times_{f\times g, S\times S, \Delta} S, n-1),$$
which is the required map once we note that $(X\times Y) \times_{f\times g, S\times S, \Delta} S \cong X \times_S Y$.
\end{proof}

\begin{rem}\label{rem:FirstLoop}
Given Lagrangian structures $h_{f}$ on $f:X \to S$ and $h_{g}$ on $g:Y \to S$, the symplectic form which is produced from Corollary \ref{cor:Const} in $\mathcal{A}^{2,cl}(X\times_{S}Y, n-1)$ can be thought of as the loop at $0$ given by the concatenation 
\begin{equation}\label{eqn:firstloop}\xymatrix{\pi_{X}^{*}f^{*}\omega \ar@{-}|-*=0@{>}[dr]_{\pi_{X}^{*}h_{f}} & & \pi_{Y}^{*} g^{*} \omega \ar@{-}|-*=0@{>}[ll] \\
& 0 \ar@{-}|-*=0@{>}[ur]_{\pi_{Y}^{*}h_{g}} & 
}
\end{equation}
in $\mathcal{A}^{2,cl}(X\times_{S}Y, n)$ where the top path is induced by the homotopy between $g\circ \pi_{Y}$ and $f\circ \pi_{X}$. A Lagrangian structure on a morphism $\phi: N \to X \times_{S} Y$ is then a homotopy between the constant loop at $0$ in $\mathcal{A}^{2,cl}(N, n)$ to the pullback of (\ref{eqn:firstloop}) by $\phi$ which is a loop at $0$ in $\mathcal{A}^{2,cl}(N, n)$. This is equivalent to, in the path space $\mathcal{P}_0(\mathcal{A}^{2,cl}(N, n))$, to a path from $\phi^{*} \pi_{X}^{*} h_{f}$ to $\phi^{*} \pi_{Y}^{*} h_{g}$:
\[\xymatrix{\phi^{*} \pi_{X}^{*} h_{f} \ar@{-}|-*=0@{>}[r]_{}& \phi^{*} \pi_{Y}^{*} h_{g} }
\]
satisfying the following: when we evaluate at the endpoint we obtain the path in $\mathcal{A}^{2,cl}(N)$ from $\pi_X^*f^* \omega$  to $\pi_Y^*g^* \omega$ that is homotopic to the natural path induced by  $g\circ \pi_{Y}\cong f\circ \pi_{X}$.
\end{rem}
\begin{rem}\label{rem:InCase}
In the case that $S$ is the point $\bullet_n$ the map (\ref{eqn:SymLagInt}) takes the pair $(X,Y)$ of $(n-1)$-shifted symplectic stacks to $X^{-} \times Y$. Since in this case we do not write anything below the product symbol, it should not cause confusion that $X \times_{\bullet_n}Y= X^{-} \times Y$ as shifted symplectic derived stacks.
\end{rem}

In the next corollary we recover the result about composition of Lagrangian correspondences proved in \cite[Theorem 4.4]{Cal}.

\begin{cor}\label{cor:LagInt}
Let $(S_i,\omega_i)$ be $n$-symplectic derived stacks for $i=0,1,2$ and let $f:X \To S_0\times S_1$ and $g:Y \To S_1\times S_2$ be maps of derived stacks. There is a map 
\[\mathcal{L}ag(f, -\omega_0\boxplus \omega_1 ) \times \mathcal{L}ag(g, -\omega_1 \boxplus \omega_2) \longrightarrow \mathcal{L}ag(f\times_{S_1} g,-\omega_0\boxplus \omega_2),
\]
where $f\times_{S_1} g: X \times_{S_1} Y \To S_0 \times S_2$. When $S_0, S_1, S_2, f$ and $g$ are clear, we write this map as $(X, Y) \mapsto Y \bullet X$.
\end{cor}
\begin{proof} According to Proposition \ref{prop:Diag} the morphism 
\[\Delta: S_0 \times S_1 \times S_2 \longrightarrow S_0 \times S^{-}_1 \times S_1 \times S^{-}_2 \times S^{-}_0 \times S_2
\]
has a canonical Lagrangian structure. Using Proposition \ref{prop:Lag2Lag} and arguing as in the proof of Corollary \ref{cor:Const} we construct the map 
\[C_{\Delta} : \mathcal{L}ag(f, -\omega_0\boxplus \omega_1 ) \times \mathcal{L}ag(g, -\omega_1 \boxplus \omega_2) \longrightarrow  \mathcal{L}ag(c_\Delta(f\times g), -\omega_0\boxplus \omega_2), \]
where 
$$c_\Delta(f\times g): (X\times Y)\times_{S_0\times S_1\times S_1\times S_2} S_0\times S_1 \times S_2 \To S_0 \times S_2$$
is the natural map induced by $f\times g$. To complete the proof simply note that 
$$(X \times Y) \times_{S_{0} \times S_1 \times S_{1} \times S_2 } (S_0 \times S_1 \times S_2) \cong X_0 \times_{S_1} X_{1}.$$
\end{proof}

Our next goal is to prove a relative version of the previous corollary. In order to do that we need to use a theorem from \cite{B}. We give a slightly different proof here in order to match the spirit of the current article.

\begin{thm}\label{thm:ThreeLag} 
Let $(S,\omega)$ be a $n$-symplectic derived stack and let $f_i: X_{i} \longrightarrow S$ be Lagrangians, for $i=0,1,2$. Denote by $X_{ij} = (X_{i} \times_{S} X_{j}, \omega_{ij})$ the $(n-1)$-symplectic derived stacks constructed in Corollary \ref{cor:Const}. Then the natural morphism 
\[\varphi: Z=X_{0} \times_{S} X_{1}\times_{S} X_{2} \To X_{01} \times X_{12}\times X_{20}
\]
has a canonical Lagrangian structure.
\end{thm}
\begin{proof} 
The construction of a natural isotropic structure on the morphism $\varphi$ can be found in \cite{B} or in Proposition \ref{prop:Characterization}. We denote this isotropic structure by $H$ and show it is non-degenerate as follows.  
Using the canonical equivalence $Z \cong X_{01} \times_{X_1} X_{12}$, we let $\pi:Z \To X_1$ be the natural projection and get an exact triangle 
\[\mbT_{Z} \To \mbT_{X_{01}} \boxplus \mbT_{X_{12}} \To \pi^{*}\mbT_{X_1} \To
\]
If we denote by $q$ the composition 
\[Z\stackrel{\varphi}\longrightarrow X_{01} \times X_{12}\times X_{20} \xrightarrow{\pi_{20}} X_{20},
\]
where $\pi_{20}$ is the obvious projection, we obtain the exact triangle 
\[\varphi^{*} \mbL_{\pi_{20}} \To \mbL_{q} \To \mbL_{\varphi} \To. 
\]
Next we observe that as the following square is Cartesian 
\[\xymatrix{X_0 \times_{S} X_1 \times_{S} X_2 \ar[r]^{q} \ar[d]_{\pi} & X_0 \times_{S} X_{2} \ar[d]^{} \\
X_1 \ar[r]_{f_1} & S
}
\]
we have that $\mbL_{q} \cong \pi^{*} \mbL_{f_1}$. Also $\mbL_{\pi_{20}} \cong \mbL_{X_{01}} \boxplus \mbL_{X_{12}}$. Putting everything together we get the exact triangle 
\[\mbL_{X_{01}} \boxplus \mbL_{X_{12}} \To \pi^{*} \mbL_{f_1} \To \mbL_{\varphi} \To
\]
or equivalently, the exact triangle 
\[\mbL_{\varphi}[-1] \longrightarrow \mbL_{X_{01}} \boxplus \mbL_{X_{12}} \To \pi^{*} \mbL_{f_1} \To. 
\]

Now consider the following diagram  
\begin{equation}
\xymatrix@=6em{\mathbb{T}_{Z}\ar[r] \ar@{.>}[d]^{\Theta_{H}}  &  \mathbb{T}_{X_{01}} \boxplus \mathbb{T}_{X_{12}} \ar[r] \ar[d]^{\Theta_{\omega_{01}} \boxplus \Theta_{\omega_{12}}} & {\pi}^{*}\mbT_{X_1} \ar[d]^{\pi^{*}\Theta_{h_1}} \\
\mathbb{L}_{\varphi}[n-2]\ar[r]   &  \mbL_{X_{01}}[n-1] \boxplus  \mathbb{L}_{X_{02}}[n-1]\ar[r] & \pi^{*} \mathbb{L}_{f_{1}}[n-1],
}
\end{equation}
where $h_1$ is the Lagrangian structure in $f_1$. It follows from the construction of $\omega_{ij}$ that both squares commute. Moreover the above discussion shows that both rows are exact. Therefore we conclude that $\Theta_H$ is a quasi-isomorphism since the other two vertical arrows in the diagram are also quasi-isomorphism. One can see from the definition of $H$ in Proposition \ref{prop:Characterization} and a bit of diagram chasing that the left vertical map is in fact $\Theta_H$.
\end{proof}

As a corollary of this theorem we obtain a ``relative" version of Corollary \ref{cor:LagInt} that will later be used to define the composition of 1-morphisms (and vertical composition of 2-morphisms) in the 2-category we construct in Section 4.  

\begin{cor}\label{cor:Comp}
Let $(S,\omega)$ be a $n$-symplectic derived stack and $f_i: X_i \To S$ be Lagrangians, for $i=0,1,2$. Denote by $X_{ij}=(X_i \times_S X_j, \omega_{ij})$ the $(n-1)$-symplectic derived stack constructed in Corollary \ref{cor:Const}. Given morphisms $\phi: N_1 \To X_{01}$ and $\psi: N_2 \To X_{12}$, there is a map
\[\mathcal{L}ag(\phi, \omega_{01}) \times \mathcal{L}ag(\psi, \omega_{12}) \longrightarrow \mathcal{L}ag((\phi,\psi),\omega_{02})
\]
where $(\phi,\psi): N_1 \times_{X_1} N_2 \To X_{02}$ is the morphism induced by $\phi$ and $\psi$.
\end{cor}
\begin{proof}
Theorem \ref{thm:ThreeLag} defines a Lagrangian structure on the morphism 
\[\varphi: X_{0} \times_{S} X_{1} \times_{S} X_{2} \longrightarrow  (X_{01} \times X_{12})^- \times X_{02}. \]
Now we apply Proposition \ref{prop:Lag2Lag} to this Lagrangian structure and, as before, obtain a map  
\[C_\varphi: \mathcal{L}ag(\phi, \omega_{01}) \times \mathcal{L}ag(\psi, \omega_{12}) \longrightarrow \mathcal{L}ag(c_\varphi(\phi \times \psi),\omega_{02})
\]
where 
$$c_\varphi(\phi \times \psi): (N_1 \times N_2)\times_{X_{01} \times X_{12}} (X_{0}\times_{S} X_{1} \times_{S} X_{2}) \To X_{02},$$
is the natural map induced by $\phi \times \psi$. To complete the proof note that we have the following equivalences of derived stacks
\begin{align*}
(N_1 \times N_2)\times_{X_{01} \times X_{12}} &(X_{0}\times_{S} X_{1} \times_{S} X_{2})= \\
& =  (N_1 \times N_2)\times_{(X_0 \times_{S} X_1) \times(X_1 \times_{S} X_2)} (X_{0}\times_{S} X_{1} \times_{S} X_{2})  \\ 
& \cong (N_1 \times N_2) \times_{X_1 \times X_1} X_{1} \\ 
& \cong N_1\times_{X_1} N_2,
\end{align*}
where we have used the universal property of homotopy fiber products.
\end{proof}
\begin{rem}\label{rem:comp}
As we saw in Remark \ref{rem:FirstLoop}, the isotropic structure $h_{N_1}$ can be interpreted as a path between $\phi_{0}^{*}h_0$ and $\phi_{1}^{*}h_1$ in $\mathcal{P}_{0}(\mathcal{A}^{2,cl}(N_1))$. Under this interpretation, one can easily check that the isotropic structure constructed above is given by the following concatenation
\[\xymatrix{{\pi_{1}^{*}\phi_{0}^{*}h_0} \ar@{-}[r]|-*=0@{>}_{\pi_{1}^{*}h_{N_1}} &{\pi_{1}^{*}\phi_{1}^{*}h_1} \ar@{-}[r]|-*=0@{>} &{\pi_{1}^{*}\psi_{1}^{*}h_1} \ar@{-}[r]|-*=0@{>}_{\pi^{*}_{2}h_{N_2}} & {\pi_{1}^{*}\psi_{2}^{*}h_2} } 
\]
Where the middle path is induced by the homotopy commutativity of the following diagram:
\begin{equation}\label{eqn:CompExplain}
\xymatrix{& & N_{1} \times_{X_1} N_{2} \ar[dr]^{\pi_{2}} \ar[dl]_{\pi_1}& & \\
& N_1 \ar[dr]^{\phi_{1}} \ar[dl]_{\phi_0} & & N_2 \ar[dr]^{\psi_{2}} \ar[dl]_{\psi_1} & \\
 X_0 & & X_1&  & X_2 
}
\end{equation}
\end{rem}

We need one more map between sets of Lagrangian structures, which will be used later to define the horizontal composition in the $2$-category to be defined in Section 4. For this we need the appropriate Lagrangian correspondence. The next proposition will provide such a correspondence and also be useful to describe symplectomorphisms (which will be introduced in Section 3).

\begin{prop}\label{prop:5} 
Let $(S,\omega)$ be an $n$-symplectic derived stack and let $f:X\longrightarrow S$ and $g:Y\longrightarrow S$ be Lagrangians, with isotropic structures $h_f$ and $h_g$. Consider a morphism of derived stacks $\Delta:W \longrightarrow X\times_{S} Y$ and denote by $u$ and $v$ the compositions of $\Delta$ with the projections to $X$ and $Y$, respectively. Assume we are given a homotopy $H$ between the paths $u^{*}h_f$ and $v^{*}h_g$, which when evaluating at one endpoint gives a path between $u^{*}f^*\omega$ and $v^{*}g^*\omega$ homotopic to the path induced by $f\circ u \cong g\circ v$. 

If $u$ and $v$ are \'{e}tale then $H$ induces a Lagrangian structure on $\Delta$ with respect to the symplectic structure on $X\times_{S} Y$ constructed in Corollary \ref{cor:Const}. On the other hand, if $\Delta$ has a Lagrangian structure then $u$ is \'{e}tale if and only if $v$ is \'{e}tale.
\end{prop}
\begin{proof}
Consider the (homotopy) commutative diagram
\begin{equation}\label{eqn:diag}
\xymatrix@=4em{W\ar[dr]^{\Delta} \ar[drr]^{v} \ar[ddr]_{u}& & \\
& X\times_{S}Y\ar[d] \ar[r]& Y \ar[d]^{g} \\ & X \ar[r]^{f} & S
}
\end{equation}
Pulling back $\omega$ and the Lagrangian structures along the maps in this diagram gives rise to the following picture in  $\mathcal{A}^{2,cl}(W,n)$
\begin{equation}\label{equation}
\resizebox{10cm}{!}{
\xymatrix@=6em{& 0 \ar@{-}|-*=0@{>}[ddr]^>>>>>>>>>>>>>{\Delta^* \pi^*_{Y}h_g} \ar@{-}|-*=0@{>}[ddl]_{\Delta^* \pi^*_{X}h_f} \ar@{.}|-*=0@{>}[d]_{u^{*}h_f} \ar@{-}|-*=0@{>}[drr]^{v^{*}h_g}& &  \\
&  \ar@{.}|-*=0@{>}[dl] u^{*}f^{*}\omega \ar@{.}|-*=0@{>}[rr] &   & v^{*}g^{*}\omega \\
\Delta^{*}\pi_{X}^{*}f^{*}\omega \ar@{-}|-*=0@{>}[rr] &  &\Delta^{*}\pi_{Y}^{*}g^{*}\omega \ar@{-}|-*=0@{>}[ur]&
}
}
\end{equation}
The commutativity of the diagram $(\ref{eqn:diag})$ determines a $2$-simplex that fills the base of the diagram, i.e. it interpolates between the four ways of pulling back $\omega$. 
By definition, the boundary of the front triangle is the pullback by $\Delta$ of the loop that defines the $(n-1)$-shifted symplectic structure on $X\times_{S}Y$. By assumption, $H$ determines a $2$-simplex that fills the back triangle. All of the other faces of the pyramid are filled in by homotopies induced by the commutativity of the two triangles in the diagram in Equation \ref{eqn:diag}. Therefore, the front triangle bounds a $2$-simplex $\mathcal{A}^{2,cl}(W,n)$. This defines an isotropic structure $h_{\Delta}$ on the morphism $\Delta$.

In order to check the non-degeneracy of
\[\Theta_{\Delta}: \mathbb{T}_{W} \longrightarrow \mathbb{L}_{\Delta}[n-2]
\] notice that $\pi_{X} \circ \Delta$ is homotopic to $u$ which gives the exact triangle 
\[\Delta^{*}\mathbb{L}_{\pi_X}\longrightarrow \mathbb{L}_{\text{u}} \longrightarrow \mathbb{L}_{\Delta} \longrightarrow
\]
Now, because $u$ is \'{e}tale, $\mathbb{L}_{\text{u}}=0$ and so we get isomorphisms 
\[\mathbb{L}_{\Delta}[-1] \simeq \Delta^{*}\mathbb{L}_{\pi_X} \simeq \Delta^{*}\pi_{Y}^{*}\mathbb{L}_{g} \simeq v^*\mathbb{L}_{g}
\]
where the middle isomorphism follows from the fact that the square in (\ref{eqn:diag}) is Cartesian. 
By definition we have the exact triangle 
\[\mathbb{T}_{v}\longrightarrow \mathbb{T}_{W} \longrightarrow v^{*}\mathbb{T}_{Y} \longrightarrow.
\]
which implies that $\mathbb{T}_{W} \simeq v^{*} \mathbb{T}_{Y}$, since $v$ is \'{e}tale.
Putting together these equivalences, we obtain 
\begin{equation}\label{eqn:WgDel}
\mathbb{L}_{\Delta}[n-2]\simeq v^{*}\mathbb{L}_{g}[n-1] \stackrel{v^{*}\Theta_{g}}\longleftarrow v^{*} \mbT_{Y} \simeq \mathbb{T}_{W}.
\end{equation}
 where $\Theta_{g}$ is an equivalence because $g$ has a Lagrangian structure. One can see by diagram chasing that this chain of equivalences is precisely $\Theta_{\Delta}$. Tracing back the argument, if we start by assuming that $\Delta$ is Lagrangian and $u$ is \'{e}tale then (\ref{eqn:WgDel}) gives an equivalence $v^{*}\mathbb{T}_{Y} \simeq \mathbb{T}_{W}$ and so $v$ is \'{e}tale.
\end{proof}

\begin{rem} The reader may have wondered why $u^{*}\Theta_f$ and $v^{*}\Theta_g$ were not both used in the proof, but because we are assuming the existence of $H$, they do not really define different maps.
\end{rem}

As a simple corollary of Proposition \ref{prop:5} we have

\begin{cor}\label{cor:diag}
Let $(S, \omega)$ be a $n$-symplectic derived stack and $f:X \longrightarrow S$ a Lagrangian in $S$. Then the diagonal $\Delta_{X}:X \longrightarrow X \times_{S}X$ has a Lagrangian structure where $X\times_{S}X$ has the symplectic structure from Corollary \ref{cor:Const}.
\end{cor}
\begin{proof}
We take $X=Y$ and $\Delta$ to be the diagonal and $u$ and $v$ the identity morphisms in Proposition \ref{prop:5}. This gives an obvious choice for the (constant) homotopy $H$. 
\end{proof}

\begin{prop}\label{prop:6}
Let $(S, \omega)$ be a $n$-symplectic derived stack and $X_0, X_1$ and $X_2$ be Lagrangians in $S$. Consider the $(n-1)$-symplectic derived stacks $X_{01} = X_{0}\times_{S}X_{1}$, $X_{12} = X_{1}\times_{S}X_{2}$, and $X_{02} = X_{0}\times_{S}X_{2}$ determined by Corollary \ref{cor:Const} and let $M_{0}$ and $M_{1}$ be Lagrangians in $X_{01}$ and  $N_{0}$ and $N_{1}$ be Lagrangians in $X_{12}$. 
Corollary \ref{cor:Comp} defines two new Lagrangians $P_{0} = M_{0} \times_{X_1} N_{0} \To X_{02}$ and $P_{1} = M_{1} \times_{X_1} N_{1}\To X_{02}.$ 

Given morphisms $\alpha: U \To M_0 \times_{X_{01}} M_1$ and $\beta: V \To N_0 \times_{X_{12}} N_1$ there is a map
\[\mathcal{L}ag(\alpha, \omega_{M_{01}}) \times \mathcal{L}ag(\beta, \omega_{N_{01}})  \To \mathcal{L}ag(\alpha \times_{X_1}\beta, \omega_{P_{01}}),
\]
where $\omega_{M_{01}}, \omega_{N_{01}}$ and $ \omega_{P_{01}}$ are the $(n-2)$-shifted symplectic structures on  $M_{0} \times_{X_{01}} M_{1}$, $N_{0} \times_{X_{12}} N_{1}$ and $P_{0} \times_{X_{02}} P_{1}$, respectively, determined by Corollary \ref{cor:Const} and $\alpha \times_{X_1}\beta$ is the induced map
$$\alpha \times_{X_1}\beta: U \times_{X_1} V \To P_0 \times_{X_{02}} P_1.$$
\end{prop}
\begin{proof}
The proof is analogous to previous ones, first we claim that the natural map
\[\varphi: P_{0} \times_{X_0 \times_S X_{1} \times_{S} X_{2}} P_{1} \To (M_{0} \times_{X_{01}} M_{1}) \times (N_{0} \times_{X_{12}} N_{1}) \times (P_{1} \times_{X_{02}} P_{0}),
\]
has a Lagrangian structure.
To see this note that Corollary \ref{cor:LagInt} implies that $M_{0} \times_{X_0} M_{1}$ and  $N_{0} \times_{X_2} N_{1}$ are Lagrangians in $X_{1} \times_{S} X_{1}$. Also the diagonal $\Delta: X_{1} \To X_{1} \times_{S} X_{1}$ has a Lagrangian structure according to Corollary \ref{cor:diag}. Applying Theorem \ref{thm:ThreeLag} to these three Lagrangians we conclude that the triple intersection:
\begin{equation}
\begin{split}(& M_{0} \times_{X_0} M_{1})\times_{X_1 \times_S X_1} (N_{0} \times_{X_2} N_{1}) \times_{X_1 \times_S X_1}  X_{1}\\
 & \cong (M_0\times_{X_1}N_0)\times_{X_0 \times_{S} X_2} (M_1 \times_{X_1} N_1) \times_{X_1 \times_{S} X_1} X_1\\ 
 & \cong (M_0 \times_{X_1} N_0) \times_{X_0 \times_{S} X_1 \times_{S} X_2}(M_1 \times_{X_1} N_1) \\ 
 & = P_{0} \times_{X_0 \times_S X_{1} \times_{S} X_{2}} P_{1}
\end{split}
\end{equation}
is Lagrangian in the product 
\begin{equation*}
\begin{split}
& ((M_{0} \times_{X_0} M_{1})\times_{X_{11}} X_1) \times (X_1\times_{X_{11}}(N_{0} \times_{X_2} N_{1})) \times((N_{0}\times_{X_2} N_{1})\times_{X_{11}} (M_{0} \times_{X_0} M_{1})) \\
& \cong (M_{0} \times_{X_{01}}M_1) \times (N_{0} \times_{X_{12}} N_1) \times ((M_{1} \times_{X_1} \times N_1)\times_{X_0 \times_{S} X_2} (M_{0} \times_{X_1} N_0)) \\
& \cong (M_0 \times_{X_{01}} M_1) \times (N_{0} \times_{X_{12}} N_1) \times (P_1 \times_{X_{02}} P_0).
\end{split}
\end{equation*}
This proves the claim once we establish that the above equivalences preserve the symplectic structures, that is they are symplectomorphic, in next section notation. We omit the details of this. In Lemma \ref{lem:LagUniqP01} we will give an alternative description of this Lagrangian.

Now we apply Proposition \ref{prop:Lag2Lag} to this Lagrangian correspondence and obtain a map
\[C_\varphi: \mathcal{L}ag(\alpha, \omega_{M_{01}}) \times \mathcal{L}ag(\beta, \omega_{N_{01}})  \To \mathcal{L}ag(c_\varphi(\alpha \times \beta), \omega_{P_{01}}).
\]
To complete the proof we just need to check that $c_\varphi(\alpha \times \beta)= \alpha \times_{X_1} \beta$, for this note:
\begin{equation}
\begin{split}& (U \times V) \times_{(M_{0} \times_{X_{01}} M_{1})\times (N_{0} \times_{X_{12}} N_{1})} (P_0 \times_{X_0 \times_{S} X_1 \times_{S} X_2} P_1) \\
& \cong (U \times V) \times_{(M_{0} \times_{X_{01}} M_{1})\times (N_{0} \times_{X_{12}} N_{1})} ((M_0 \times_{X_{0}}N_0) \times_{X_0 \times_{S} X_1 \times_{S} X_2} (M_1 \times_{X_{1}}N_1) ) \\ 
& \cong (U \times V) \times_{(M_{0} \times_{X_{01}} M_{1})\times (N_{0} \times_{X_{12}} N_{1})} ((M_0 \times_{X_{01}}M_1) \times_{X_1} (N_0 \times_{X_{12}}N_1))
\\& \cong U\times_{X_1} V.
\end{split}
\end{equation}
\end{proof}

\begin{rem}\label{rem:fillingboxes}
We now explain the operation in Proposition \ref{prop:6} in a way that will be helpful later. Consider the following commutative diagram
\begin{equation}\label{eqn:2MorExplain}
\xymatrix{& & M_0 \ar[dr]^{\phi_{1}} \ar[dl]_{\phi_0}& & \\
& X_0  & U \ar[u]^(.4){\alpha_0} \ar[d]^(.4){\alpha_1} & X_1 & \\
  & & M_1\ar[ur]_{\psi_{1}} \ar[ul]^{\psi_0}  &  &  
}
\end{equation}
determining a map $\alpha:U \to M_{01}$.
Recall from Remark \ref{rem:FirstLoop} that a Lagrangian structure in $\phi$ is given by an appropriate path $h_0$ in $\mathcal{P}_0(\mathcal{A}^{2,cl}(M_0,n))$. Using this interpretation, an isotropic structure on $\alpha:U \to M_{01}$ is equivalent to a filling $H_{U}$ of the square
\begin{equation}\label{eqn:firstfill}
\xymatrixcolsep{5pc}\xymatrix{{\alpha_{0}^{*}\phi_{0}^{*} h_0} \ar @{} [dr] |{\boldsymbol{H_{U}}} \ar@{-}[r]|-*=0@{>}^{\alpha_{0}^{*}h_{M_0}} \ar@{-}[d]|-*=0@{>} & {\alpha_{0}^{*}\phi_{1}^{*} h_1} \ar@{-}[d]|-*=0@{>} \\
{\alpha_{1}^{*}\psi_{0}^{*} h_0} \ar@{-}[r]|-*=0@{>}_{\alpha_{1}^{*} h_{M_1}}  & {\alpha_{1}^{*}\psi_{1}^{*} h_1} }
\end{equation}
in the path space $\mathcal{P}_0(\mathcal{A}^{2,cl}(U,n))$ satisfying an additional requirement. Evaluating at the endpoint $H_U$ determines a $2$-simplex in $\mathcal{A}^{2,cl}(U,n)$ interpolating between the four ways of pulling-back $\omega$ to $U$, we require that this is homotopic to the $2$-simplex induced by the commutativity of (\ref{eqn:2MorExplain}). 

If we also consider 
\begin{equation}\label{eqn:2MorExplain2}
\xymatrix{& & N_0 \ar[dr]^{\tau_{2}} \ar[dl]_{\tau_1}& & \\
& X_1  & V \ar[u]^(.4){\beta_0} \ar[d]^(.4){\beta_1} & X_2 & \\
  & & N_1\ar[ur]_{\kappa_{2}} \ar[ul]^{\kappa_1}  &  &  
}
\end{equation}
the isotropic structure on $\alpha \times_{X_1} \beta$ constructed in Proposition \ref{prop:6} is the concatenation of the three squares in the diagram
\begin{equation}\label{eqn:threesquares}
\xymatrix@R=1.7pc{\pi_{U}^{*}\alpha_{0}^{*}\phi_{0}^{*} h_0 \ar @{} [ddrr] |{\pi_{U}^{*}\boldsymbol{H_{U}}} \ar@{-}[rr]^{\pi_{U}^{*}\alpha_{0}^{*}h_{M_0}}|-*=0@{>} \ar@{-}[dd]|-*=0@{>}
 && \pi_{U}^{*}\alpha_{0}^{*}\phi_{1}^{*} h_1 \ar@{-}[dd]|-*=0@{>} \ar@{-}[rr]|-*=0@{>} 
&& \pi_{V}^{*}\beta_{0}^{*}\tau_{1}^{*} h_1 \ar @{} [ddrr] |{\pi_{V}^{*}\boldsymbol{H_{V}}} \ar@{-}[rr]|-*=0@{>}^{\pi_{V}^{*}\beta_{0}^{*}h_{N_0}} \ar@{-}[dd]|-*=0@{>}
&& \pi_{V}^{*}\beta_{0}^{*}\tau_{2}^{*} h_2 \ar@{-}[dd]|-*=0@{>}
\\
&&&&&&  \\
\pi_{U}^{*}\alpha_{1}^{*}\psi_{0}^{*} h_0  \ar@{-}[rr]|-*=0@{>}_{\pi_{U}^{*}\alpha_{1}^{*}h_{M_1}}  
&& \pi_{U}^{*}\alpha_{1}^{*}\psi_{1}^{*} h_1 \ar@{-}[rr]|-*=0@{>}
&&\pi_{V}^{*}\beta_{1}^{*}\kappa_{1}^{*} h_1 \ar@{-}[rr]|-*=0@{>}_{\pi_{V}^{*}\beta_{1}^{*}h_{N_1}}  
&&\pi_{V}^{*}\beta_{1}^{*}\kappa_{2}^{*} h_2}
\end{equation}
where the filling of the middle square comes from the homotopy given by pulling $h_1$ back to $U \times_{X_1} V$ in the four different ways from $X_1$.
\end{rem}

\section{Symplectomorphisms and Lagrangeomorphisms}

In this section we will introduce the notions of equivalence of $n$-symplectic derived stacks and Lagrangians, which we will call \emph{symplectomorphism} and \emph{Lagrangeomorphism} respectively. We will then show that the Lagrangians constructed in Theorem \ref{thm:ThreeLag}  and Proposition \ref{prop:6} are unique up to Lagrangeomorphism and the operation defined in Corollary \ref{cor:Comp} is associative up to Lagrangeomorphism.

\begin{defn}
Let $S_0$ and $S_1$ be $n$-symplectic derived stacks. A \emph{symplectomorphism} is a  pair consisting of an equivalence $\phi:S_0 \longrightarrow S_1$ of derived stacks and a Lagrangian structure on the graph of $\phi$,
\[\Gamma_{\phi}:S_{0}\To S_{0}\times S_{1}.
\]
\end{defn}

\begin{defn}\label{defn:IsomOfLags}
Let $(S,\omega)$ be an $n$-symplectic derived stack and let $f_0:X_0 \longrightarrow S$ and $f_1:X_1\longrightarrow S$ be Lagrangians. A \emph{Lagrangeomorphism} consists of an  equivalence $\phi:X_0 \longrightarrow X_1$ of derived stacks together with a homotopy $f_{1}\circ \phi \cong f_0$ and a Lagrangian structure on the morphism 
\[\Gamma_{\phi}: X_{0}\longrightarrow X_{0}\times_{S} X_{1},
\]
induced by the graph of $\phi$ and the homotopy. Here we are using the $(n-1)$-symplectic structure on $X_{0}\times_{S} X_{1}$ from Corollary \ref{cor:Const}.
\end{defn}

\begin{rem}If we take $S=\bullet_{n}$ in Definition \ref{defn:IsomOfLags} then $X_0$ and $X_1$ are $(n-1)$-shifted symplectic derived Artin stacks and an isomorphism $\phi$ is a Lagrangeomorphism of these Lagrangians in $\bullet_{n}$ if and only if it is a symplectomorphism. 
\end{rem}

We now give two corollaries of Proposition \ref{prop:5}.

\begin{cor}\label{cor:EtaleHLag}Let $S_0$ and $S_1$ be $n$-symplectic derived stacks and let $\phi:S_0 \To S_1$ be an equivalence of derived stacks. A path $h$ in $\mathcal{A}^{2,cl}(S_0, n)$ between $\phi^{*} \omega_1$ and $\omega_0$ determines a Lagrangian structure on $\Gamma_{\phi}$ and so a symplectomorphism. On the other hand, any symplectomorphism determines such data $(\phi,h)$. 
\end{cor}
\begin{proof} Take $S$ to be a point in Proposition \ref{prop:5} and let $\Delta = \Gamma_{\phi}$.
\end{proof}

\begin{cor}\label{cor:HIsomLag} Let $(S,\omega)$ be an $n$-symplectic derived stack and $f:X\longrightarrow S$ and $g:Y\longrightarrow S$  be Lagrangians in $S$. Let $\phi:X \To Y$ be an equivalence of derived Artin stacks such that $g\circ \phi \cong f$. Let $H$ be  a homotopy in $\mathcal{P}_0(\mathcal{A}^{2,cl}(X_0,n))$ between $h_f$ and $\phi^{*}h_g$, which evaluates at the endpoint to a path homotopic in $\mathcal{A}^{2,cl}(X_0,n)$ to the path between $f^*\omega$ and $(g\circ\phi)^*\omega$ induced by $g\circ\phi\cong f$. Then $H$ induces a Lagrangian structure on $\Gamma_{\phi}: X \To X \times_{S} Y$, that is, a Lagrangeomorphism. Moreover any Lagrangeomorphism is determined in this way. 
\end{cor}
\begin{proof}
In Proposition \ref{prop:5}, take $\Delta=\Gamma_\phi$, $u=id$ and $v=\phi$. This immediately proves the statement.
\end{proof}

\begin{lem}\label{lem:Isom2equiv} Lagrangeomorphism is an equivalence relation for Lagrangians in a $n$-symplectic derived stack $(S,\omega)$.
\end{lem}
\begin{proof}
Let $f:X \To S$ be a Lagrangian, Corollary \ref{cor:diag} shows that the diagonal $\Delta:X \To X \times_{S} X$ is a Lagrangian which implies reflexivity, since $\Gamma_{\text{id}_X} = \Delta_X$. 

Next we show symmetry, let $g:Y\To S$ be another Lagrangian and suppose we have a Lagrangeomorphism $(\phi, H_\phi)$ from $X$ to $Y$. By definition $\phi:X \to Y$ is an equivalence of derived stacks so we can choose an inverse $\psi:Y \To X$. Then we have homotopies $f \circ \psi \cong g$ and $\phi\circ\psi\cong id$. This last homotopy induces a path from $\psi^{*}\phi^{*}h_g$ to $h_g$ which we concatenate with $\psi^{*}h_\phi$ to obtain a path from $h_g$ to $\psi^* h_f$. Corollary \ref{cor:HIsomLag} now shows that this data determines a Lagrangeomorphism from $Y$ to $X$. 

Consider two Lagrangeomorphisms $\phi_0: X_0 \longrightarrow X_1$ and $\phi_1:X_1 \longrightarrow X_2$
over $S$ given by Lagrangian structures on
\[\Gamma_{\phi_0}: X_0 \longrightarrow X_0 \times_{S} X_1 \ \ \ \ \ \text{and} \ \ \ \ \ \ \Gamma_{\phi_1}: X_1 \longrightarrow X_1 \times_{S} X_2.
\]
Corollary \ref{cor:Comp} implies that 
\[X_{0}\times_{X_1} X_1 \stackrel{q}\longrightarrow X_0 \times_{S} X_2
\]
is Lagrangian where $q$ is induced by $id_{X_0}\times\phi_1: X_0 \times X_1 \longrightarrow X_0 \times X_2$. Because there is an equivalence between $X_0 \times_{X_1} X_1$  and $X_0$ commuting up to homotopy with the morphisms $q$ and $\Gamma_{\phi_1 \circ \phi_0}$ over $X_0 \times_{S} X_2$ we can pullback this Lagrangian structure to $\Gamma_{\phi_1 \circ \phi_0}:X_0 \to X_0 \times_{S} X_2$. This gives a Lagrangeomorphism $X_0 \to X_2$ and hence proves transitivity.
\end{proof}

The next two propositions show that the operation defined in Corollary \ref{cor:Comp} is associative up to Lagrangeomorphism. Moreover the diagonal serves as a unity and Lagrangeomorphism are invertible with respect to this unit, again up to Lagrangeomorphism. From now on we refer to this operation as composition of relative Lagrangian correspondences.

\begin{prop}\label{prop:Associativity}  Let $X_i$, for $i=0,1,2,3$ be Lagrangians in a $n$-symplectic derived stack $S$ and consider Lagrangians $N_{1} \to X_{01}$, $N_{2} \to X_{12}$, $N_{3} \to X_{23}$. Applying Corollary \ref{cor:Comp} we obtain  Lagrangians $N_1 \times_{X_1} (N_2 \times_{X_2} N_3)$ and $(N_1 \times_{X_1} N_2) \times_{X_2} N_3$ in $X_{03}$. There is a canonical Lagrangeomorphism between them.
\end{prop}
\begin{proof}  
Let \[\rho: N_1 \times_{X_1} (N_2 \times_{X_2} N_3) \longrightarrow (N_1 \times_{X_1} N_2) \times_{X_2} N_3
\]
be one of the canonical equivalences coming from the universality of homotopy limits. Then $\rho$ (homotopy) commutes with the induced morphisms of the two sides to $X_0 \times_{S} X_3$. According to Corollary \ref{cor:HIsomLag} to determine a Lagrangeomorphism we need to construct a homotopy between the isotropic structure on $N_{1} \times_{X_1} (N_{2} \times_{X_3} N_{3})$ and the pullback by $\rho$ of the isotropic structure on $(N_1 \times_{X_1} N_2)\times_{X_3}N_3$. It will be clear from our construction that our homotopy will satisfy the additional requirement stated in the corollary.
Consider the commutative diagram:

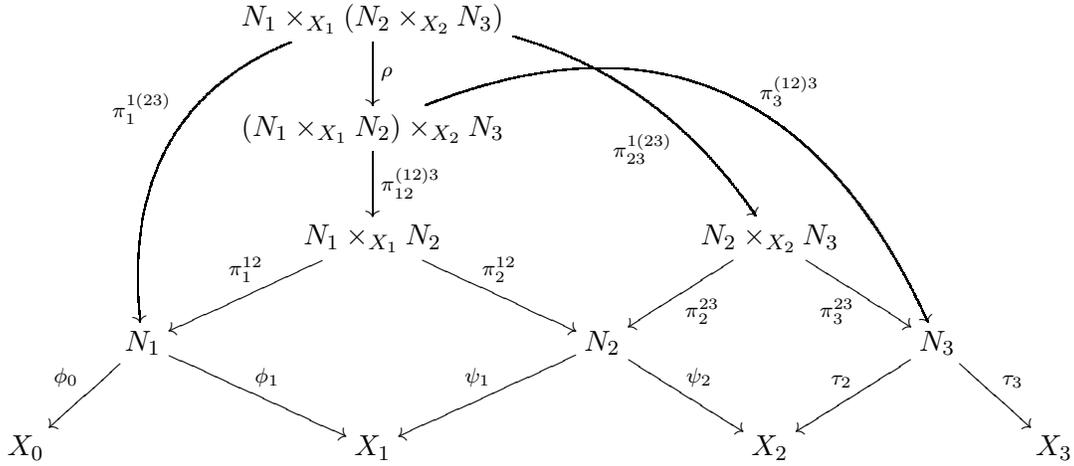
\[
\xymatrix{&& N_{1} \times_{X_1} (N_2 \times_{X_2} N_3) \ar[d]^{\rho} \ar@/^2.0pc/@[black][ddrr]_(.7){\pi^{1(23)}_{23}}  \ar@/_3.0pc/@[black][dddl]_{\pi_{1}^{1(23)}}&&& \\
&& (N_{1} \times_{X_1} N_2) \times_{X_2} N_3 \ar[d]^{\pi^{(12)3}_{12}} \ar@/^5.0pc/@[black][ddrrr]^(.6){\pi^{(12)3}_{3}} &&& \\
&& N_{1}\times_{X_1} N_2 \ar[dl]_{\pi^{12}_1}\ar[dr]^{\pi^{12}_{2}} && N_{2} \times_{X_2} N_3  \ar[dl]^{\pi_{2}^{23}} \ar[dr]_{\pi_{3}^{23}} &\\
& N_1 \ar[dl]_{\phi_0} \ar[dr]^{\phi_1}&& N_2 \ar[dr]^{\psi_2} \ar[dl]_{\psi_1}& &N_3\ar[dr]^{\tau_3}\ar[dl]_{\tau_2} \\
X_0 && X_1 &&X_2 && X_3
}
\]

Applying Remark \ref{rem:comp}, and working in the path space $\mathcal{P}_0(\mathcal{A}^{2,cl}(N_{1} \times_{X_1}(N_{2}\times_{X_2}N_{3}),n))$ the Lagrangian structure on $N_{1} \times_{X_1} (N_{2} \times_{X_3} N_{3})$ is given by the top row of the following diagram while the bottom row is the pullback by $\rho$ of the Lagrangian structure on $(N_{1} \times_{X_1} N_{2}) \times_{X_3} N_{3}$. 

\[\xymatrix{\pi^{1(23)*}_{1}\phi_{0}^{*}h_0  \ar@{-}[r]|-*=0@{>}\ar@{-}[d]|-*=0@{>}_{\pi^{1(23)*}_{1}h_{N_1}} 
 & \rho^{*}\pi^{(12)3*}_{12}\pi^{12*}_{1}\phi_{0}^{*}h_0  \ar@{-}[d]|-*=0@{>}^{\rho^{*}\pi^{(12)3*}_{12}\pi^{12*}_{1}h_{N_1}} \\ \pi^{1(23)*}_{1}\phi_{1}^{*}h_1  \ar@{-}[r]|-*=0@{>}\ar@{-}[d]|-*=0@{>}_{} 
 & \rho^{*}\pi^{(12)3*}_{12}\pi^{12*}_{1}\phi_{1}^{*}h_1  \ar@{-}[d]|-*=0@{>}^{} \\ \pi^{1(23)*}_{23}\pi^{23*}_{2}\psi_{1}^{*}h_1  \ar@{-}[r]|-*=0@{>}\ar@{-}[d]|-*=0@{>}_{\pi^{1(23)*}_{23}\pi^{23*}_{2} h_{N_2}} 
 & \rho^{*}\pi^{(12)3*}_{12}\pi^{12*}_{2} \psi^{*}_{1}h_1  \ar@{-}[d]|-*=0@{>}^{\rho^{*}\pi^{(12)3*}_{12}\pi^{12*}_{2}h_{N_2}} \\\pi^{1(23)*}_{23}\pi^{23*}_{2}\psi_{2}^{*}h_2  \ar@{-}[r]|-*=0@{>}\ar@{-}[d]|-*=0@{>}_{} 
 & \rho^{*}\pi^{(12)3*}_{12}\pi^{12*}_{2} \psi^{*}_{2}h_2  \ar@{-}[d]|-*=0@{>}^{} \\  \pi^{1(23)*}_{23}\pi^{23*}_{3} \tau^{*}_2 h_2   \ar@{-}[r]|-*=0@{>}\ar@{-}[d]|-*=0@{>}_{\pi^{1(23)*}_{23}\pi^{23*}_{3} h_{N_3}} 
 & \rho^{*}\pi^{(12)3*}_{3} \tau_{2}^{*}h_2  \ar@{-}[d]|-*=0@{>}^{\rho^{*}\pi^{(12)3*}_{3} h_{N_3}}   \\\pi^{1(23)*}_{23}\pi^{23*}_{3} \tau^{*}_3 h_3 \ar@{-}[r]|-*=0@{>} &\rho^{*}\pi^{(12)3*}_{3} \tau_{3}^{*}h_3 }
\]

The homotopy is given by patching together homotopies which fill in the squares in this diagram. The top square is filled in using the homotopy between $\pi_{1}^{12}\circ \pi_{12}^{(12)3} \circ \rho$ and $\pi_{1}^{1(23)}$. The next square to the down comes from the fact that there are four homotopic maps in the diagram from $N_{1} \times_{X_1}(N_{2}\times_{X_{2}}N_3)$ to $X_1$. The fourth one from the top is analogous to the second one and third and the fifth squares from the top are analogous to the top square. This completes the proof of the proposition.
\end{proof}

\begin{prop}\label{prop:IdOriginal}
Let $X_0$ and $X_1$ be Lagrangians in a $n$-symplectic derived stack $S$ and consider a Lagrangian $\phi:N \to X_{01}$. Then the Lagrangians $N \times_{X_1} \Delta_{X_1}$ and $\Delta_{X_0} \times_{X_0} N$ in $X_{01}$ are Lagrangeomorphic to $N$ by canonical Lagrangeomorphisms. 

If $N=\Gamma_\varphi$ is a Lagrangeomorphism then $M=\Gamma_\psi$, the graph of $\psi$ a homotopy inverse of $\varphi$, is a Lagrangeomorphism. Moreover $N \times_{X_1} M$ is Lagrangeomorphic to $\Delta_{X_0}$ and $M\times_{X_1}N $ is Lagrangeomorphic to $\Delta_{X_1}$.
\end{prop}
\begin{proof}
By definition of fiber product we can choose an equivalence of derived stacks $\rho: N \times_{X_1} \Delta_{X_1} \To N$. Now consider the following diagram 
\begin{equation}\label{eqn:IDspaces}
\xymatrix{& & N \times_{X_1} \Delta_{X_1} \ar[dl]_{\rho} \ar[dr]^{\pi_{X_1}}& & \\
& N \ar[dl]_{\phi_0} \ar[dr]^{\phi_1}& & \Delta_{X_1}\ar[dl]_{\text{id}} \ar[dr]^{\text{id}}& \\
X_0 && X_1 && X_1 }
\end{equation}
In the path space $\mathcal{P}_{0}(\mathcal{A}^{2,cl}(N \times_{X_1} \Delta_{X_1}, n))$ we have
\begin{equation}\label{eqn:IDPaths}
\xymatrix{\rho^{*}\phi^{*}_{0}h_0\ar@{-}[rr]|-*=0@{>}^{\rho^{*}h_{N}}  \ar@{-}[ddrr]|-*=0@{>}
 && \rho^{*}\phi_{1}^{*} h_1 \ar@{-}[ddrr]|-*=0@{>}\ar@{-}[rr]|-*=0@{>}
&& \pi_{X_1}^{*} h_1  \ar@{-}[rr]|-*=0@{>}^{} \ar@{-}[dd]|-*=0@{>}
&& \pi_{X_1}^{*}h_1 \ar@{-}[ddll]|-*=0@{>}
\\
&&&&&& \\
&& \rho^{*} \phi_{0}^{*}h_0  \ar@{-}[rr]|-*=0@{>}_{\rho^{*}h_{N}}
&&\rho^{*}\phi_{1}^{*} h_1   
&&}.
\end{equation}
where the unlabelled edges are homotopies determined by the commutativity of (\ref{eqn:IDspaces}). Then it follows from the definitions that the top path is the Lagrangian structure on $N\times_{X_1} \Delta_{X_1}$ while the bottom path is the Lagrangian structure on $N$. Again commutativity of the previous diagram provides homotopies filling the square and the triangles. This homotopy, together with $\rho$, determines the required Lagrangeomorphism.
The proof for $\Delta_{X_0} \times_{X_0}N$ is similar. 

For the second part of the statement we proceed as follows. The proof of Lemma \ref{lem:Isom2equiv} shows that $M$ is a Lagrangeomorphism. Then notice that $\Gamma_{\phi} \times_{X_1} \Gamma_{\psi}$ is equivalent as a derived stack over $X_{00}$ to $\Gamma_{\psi \circ \phi}$ which is equivalent to the diagonal $\Delta_{X_{0}}$. An argument analogous to the above shows that their Lagrangian structures are homotopic via this equivalence and so we get a Lagrangeomorphism between  $N \times_{X_1} M$ and $\Delta_{X_0}$. 
Finally the Lagrangeomorphism between $M\times_{X_1}N $ and $\Delta_{X_1}$ is constructed in the same way.
\end{proof}

The next few propositions characterize, up to Lagrangeomorphism, the Lagrangians we constructed in Theorem \ref{thm:ThreeLag} and in the proof of Proposition \ref{prop:6}, as well as a few other Lagrangians that we construct using the results from Section 2.

\begin{prop}\label{prop:Characterization} 
Let $(S, \omega)$ be $n$-symplectic derived stack and let $X_{0}, \dots, X_{m}$ be Lagrangians in $S$. Denote by $X_{ij}$ be the $(n-1)$-symplectic derived stacks $X_{i} \times_{S} X_j$ and consider the $(n-1)$-symplectic derived stack $W= X_{01} \times X_{12} \times \cdots \times X_{(m-1)m} \times X_{m0}$ with the product $(n-1)$-symplectic form $\omega_{W}$. We have the following
\begin{itemize}
\item[\textbf{(a)}] The canonical morphism
\[\phi:X_{01...m}= X_{0} \times_{S} X_{1} \times_{S} \cdots \times_{S} X_{m} \to W
\] 
has a Lagrangian structure
\item[\textbf{(b)}] The Lagrangian from (a) can be uniquely characterized as follows: any Lagrangian $\psi:N \to W$ satisfying conditions (1) and (2)  below is Lagrangeomorphic to $X_{01...m}$.
\begin{enumerate}
\item As a derived stack, $N$ is a homotopy limit of the following diagram 
\begin{equation}\label{eqn:Xij}
\xymatrix{& X_{01} \ar[dl]\ar[dr]&& X_{12}\ar[dl]\ar[dr] & &\cdots & X_{(m-1)m} \ar[dr]&& X_{m0} \ar[dl] \ar@/_0.5pc/[dllllllll]\\
X_0 \ar[drrrr] && X_1\ar[drr] &&X_2\ar[d]&\cdots && X_m \ar[dlll] \\ && && S &&&&}
\end{equation}
\item The isotropic structure on $\psi$, considered as a $2$-simplex in $\mathcal{A}^{2, cl}(N, n)$ with boundary the pullback of the loop defining $\omega_{W}$, is homotopic (relative to its boundary) to the $2$-simplex $\Theta_N:=\Theta+ \sum^{n}_{i=0} \Theta_{i}$. Here $\Theta_N$ is defined as follows: each of the isotropic structures $h_i$ in $X_i$ pulls back in two different ways to $N$, by definition of $N$ there is homotopy between these which we call $\Theta_{i}$. Note that since $h_i$ is a path in $\mathcal{A}^{2, cl}(X_i, n)$, $\Theta_i$ is a $2$-simplex in $\mathcal{A}^{2, cl}(N, n)$. Also, because $N$ is a homotopy limit, there is  a $2$-simplex $\Theta$ in $\mathcal{A}^{2, cl}(N, n)$ providing a homotopy between the $2(m+1)$ ways of pulling back $\omega$, along all the morphisms in the diagram, from $S$ to one of the $X_i$ and then to one of the $X_{ij}$ and finally to $N$. 
\end{enumerate}
\end{itemize}
\end{prop}
\begin{proof}  Part (a) was the main theorem in \cite{B} and this general case is entirely analogous to the special case discussed in Theorem \ref{thm:ThreeLag}. To prove (b) one must first observe that $X_{01...m}$ is a homotopy limit of the diagram (\ref{eqn:Xij}) and the isotropic structure on $\phi$ certainly satisfies these requirements as that is how it was constructed in \cite{B}.  The existence of the required Lagrangeomorphism now follows from Corollary \ref{cor:HIsomLag}.
\end{proof}

\begin{cor}\label{cor:NotationX}
Let $X_0,X_1,X_2,X_3$ be Lagrangians in $S$. Then we have the following Lagrangian correspondences
\begin{align*}
X_{012}\times \Delta_{X_{23}} \To (X_{01}\times X_{12}\times X_{23})^- \times (X_{02}\times X_{23})\\
\Delta_{X_{01}}\times X_{123}  \To (X_{01}\times X_{12}\times X_{23})^- \times (X_{01}\times X_{13}).
\end{align*}
The following two Lagrangian correspondences (obtained by composition) are Lagrangeomorphic 
\begin{equation}\label{eqn:ComplicatedXThing}
X_{023} \bullet (X_{012}\times \Delta_{X_{23}}) \cong X_{013} \bullet (\Delta_{X_{01}}\times X_{123})
\end{equation}
as Lagrangians in $(X_{01}\times X_{12}\times X_{23})^- \times X_{03}$.
\end{cor}
\begin{proof} 
To prove this we apply Proposition \ref{prop:Characterization} for $m=4$. It follows from general properties of fiber products that both sides of (\ref{eqn:ComplicatedXThing}) are equivalent to $X_{0123}$ as derived stacks. A long but straightforward check then shows that the Lagrangian structures are homotopic to the one described in Proposition \ref{prop:Characterization}. Therefore we can apply Proposition \ref{prop:Characterization} and prove the claim.
\end{proof}

We now give a characterization of the Lagrangian which appears in the proof of Proposition \ref{prop:6}. Recall the situation, we have Lagrangians $X_0,X_1, X_2$ in $S$ and Lagrangians $M_0, M_1$ in $X_{01}$ and $N_0, N_1$ in $X_{12}$. We denote by $P_i=M_i\times_{X_1}N_i$ the Lagrangians in $X_{02}$ obtained by composition of relative Lagrangian correspondences. In the proof of Proposition \ref{prop:6} we constructed a Lagrangian:
\begin{equation}\label{eqn:P01Lag}
P_{0} \times_{X_{012}} P_{1} \to
(M_{01} \times N_{01})^- \times P_{01}
\end{equation}
We will show this is unique in the appropriate sense. Consider the diagram:
\begin{equation}\label{eqn:P01Dag}
\xymatrix{ & M_{01}\ar[dl]\ar[d] & P_{01} \ar[dl] \ar[dll] \ar[dr] \ar[drr]& N_{01} \ar[dr]\ar[d]  & \\
M_0\ar[dr] \ar[drr]& M_1 \ar[d] \ar[dr]&&  N_0 \ar[d] \ar[dl]& N_1\ar[dl] \ar[dll] \\
&X_0 \ar[dr]& X_1 \ar[d] & X_2 \ar[dl] \\
&& S &&
}
\end{equation}
Let $K$ be a homotopy limit of the above diagram. The universal property of homotopy limit gives a map 
\begin{equation}\label{eqn:KLag}
K \to
(M_{01} \times N_{01})^- \times P_{01}
\end{equation}
We now construct directly an isotropic structure on the morphism (\ref{eqn:KLag}). Recall that the $X_i$ come with isotropic structures $h_i$, the $M_i$ with isotropic structures $h_{M_i}$ and the $N_i$ with isotropic structures $h_{N_i}$. Let $\Theta_{M_0}$ be the $2$-simplex in $\mathcal{P}_{0}(\mathcal{A}^{2,cl}(K,n))$ giving the homotopy between the two pullbacks of $h_{M_0}$ along $K \to P_{01} \to M_0$ and $K \to M_{01} \to M_0$ and similarly for $\Theta_{N_0}$, $\Theta_{M_1}$ and $\Theta_{N_1}$. Additionally denote by $\Theta_{i}$ the homotopy between  the four different pullbacks of $h_i$ to $K$.  

In the space $\mathcal{P}_{0}(\mathcal{A}^{2,cl}(K,n))$ (and suppressing pullbacks) we get the diagram
\begin{equation}\label{eqn:Po1Fig}
\resizebox{!}{.5cm}{
\xymatrix{&h_0 \ar@{}[ddr]^(.6){\Theta_{M_0}} \ar@/_3pc/[dddddd]|-*=0@{>} \ar@{-}[dd]|-*=0@{>}\ar@{-}[rr]|-*=0@{>}^{h_{M_0}}&&h_1 \ar@{-}[r]|-*=0@{>}\ar@{-}[dd]|-*=0@{>}& h_1  \ar@{}[ddr]^(.6){\Theta_{N_0}}\ar@{-}[dd]|-*=0@{>} \ar@{-}[rr]|-*=0@{>}^{ h_{N_0}}&&h_2 \ar@/^3pc/[dddddd]|-*=0@{>} \ar@{-}[dd]|-*=0@{>}\\
&&&&&&&&&&
\\
&h_0 \ar@{}[dl]^(.7){\Theta_{0}}  \ar@{-}[dd]|-*=0@{>}\ar@{-}[rr]|-*=0@{>}_{h_{M_0}}&&h_1 \ar@{}[dr]_(.7){\Theta_{1}} \ar@{-}[dd]|-*=0@{>}& h_1\ar@{-}[dd]|-*=0@{>} \ar@{-}[rr]|-*=0@{>}_{h_{N_0}}&&h_2  \ar@{}[dr]_(.7){\Theta_{2}}  \ar@{-}[dd]|-*=0@{>}
\\
&&&&&&&&&&
\\
&h_0  \ar@{}[ddr]^(.6){\Theta_{M_1}} \ar@{-}[dd]|-*=0@{>}\ar@{-}[rr]|-*=0@{>}_{h_{M_1}}&&h_1\ar@{-}[dd]|-*=0@{>}& h_1  \ar@{}[ddr]^(.6){\Theta_{N_1}} \ar@{-}[dd]|-*=0@{>}\ar@{-}[rr]|-*=0@{>}_{ h_{N_1}}&&h_2 \ar@{-}[dd]|-*=0@{>}
\\
&&&&&&&&&&
\\
&h_0 \ar@{-}[rr]|-*=0@{>}_{h_{M_1}}&&h_1 \ar@{-}[r]|-*=0@{>}& h_1 \ar@{-}[rr]|-*=0@{>}_{h_{N_1}}&&h_2
}}
\end{equation}

Note that, by definition, the boundary of the two unlabelled squares are (the pullback of) the $(n-2)$-shifted symplectic structures on $M_{01}$ and $N_{01}$, respectively. Also (the pullback of) the $(n-2)$-shifted symplectic structure on $P_{01}$ is the outside boundary of the diagram. The sum $\Theta_{0} + \Theta_{1} + \Theta_{2} + \Theta_{M_0} + \Theta_{N_0} + \Theta_{M_1}+ \Theta_{N_1}$ therefore gives a homotopy between $\omega_{M_{01}} + \omega_{N_{01}}$ and $\omega_{P_{01}}$ (suppressing pullbacks to $K$), that is an isotropic structure on (\ref{eqn:KLag}).

\begin{lem}\label{lem:LagUniqP01} 
Up to Lagrangeomorphism, there is a unique Lagrangian $K$ whose underlying derived stack is the homotopy limit of (\ref{eqn:P01Dag}) and whose isotropic structure is homotopic to the one explained above. Furthermore, the Lagrangian $P_{0} \times_{X_{012}} P_{1} \to
(M_{01} \times N_{01})^- \times P_{01}$ which appears in the proof of Proposition \ref{prop:6} has these properties.
\end{lem}
\begin{proof} The uniqueness follows immediately from the definition of Lagrangeomorphism. Tracing back through the construction in Proposition \ref{prop:6} we can see that the Lagrangian structure defined there is homotopic to the one just described above. As a result of this and Proposition \ref{prop:6} the isotropic structure on $K$ is in fact non-degenerate and so $K$ is Lagrangian.
\end{proof}

Consider now the same situation as described above but with extra Lagrangians $M_{2}$ in $X_{01}$ and $N_{2}$ in $X_{12}$. 
We will now prove one more uniqueness result for composition of the Lagrangians obtained in Lemma \ref{lem:LagUniqP01}. Consider the diagram

\begin{equation}\label{eqn:P02Dag}
\xymatrix{ & M_{01} \ar[d] \ar[dl] & M_{12}\ar[dl]\ar[d] & P_{02} \ar[dl] \ar[dlll] \ar[dr] \ar[drrr]&N_{01} \ar[d] \ar[dr]& N_{12} \ar[dr]\ar[d]  & \\
M_0 \ar[drr] \ar[drrr] & M_1\ar[dr] \ar[drr]& M_2 \ar[d] \ar[dr]&&  N_0 \ar[d] \ar[dl]& N_1\ar[dl] \ar[dll] & N_2  \ar[dll] \ar[dlll] \\
&&X_0 \ar[dr]& X_1 \ar[d] & X_2 \ar[dl]& \\
&& &S &&&
}
\end{equation}
A homotopy limit $K$ of this diagram has a natural morphism 
\begin{equation}\label{eqn:HoLimKP02}
K \to (M_{01} \times M_{12} \times N_{01} \times N_{12})^- \times P_{02}.
\end{equation}
We now construct an isotropic structure on this morphism. This is very similar to the discussion before Lemma \ref{lem:LagUniqP01}.  In the space $\mathcal{P}_{0}(\mathcal{A}^{2,cl}(K,n))$ we have $2$-simplices $\Theta_{M_i}$ for $i=0,1,2$, $2$-simplices $\Theta_{N_i}$ for $i=0,1,2$ and also $2$-simplices $\Theta_i$ for $i=0,1,2$. A diagram similar to (\ref{eqn:Po1Fig}) and similar considerations show that $\sum_{i=0}^{2} (\Theta_{M_{i}}+ \Theta_{N_{i}}+ \Theta_{{i}})$ determines an isotropic structure on the morphism (\ref{eqn:HoLimKP02}).

\begin{lem}\label{lem:LagUniqP02} 
Up to Lagrangeomorphism, there is a unique Lagrangian $K$ whose underlying stack is the homotopy limit of (\ref{eqn:P02Dag}) and whose isotropic structure is homotopic to the one explained above. 
One such Lagrangian $K$ can be constructed as the composition
\[
(P_0 \times_{X_{012}}P_2)\bullet(M_{012}\times N_{012})
\]
\end{lem}
 
We end this section by describing the behaviour of the operation $C_f$ defined in Proposition \ref{prop:Lag2Lag} under Lagrangeomorphism and composition of Lagrangian correspondences. 

\begin{prop}\label{prop:cXncYequiv}
Let $S_0$, $S_1$ and $S_2$ be $n$-symplectic derived stacks and let $f:X\longrightarrow S_0^{-} \times S_1$, $g:Y\longrightarrow S_0^{-} \times S_1$ and $h:Z \longrightarrow S_1^- \times S_2$ be Lagrangian correspondences and consider Lagrangians $e:N \longrightarrow S_0$ and $e':N'\longrightarrow S_0$. We have the following:
\begin{itemize}
\item[\textbf{(a)}] If $X$ is Lagrangeomorphic to $Y$ and $N$ is Lagrangeomorphic to $N'$ then $C_{X}(N)$ and $C_Y(N')$ are Lagrangeomorphic in $S_1$. 
\item[\textbf{(b)}] We have a Lagrangeomorphism
\[C_{Z}(C_{X}(N)) \cong C_{Z \bullet X}(N),\]
where $Z \bullet X$ is the Lagrangian constructed in Corollary \ref{cor:LagInt}.
\end{itemize}
\end{prop}
\begin{proof} 
The first part of the statement is easy and left to the reader. The second part follows from Proposition \ref{prop:Associativity} by taking $S= \bullet_{n+1}$, $X_{0} = \bullet_{n}$, $X_1=S_0$, $X_2=S_1$ and $X_3=S_2$, $N_1=N$, $N_2=X$ and $N_3=Z$ in that proposition.
\end{proof} 

\begin{rem}\label{rem:OtherWay}
A different but equivalent way to establish associativity, that is to prove Proposition \ref{prop:Associativity}, would be to first prove Proposition \ref{prop:cXncYequiv} and use part (b) to show that
\[
(N_1 \times_{X_1} N_2) \times_{X_2} N_3 \cong C_{X_{023} \bullet (X_{012}\times \Delta_{X_{23}})}(N_{1}\times N_{2}\times N_3)
\]
and
\[
N_1 \times_{X_1} (N_2 \times_{X_2} N_3) \cong C_{ X_{013} \bullet (\Delta_{X_{01}}\times X_{123})}(N_{1}\times N_{2}\times N_{3}).
\]
Then Corollary \ref{cor:NotationX} and part (a) of Proposition \ref{prop:cXncYequiv} imply associativity.
\end{rem}
 
\section{A $2$-category of Lagrangians}

Fix a $n$-symplectic derived stack $(S, \omega)$. In this section we will define a bicategory (or weak $2$-category) $\fkLag(S,\omega)$, whose objects are Lagrangians in $S$. 
But first we review some basics of bicategories (following \cite{S-P}) and setting up some notation.

\subsection{Basics of $2$-categories}\

We start by defining a weak $2$-category, also called a bicategory. From now on we adopt the name bicategory for simplicity.

\begin{defn}
A bicategory (or weak $2$-category) $C$ consists of the following data:
\begin{enumerate}
\item a collection of objects $C_0$; 
\item for each two objects $X,Y$ a category $C(X,Y)$;
\item for any three objects $X_0,X_1,X_2$, composition functors 
 \begin{equation*}
 \mu_{XYZ}: C(Y,Z) \times C(X,Y) \longrightarrow C(X,Z);  
 \end{equation*}
\item for each object $X$, an object $I_{X} \in C(X,X)$;
\item for any four objects $W,X,Y,Z$, natural isomorphisms 
\begin{align*}
a : & \ \mu_{WXZ} \circ (\mu_{XYZ} \times \id_{C(W,X)}) \to \mu_{WYZ} \circ(\id_{C(Y,Z)} \times \mu_{WXY}) & \textrm{(associator)} \\
l : & \ \mu_{XYY}\circ (I_{Y}\times \id_{C(X,Y)}) \to \id_{C(X,Y)} & \textrm{(left unitor)} \\
r : & \ \mu_{XXY}\circ (\id_{C(X,Y)}\times I_{X}) \to \text{id}_{C(X,Y)} & \textrm{(right unitor)}
\end{align*}
and therefore $2$-isomorphisms $a_{hgf}: (h \circ g)\circ f \to h \circ(g \circ f)$, $l_f : I_Y \circ f \to f$ and $r_f:f \circ I_X \to f$.
\end{enumerate}
These are required to satisfy the pentagon and triangle axioms which we will write explicitly in Lemmas \ref{lem:weakAssoc1mor} and \ref{lem:UnitorsSatisfy}.
\end{defn} 

Given objects $X,Y$ in $C$, the objects of $C(X,Y)$ are called $1$-morphisms of $C$, this collection is denoted $C_1(X,Y)$. For each object $X$ in $C$ we call $I_{X}$ the identity $1$-morphism. The morphisms of $C(X,Y)$ are referred to as $2$-morphisms of $C$. Given $f,g \in C_{1}(X,Y)$, we write $C_{2}(f,g)$ for the set of morphisms from $f$ to $g$ in $C(X,Y)$. For any $1$-morphism $f$ we call $1_f \in C_{2}(f,f)$ the identity $2$-morphism. We will use the notation $\mu_{XYZ}(g,f) =g \circ f$ for $f$ an object of $C(X,Y)$ and $g$ an object of $C(Y,Z)$. For $\alpha$ a morphism in  $C(X,Y)$ and $\beta$ a morphism in  $C(Y,Z)$ we write  $\mu_{XYZ}(\beta,\alpha) =\beta * \alpha$, this is called the horizontal composition of $\alpha$ and $\beta$. The composition in the categories $C(X,Y)$ is called vertical composition and for $\xi, \eta$ composable morphisms in $C(X,Y)$ we write their composition as $\eta \odot \xi$.

\begin{defn}\label{def:functor}
Let $C$ and $D$ be bicategories. A homomorphism of bicategories (or weak $2$-functor) $F:C\to D$ consists of the following data
\begin{enumerate}
\item a function $F: C_0 \To D_0$; 
\item for each two objects $X,Y$, a functor $F_{XY}: C(X,Y) \To D(F(X),F(Y))$;
\item for any three objects $X,Y,Z$, natural isomorphisms
\begin{align*}
F_{XYZ} : & \ \mu^D_{F(X)F(Y)F(Z)} \circ (F_{YZ} \times F_{XY}) \to F_{XZ} \circ \mu^C_{XYZ}, \\
F_X : & \ F_{XX}\circ I^C_X\to I^D{F(X)}. 
\end{align*}
and therefore $2$-isomorphisms $F_{g,f}: F(g)\circ F(f) \to F(g\circ f)$ and $F_X: F(I^C_X) \to I^D_{F(X)}$.
\end{enumerate}
These must satisfy the following identities
\begin{align*}
F_{g,f\circ e}\odot (1_{F(g)}*F_{f,e})\odot a^D_{F(g)F(f)F(e)} & = F(a^C_{gfe})\odot F_{g\circ f, e} \odot (F_{g,f}*1_{F(e)}), \\
F(r^C_f)\odot F_{f,I_X} & = r^D_{F(f)}\odot (1_{F(f)}*F_X), \\
F(l^C_f)\odot F_{I_Y,f} & = l^D_{F(f)}\odot (F_Y*1_{F(f)}).
\end{align*}
\end{defn}

One can also define bicategories enriched in some symmetric monoidal category $\textbf{M}$. We require that, for each pair of objects $X,Y$ in $C$, the category $C(X,Y)$ is enriched over $\textbf{M}$, the functors $\mu_{012}$ are functors of $\textbf{M}$-enriched categories and finally $a, r$ and $l$ are $\textbf{M}$-natural transformations.

In this article, we will construct bicategories enriched over two different symmetric monoidal categories (other than the category of sets). 
First we will take $\textbf{M}$ to be the category  $\textbf{gr-Vect}$ of graded vector spaces (and degree preserving homomorphisms), with the monoidal structure given by the tensor product and symmetric structure given by $a\otimes b \to (-1)^{|a||b|} b \otimes a$, where $|\cdot|$ stands for degree.

We will also consider a less common symmetric monoidal category, that we will denote by $\textbf{gr-Inv}$. Objects are sets $S$ equipped with an involution $-1:S\to S$ and a degree function $|\cdot|:S \to \mathbb{Z}$, invariant under $-1$, and whose morphism are maps that preserve both structures. The monoidal structure is defined as
\[S\otimes T = S\times T/\sim,
\]
where we take the equivalence relation $(-s,t)\sim(s,-t)$ and define the degree in the product as the sum of the degrees on each factor. The symmetric structure is defined as $(s,t)\to(-1)^{|s||t|}(t,s)$.

The most visible difference between a standard bicategory and a bicategory enriched over one of the symmetric monoidal categories just described is in the {\em compatibility between the horizontal and vertical composition} of $2$-morphisms. This is equivalent to the statement that $\mu_{XYZ}$ is a functor, and for $1$-morphisms $f_j \in C_1(X,Y)$ and $g_j \in C_1(Y,Z)$ (for $j=0,1,2$) it states
\[(\eta_{2} * \xi_{2}) \odot (\eta_{1} * \xi_{1}) = (\eta_{2} \odot \eta_{1}) * (\xi_{2} \odot \xi_{1})
\]
for any $2$-morphisms $\xi_i \in C_2(f_{i-1},f_i)$ and $\eta_i \in C_2(g_{i-1},g_i)$ for $i=1,2$. While in the enriched case it reads
\[(\eta_{2} * \xi_{2}) \odot (\eta_{1} * \xi_{1}) = (-1)^{|\eta_1||\xi_2|}(\eta_{2} \odot \eta_{1}) * (\xi_{2} \odot \xi_{1}).
\]

Similarly we can easily define homomorphism of bicategories enriched over    $\textbf{M}$. Or, if we are given a symmetric monoidal functor $\textbf{M}_1 \to \textbf{M}_2$, we can also consider homomorphisms from a $\textbf{M}_1$-enriched bicategory to a $\textbf{M}_2$-enriched bicategory.

For us the relevant examples are the following functors. First we have $\textbf{gr-Inv} \to \textbf{Sets}$ which sends $(S, -1, |-|)$ to the quotient $S/-1$ and forgets the grading. Secondly we have $\textbf{gr-Inv} \to \textbf{gr-Vect}$ which sends $(S, -1, |-|)$ to the graded vector space generated by the set $S/-1$.

Finally we will also need the definitions of a {\em symmetric monoidal bicategory} and homomorphisms of these. We refer the reader to Chapter 2.2 of \cite{S-P} for these.

\subsection{The $2$-category $\fkLag(S)$}\

We now define the objects, morphisms and compositions in what will become our bicategory of Lagrangians. We will denote the bicategory by $\fkLag(S,\omega)$, or $\fkLag$, when $(S,\omega)$ is fixed. 

\begin{defn}\label{defn:THECAT}
Let $(S, \omega)$ be an $n$-symplectic derived stack. The objects of $\fkLag(S,\omega)$ are Lagrangians in $(S, \omega)$.
Given two  Lagrangians $X_{0}$ and $X_{1}$ in $S$ the $1$-morphisms are defined to be 
\[\fkLag_1(X_0,X_1)= \mathcal{L}ag(X_0 \times_{S} X_{1}).\] 
Given two Lagrangians $N_{0},N_{1}$ in $X_0 \times_{S} X_{1}$, we define the set of $2$-morphisms between them as \[\fkLag_2(N_0,N_1)=\mathcal{L}ag(N_{0} \times_{(X_{0}\times_{S} X_{1})} N_{1})/\sim,\] 
that is, Lagrangeomorphism equivalence classes of Lagrangians.
In the definitions of $1$-morphisms and $2$-morphisms we chose a model for the homotopy fiber products and use the $(n-1)$-symplectic derived stack $X_{01}=(X_0 \times_{S} X_{1}, \omega_{01})$ and $(n-2)$-symplectic derived stack $N_{01}=(N_{0} \times_{(X_{0}\times_{S} X_{1})} N_{1},\omega_{N_{01}})$ provided by Corollary \ref{cor:Const}.

The composition of $1$-morphisms is defined by 
\begin{equation}
\fkLag_1(X_1, X_2) \times \fkLag_1(X_0, X_1) \stackrel{\circ}\longrightarrow \fkLag_1(X_0,X_2)
\end{equation}
\[(N_1, N_0) \mapsto N_{1}\circ N_{0}  
\]
where $N_{1}\circ N_{0}$ is the pair consisting of the map $N_{0} \times_{X_1} N_{1} \to X_{02}$ and the Lagrangian structure discussed in Corollary \ref{cor:Comp}. Using the notation from Definition \ref{def:bullet} $N_{1}\circ N_{0}=C_{X_{012}}(N_0\times N_1)$, where $X_{012}$ is the Lagrangian constructed in Theorem \ref{thm:ThreeLag}. Again here we choose a representative for the homotopy fiber product. 

The vertical composition of $2$-morphisms 
\begin{equation}\label{eqn:vert}
\fkLag_2(N_1, N_2) \times \fkLag_2(N_0,N_1) \stackrel{\odot}\longrightarrow \fkLag_2(N_0, N_2)
\end{equation}
is defined as
\[(U_1,U_0) \mapsto U_{1} \odot U_{0} 
\]
where $U_{1} \odot U_{0}$ is the pair consisting of the natural morphism $U_{0}\times_{N_1} U_{1} \to N_{02}$ along with the Lagrangian structure constructed in Corollary \ref{cor:Comp}, therefore  $ U_{1} \odot U_{0}=C_{N_{012}}(U_0\times U_1)$.

Denote $P_0=N_0 \circ M_0$ and $P_1=N_1 \circ M_1$, we define horizontal composition of $2$-morphisms as  
\begin{equation}\label{eqn:horiz}
\fkLag_2(N_0,N_1) \times \fkLag_2(M_0,M_1) \stackrel{*}\longrightarrow \fkLag_2(N_0 \circ M_0,N_1 \circ M_1)
\end{equation}
by 
\[(V,U) \mapsto V * U =C_{P_0 \times_{X_{012}} P_1}(U \times V)
\]
where $P_0 \times_{X_{012}} P_1$ is the Lagrangian constructed in in Proposition \ref{prop:6}. Recall that the underlying morphism of derived stacks of $V*U$ is the induced morphism $U \times_{X_1} V\to P_{0} \times_{X_{02}} P_{1}$.

For each object $X$ define the identity $1$-morphism $I_X$ to be the diagonal $\Delta_X: X \to X \times_S X$ constructed in Corollary \ref{cor:diag}.
The identity $2$-morphisms will be introduced in the proof of Lemma \ref{lem:2-morAssoc}. 
\end{defn}

We will now prove a sequence of lemmas which show that the data 
\[\{\circ, *,  \odot, \fkLag(S,\omega)_1(-,-), \fkLag(S,\omega)_2(-,-)\}\] along with the unitors and associators which we define in Definitions \ref{defn:Assoc} and \ref{defn:unitors} define a bicategory.

\begin{lem}\label{lem:2-morAssoc}
The vertical composition of $2$-morphisms defined in (\ref{eqn:vert}) is associative and has units. 
\end{lem} 
\begin{proof}
Associativity follows immediately from Proposition \ref{prop:Associativity}. Given $M\in \fkLag_{1}(X_0,X_1)$, we define the identity $2$-morphism $1_{M}$ to be the Lagrangian $\Delta: M \to M \times_{X_{01}} M$ constructed in Corollary \ref{cor:diag}. Finally Proposition \ref{prop:IdOriginal} shows that this is indeed an identity for $\odot$.
\end{proof}

\begin{lem}\label{lem:Compat}
Vertical and horizontal composition of $2$-morphism are compatible. This means that given objects $X_0, X_1$ and $X_2$, 1-morphisms $M_0, M_1, M_2 \in \fkLag_1(X_0, X_1)$ and $N_0, N_1, N_2\in$ $ \fkLag_1(X_1,X_2)$, and $2$-morphisms $U_1 \in\fkLag_2(M_0, M_1)$,  $U_2 \in\fkLag_2(M_1, M_2)$, $V_1 \in\fkLag_2(N_0, N_1)$ and $V_2 \in\fkLag_2(N_1, N_2)$ we have:
\[
(V_2 * U_2) \odot (V_1 * U_1) = (V_{2} \odot V_{1}) * (U_2 \odot U_1).
\]
\end{lem}
\begin{proof}
From the definitions we have
\begin{align*}
 (V_2 * U_2) \odot (V_1 * U_1)&= C_{P_{012}}\big(C_{P_0\times_{X_{012}}P_1}(U_1\times V_1)\times C_{P_1\times_{X_{012}}P_2}(U_2\times V_2)\big)\\
 &\cong C_{P_{012} \bullet ( (P_0\times_{X_{012}}P_1) \times (P_1\times_{X_{012}} P_2) )} \big(U_1\times U_2 \times V_1 \times V_2 \big),
\end{align*}
using Proposition \ref{prop:cXncYequiv}(b). 
If we denote by $\rho$ the symplectomorphism
\[\rho:M_{01}\times M_{12}\times N_{01}\times N_{12} \To M_{01}\times N_{01}\times M_{12}\times N_{12},
\]
which interchanges the two middle factors, then we have the Lagrangeomorphism 
\[C_{\Gamma_\rho}\big( U_1\times U_2 \times V_1 \times V_2 \big) \cong U_1\times   V_1\times U_2 \times V_2. 
\]
Therefore, by a similar argument we have 
\begin{align}
(V_{2} \odot V_{1}) * (U_2 \odot U_1) & \cong C_{(P_0\times_{X_{012}}P_2)\bullet(M_{012}\times N_{012})}\big (U_1\times   V_1\times U_2 \times V_2 \big).\\
 & \cong C_{(P_0\times_{X_{012}}P_2)\bullet(M_{012}\times N_{012})\bullet \Gamma_\rho}\big(U_1\times U_2 \times V_1 \times V_2 \big).
\end{align}
Hence, by Proposition \ref{prop:cXncYequiv}(a), the proof will be complete once we establish a Lagrangeomorphism
\[P_{012} \bullet ( (P_0\times_{X_{012}}P_1) \times (P_1\times_{X_{012}} P_2) ) \cong (P_0\times_{X_{012}}P_2)\bullet(M_{012}\times N_{012})\bullet \Gamma_\rho.
\]
In turn, the existence of such Lagrangeomorphism follows from Lemma \ref{lem:LagUniqP02}. For this note that both Lagrangians are homotopy limits of the diagram (\ref{eqn:P02Dag}) as derived stacks. Inspecting the constructions of the Lagrangian structures in both cases one can see that they are homotopic to the one described in Lemma \ref{lem:LagUniqP02}. Hence we can apply the lemma to conclude the proof.
\end{proof}

\begin{defn} \label{defn:Assoc}
Consider a sequence of $1$-morphisms
\[X_0 \stackrel{N_1}\longrightarrow X_{1} \stackrel{N_2}\longrightarrow X_{2} \stackrel{N_3}\longrightarrow X_{3}
\]
We define the associator 
\[W_{N_3 N_2 N_1} \in \fkLag_2((N_{3} \circ N_2)\circ N_1, N_{3} \circ (N_{2} \circ N_1))\] 
as the Lagrangian constructed in Proposition \ref{prop:Associativity}. 
It follows from Proposition \ref{prop:IdOriginal} that this is invertible and hence a 2-isomorphism. 
\end{defn}

\begin{lem}\label{lem:naturalAssoc}
The associator is natural, meaning that given $2$-morphisms $U_i \in \fkLag_2(M_{i}, N_{i})$, for $i=1,2,3$, we have
\[(U_{3}*(U_2 *U_1 )) \odot W_{M_3 M_2 M_1} =  W_{N_3 N_2 N_1} \odot ((U_3 * U_2) * U_1).
\]
\end{lem}
\begin{proof}
We will denote by $L_i=M_i \times_{X_{i-1, i}}N_i$ the $(n-2)$-symplectic derived stack and by $M_{12}=M_1\times_{X_1}M_2$ the Lagrangian in $X_{02}$. Recall from definition and the properties of $C_{-}$ we have
\[U_{3}*(U_2 *U_1 )=C_{(M_{(12)3}\times_{X_{023}}N_{(12)3})\bullet ((M_{12}\times_{X_{012}}N_{12}) \times \Delta_{L_3})}(U_1 \times U_2 \times U_3).
\]
Also, by the definition of $\odot$ and Proposition \ref{prop:Characterization} we have 
\[W_{N_3 N_2 N_1} \odot ((U_3 * U_2) * U_1) \odot  W_{M_3 M_2 M_1}^{-1}= C_{J}(W_{M_3 M_2 M_1}^{-1} \times ((U_3 * U_2) * U_1) \times W_{N_3 N_2 N_1} ),
\]
where $J=M_{(12)3} \times_{X_{03}} M_{1(23)} \times_{X_{03}} N_{1(23)} \times_{X_{03}} N_{(12)3}$.
Now as above, $(U_3 * U_2) * U_1= C_{K}(U_1 \times U_2 \times U_3)$, with $K=(M_{1(23)}\times_{X_{013}}N_{1(23)})\bullet(\Delta_{L_1}\times(M_{23}\times_{X_{123}}N_{23}))$.
Next we observe that the associator $W_{N_3 N_2 N_1}$, which was defined as the graph of a Lagrangeomorphism $\Gamma_{\rho_N}$, can equivalently be described as $C_{\Gamma_{\rho_N}}(\bullet)$ where we consider $\Gamma_{\rho_N}$ as a Lagrangian correspondence from a point to $N_{1(23)}\times_{X_{03}}N_{(12)3}$. Therefore we have
\[W_{N_3 N_2 N_1} \odot ((U_3 * U_2) * U_1) \odot  W_{M_3 M_2 M_1}^{-1}= C_{J\bullet(\Gamma_{\rho_M^{-1}}\times K \times \Gamma_{\rho_N})}(U_1 \times U_2 \times U_3).
\]

Hence the lemma will be a consequence of the following Lagrangeomorphism
\begin{equation}\label{naturality}
(M_{(12)3}\times_{X_{023}}N_{(12)3})\bullet ((M_{12}\times_{X_{012}}N_{12}) \times \Delta_{L_3})
\cong J\bullet(\Gamma_{\rho_M^{-1}}\times K \times \Gamma_{\rho_N}),
\end{equation}
between Lagrangians in $(L_1 \times L_2 \times L_3)^- \times (M_{(12)3}\times_{X_{03}}N_{(12)3})$.
The existence of such Lagrangeomorphism follows from a statement analogous to Proposition \ref{prop:Characterization} and Lemma \ref{lem:LagUniqP01}, that is we can show that there is a unique Lagrangian in $(L_1 \times L_2 \times L_3)^- \times (M_{(12)3}\times_{X_{03}}N_{(12)3})$ satisfying a natural condition that we will not spell out. Then one checks that both Lagrangians in (\ref{naturality}) satisfy this requirement.
\end{proof}

\begin{lem}\label{lem:weakAssoc1mor}
The associator satisfies the pentagon axiom. This states that given a sequence of $1$-morphisms
\[X_0 \stackrel{N_1} \longrightarrow X_1  \stackrel{N_2} \longrightarrow X_2   \stackrel{N_3} \longrightarrow X_3  \stackrel{N_4} \longrightarrow X_4,
\]
we have
\[W_{43(21)} \odot W_{(43)21} = (1_{N_4} * W_{321}) \odot W_{4(32)1} \odot (W_{432}*1_{N_1}) 
\]
where we have simplified the notation so that $ W_{(43)21}$ stands for $W_{(N_4\circ N_3) N_2 N_1}$.
\end{lem}
\begin{proof}
To prove the pentagon axiom, we first notice that the underlying space of $W_{321}$ is $\Gamma_{\rho_{321}}$ where $\rho_{321}$ is the morphism $\rho$ appearing in the proof of Proposition \ref{prop:Associativity}. Notice that the $\odot$-composition of two graphs of morphisms is the graph of the composition of these morphisms. Also, we can see that $W_{432} * 1_{N_1}$ is Lagrangeomorphic to $\Gamma_{id_{N_1} \times_{X_1} \rho_{432}}$, where
\[id_{N_1} \times_{X_1} \rho_{432}: N_1 \times_{X_1} (N_2 \times_{X_2} (N_3 \times_{X_3} N_4))\To N_1 \times_{X_1} ((N_2 \times _{X_2} N_3 )\times_{X_3} N_4).
\]
Similarly $1_{N_4} * W_{321}$ is Lagrangeomorphic to $\Gamma_{\rho_{321}\times_{X_3} id_{N_4}}$. 
Therefore the equation we are trying to show, follows from establishing a Lagrangeomorphism between the graphs of $\rho_{43(21)} \circ \rho_{(43)21} $ and $ (\rho_{321}\times_{X_3}id_{N_4} ) \circ \rho_{4(32)1} \circ (id_{N_1}\times_{X_1} \rho_{432})$. 

First we can choose a homotopy equivalence 
\begin{equation}\label{eqn:rhoIsom}
\rho_{43(21)} \circ \rho_{(43)21} \to  (\rho_{321}\times_{X_3}id_{N_4} ) \circ \rho_{4(32)1} \circ (id_{N_1}\times_{X_1} \rho_{432}).
\end{equation}
This is because (1) both sides are equivalences between $N_1 \times_{X_1} (N_2 \times_{X_2} (N_3 \times_{X_3} N_4))$ and $((N_1 \times_{X_1} N_2) \times_{X_2} N_3) \times_{X_3} N_4$ which homotopy commute with the system given by the projections to the $N_i$ and $X_j$ and (2) both $N_1 \times_{X_1} (N_2 \times_{X_2} (N_3 \times_{X_3} N_4))$ and $((N_1 \times_{X_1} N_2) \times_{X_2} N_3) \times_{X_3} N_4$ are homotopy limits of the same system. 

The equivalence of graphs induced by (\ref{eqn:rhoIsom}) homotopy commutes with the projections of both graphs to $(((N_4 \circ N_3) \circ N_2) \circ N_1) ))  \times_{X_{04}} (N_4 \circ (N_3 \circ (N_2 \circ N_1) ))  $. According to Corollary \ref{cor:HIsomLag} what remains is to show that that there is a homotopy between the two isotropic structures (one of which is pulled back by this equivalence of graphs). We do not include the details of the diagrams needed to establish this homotopy as similar proofs appear throughout this article.
\end{proof}

\begin{defn}\label{defn:unitors}
Fix objects $X_0$ and $X_1$ and consider $M\in \fkLag_{1}(X_0,X_1)$. We define the unitors $l_{M} \in \fkLag_{2}(I_{X_1}\circ M, M)$ and $r_{M} \in \fkLag_{2}(M \circ I_{X_0},M)$ to be the graphs of the Lagrangeomorphisms $\rho_{l}:M \times_{X_1} \Delta_{X_1} \to M$ and $\rho_{r}:\Delta_{X_0} \times_{X_0} M \to M$ constructed in Proposition \ref{prop:IdOriginal}.
\end{defn}

We leave the proof of the following lemma to the reader as it can be proved using the same techniques as the previous lemmas.

\begin{lem}\label{lem:UnitorsSatisfy}
The unitors are natural, which means given objects $X_0$ and $X_1$ and $1$-morphisms $M, N \in \fkLag_1(X_0, X_1)$ we have $U \odot r_{M}= r_{N} \odot (U* 1_{I_{X_0}})$ and $U \odot l_{M}= l_{N} \odot (1_{I_{X_1}}* U)$, for any $U\in \fkLag_{2}(M,N)$. 

Moreover they satisfy the triangle axiom which states that
\[({1}_{P}*l_M ) \circ W_{P,I_{X_1},M} = r_P* 1_{M},
\]
for any $1$-morphism $P\in \fkLag_1(X_1,X_2)$.
\end{lem}

Summarizing the results in this section, we have proven the following 

\begin{thm}\label{thm:Main2Cat} 
Let $(S, \omega)$ be a $n$-symplectic derived stack. Then $\fkLag(S, \omega)$ as defined above is a bicategory.
\end{thm}

In the case of $S=\bullet_{n+1}$ we use the notation $\fkSymp^{n}= \fkLag(\bullet_{n+1},\omega)$. In this case the theorem gives the following

\begin{cor}\label{cor:Symp0}
There exists a bicategory $\fkSymp^{n}$ whose objects are $n$-symplectic derived stacks, the $1$-morphisms are Lagrangian correspondences, and the $2$-morphisms are relative Lagrangian correspondences up to Lagrangeomorphism.
\end{cor}

The bicategory $\fkSymp^{n}$ has an additional structure, namely that of a symmetric monoidal bicategory (see Definition 2.1 \cite{S-P}). 

\begin{thm}
The bicategory $\fkSymp^{n}$ is a symmetric monoidal bicategory. The monoidal structure 
\[
\fkSymp^{n} \times \fkSymp^{n} \to \fkSymp^{n},
\]
at the level of objects, sends $((S_1\omega_1), (S_2, \omega_2))$ to $(S_1 \times S_2, \omega_1 \boxplus \omega_2)$ and has the point $\bullet_{n}$ as the unit.
\end{thm} 
\begin{proof}
We define the monoidal structure on morphisms by product of Lagrangians, as defined in Proposition \ref{product}. Together with some natural isomorphisms which we do not write down, this defines a morphism of bicategories. This morphism of bicategories along with several obvious compatibility natural transformations defines the structure of a symmetric monoidal bicategory in the sense of Definition 2.1 of \cite{S-P}. The details are straightforward but tedious.
\end{proof}

\section{Orientations and Perverse sheaves}\label{Sec:OrPerv}

In this section we will discuss some facts about perverse sheaves that are needed to linearize the bicategory $\fkSymp^{0}$. The starting point is the construction in \cite{BBBJ} and \cite{BBDJS} of a canonical perverse sheaf on \emph{oriented} $(-1)$-symplectic derived stacks. The second ingredient is a conjecture by Joyce that an oriented Lagrangian in a $(-1)$-symplectic derived stack determines a section of the perverse sheaf. Here we give a more refined version of this conjecture and provide a local construction of the section.

But for all of these constructions we need to impose some orientability requirements, so we cannot linearize $\fkSymp^{0}$ directly. 
So we will rather linearize an oriented version of it which we denote by $\fkSymp^{or}$. 

\subsection{Orientations on Lagrangians}\label{Ssec:Or} \

We will define two distinct notions of orientations for Lagrangians. One for Lagrangians in $0$-symplectic derived stacks and another for Lagrangians in $(-1)$-symplectic derived stacks.

First we establish some notation. For a derived Artin stack $Q$, we define its canonical bundle $K_{Q}$ as the line bundle $\det ( \mathbb{L}_{Q})$. If the derived Artin stack $Q$ has a $0$-symplectic structure $\omega_{Q}$ then this can be used to trivialize $K_{Q} = \det(\mathbb{L}_{Q})$ and we always use this trivialization. 

We start with the definition of (relatively) oriented Lagrangian in a $0$-symplectic derived stack. It is inspired in the notion of relatively spin Lagrangian introduced in Lagrangian Floer cohomology. 

\begin{defn}\label{defn:Or0Lag}
Let $(S, \omega)$ be a $0$-symplectic derived stack and let $E$ be a line bundle on $S$. An $E$-oriented Lagrangian in $S$ is a triple consisting of a Lagrangian $f: L\to S$, a line bundle $R_{L}$ on $L$ and an isomorphism 
\[\gamma_{L}: R_{L}^{\otimes 2} \to K_{L}\otimes f^{*}E.
\] 

When $S$ is a point, $L$ is $(-1)$-symplectic and this recovers the notion of orientation on a $(-1)$-symplectic derived stack $X$ introduced in \cite{BBDJS}. Concretely it consists of line bundle $R_X$ and an isomorphism
\[\gamma_{X}: R_{X}^{\otimes 2} \longrightarrow K_{X}
\] 
\end{defn}

\begin{example}
Given a smooth scheme $U$ and a regular function $f \in \mathcal{O}(U)$ consider the derived critical locus 
\[Crit(f):= U \times_{df, T^{*}U,0} U \stackrel{\iota}\to U,
\] 
which is a $(-1)$-symplectic derived scheme. Denoting by $\alpha:Crit(f) \to T^{*}U$ the induced morphism, we have $K_{Crit(f)} \cong \iota^{*} K_{U}^{\otimes 2} \otimes \alpha^{*} K_{T^{*}U}^{-1} \cong \iota^{*} K_{U}^{\otimes 2}$, since $T^*U$ is symplectic. This defines a canonical orientation on $Crit(f)$, with $R_{Crit(f)} = \iota^{*}K_{U}$.
\end{example}

Let $X$ be a $(-1)$-symplectic derived stack and $\phi:M \to X$ be a Lagrangian. Then we have by definition a quasi-isomorphism $\mathbb{T}_{M} \cong \mathbb{L}_{\phi}[-2]$. 
Using the exact triangle $\phi^{*}\mathbb{L}_{X}\to \mathbb{L}_{M}\to \mathbb{L}_{\phi}$ we get 
\[\det(\mathbb{L}_{M})^{-1} \cong \det(\mathbb{T}_{M}) \cong \det(\mathbb{L}_{\phi}) \cong (\det(\mathbb{L}_{M})) \otimes \phi^{*}\det(\mathbb{L}_{X})^{-1}.
\]
Therefore, this determines a canonical isomorphism 
\begin{equation}\label{eqn:alpha} 
\alpha_{M}: (\det \mathbb{L}_{M})^{\otimes 2} \to \phi^{*} (\det \mathbb{L}_{X}).
\end{equation}

\begin{defn}\label{defn:OrLag}
Consider a $(-1)$-symplectic derived stack $X$ with an orientation $R_X$. An oriented Lagrangian in $X$ is a pair consisting of a Lagrangian $\phi:M \to X$ and an isomorphism 
\[\beta_{M}: K_{M} \longrightarrow \phi^{*}R_{X}
\]
such that $\gamma_{X} \circ \beta_{M}^{\otimes 2} = \alpha_{M}$.
\end{defn}

\begin{rem}
Note that the space of orientations on a $(-1)$-Lagrangian $M$ has a $\mathbb{Z}_2$ action, given by multiplying the orientation $\beta_{M}$ by $\pm 1$. If $M$ is a $(-1)$-Lagrangian with some orientation $\beta_{M}$, we will denote by $-M$ the same Lagrangian with orientation $-\beta_{M}$, which we refer to as the reverse orientation.
\end{rem}

We now prove several lemmas establishing some properties of orientations. 

\begin{lem}\label{lem:OrientedProduct}
Let $S_0$ and  $S_1$ be $0$-symplectic derived stacks with line bundles $E_i$ in $S_i$, for $i=0,1$.  Given an $E_{0}$-oriented Lagrangian $X_0 \to S_0$ and an $E_1$-oriented Lagrangian $X_1 \to S_1$, there is an induced $E_{0}\boxtimes E_1$-orientation on the product Lagrangian $X_0 \times X_1 \to S_0 \times S_1$ discussed in Proposition \ref{product}.

Let $Y_0$ and  $Y_1$ be oriented $(-1)$-symplectic derived stacks. Given oriented Lagrangians $M_0 \to Y_0$ and $M_1 \to Y_1$, there is an induced orientation on the product Lagrangian $M_0 \times M_1 \to Y_0 \times Y_1$.
\end{lem}
\begin{proof}
The first statement easily follows from the fact that $K_{X_0 \times X_1}\cong K_{X_0}\boxtimes K_{X_1}$. For the second part note that the map $\alpha$ defined in \ref{eqn:alpha} satisfies
\[\alpha_{M_0\times M_1}=\alpha_{M_0}\boxtimes\alpha_{M_1},
\]
for a product Lagrangian. This implies the result.
\end{proof}

\begin{lem}\label{lem:OrienConv}
Let $S_0, S_1, S_2$ be $0$-symplectic derived stacks with line bundles $E_i$ on $S_i$ for $i=0,1,2$. Given $f: N_1\to S^{-}_0\times S_1$, a $(E^{-1}_{0}\boxtimes E_{1})$-oriented Lagrangian and 
$g: N_2\to S^{-}_1\times S_2$, a  $(E^{-1}_{1}\boxtimes E_{2})$-oriented Lagrangian. Then the Lagrangian  $N_{2} \bullet N_{1}$ has a natural  $(E^{-1}_{0}\boxtimes E_{2})$-orientation.
\end{lem}
\begin{proof} 
Let $(R_{N_i}, \gamma_{N_i})$ be the orientations of the $N_i$ for $i=1,2$.  
Recall the Lagrangian  $N_{2} \bullet N_{1}$ is defined as a Lagrangian structure on the map $h: N_{1} \times_{S_1} N_2\to S_{0}^{-} \times S_{2}$, induced by $f_0$ and $g_2$. We  define an $(E^{-1}_{0}\boxtimes E_{2})$-orientation by taking $R_{N_{1} \times_{S_1} N_2}= R_{N_{1}} \boxtimes R_{N_{2}}$ and $\gamma_{N_{1} \times_{S_1} N_2}$ equal to the composition
\begin{equation}
\begin{split}R^{\otimes 2}_{N_{1} \times_{S_1} N_2} \cong R^{\otimes 2}_{N_{1}} \boxtimes R^{\otimes 2}_{N_{2}}\xrightarrow{\gamma_{N_1} \boxtimes \gamma_{N_2}} & (K_{N_1} \otimes f^{*}(E^{-1}_{0}\boxtimes E_{1})) \boxtimes (K_{N_2} \otimes g^{*}(E^{-1}_{1}\boxtimes E_{2})) \\
& \cong (K_{N_1} \boxtimes K_{N_2})  \otimes h^{*}(E^{-1}_{0}\boxtimes E_{2}) \\ & \cong K_{N_{1} \times_{S_1} N_2} \otimes h^{*}(E^{-1}_{0}\boxtimes E_{2}).
\end{split}
\end{equation}
Here we have use the fact that $K_{S_1}$ and $f_1^*(E_1)\boxtimes g_1^*(E_1^{-1})$ have canonical trivializations.
\end{proof}

\begin{lem}\label{lem:OrienProds}
Let $S$ be a $0$-symplectic derived stack with line bundle $E$ and consider $N_0, N_1, N_2$, $E$-oriented Lagrangians in $S$. 
\begin{itemize}
\item[{\bf(a)}]
The $E$-orientations on $N_0,N_1$ induce an orientation on the $(-1)$-symplectic derived stack $N_{01}= N_0 \times_{S} N_1$. 
\item[\bf(b)] Using the orientations on $N_{01}, N_{12}, N_{20}$  from part (a) and the orientation on their product discussed in Lemma \ref{lem:OrienConv}, there is a natural orientation on the Lagrangian 
\[\varphi:N_{0} \times_{S} N_{1} \times_{S} N_{2} \to N_{01} \times N_{12} \times N_{20}
\]
defined in Theorem \ref{thm:ThreeLag}.
\end{itemize}
\end{lem}
\begin{proof} Denote by  $f_{i}:N_i \to S$ the Lagrangian morphisms and by $(R_{N_i}, \gamma_{N_i})$ the orientations. For part (a) we define $R_{N_{01}}= (R_{N_0}\otimes f_0^{*} E^{-1})\boxtimes R_{N_1}$. The composition defining $\gamma_{N_{01}}$ can be easily constructed using the isomorphism $K_{N_{01}}\cong K_{N_0}\boxtimes K_{N_1}$.

For part (b) we define the orientation on the triple fiber product as the composition:
\begin{align*}
 K_{N_{012}}  & \cong K_{N_0} \boxtimes K_{N_1} \boxtimes K_{N_2} \cong (R^{\otimes 2}_{N_0}\otimes E^{-1})\boxtimes (R^{\otimes 2}_{N_1}\otimes E^{-1})\boxtimes (R^{\otimes 2}_{N_2}\otimes E^{-1}) \\
 & \cong (R_{N_0} \otimes  E^{-1}\otimes R_{N_1})\boxtimes (R_{N_1} \otimes E^{-1} \otimes R_{N_2})\boxtimes (R_{N_2} \otimes E^{-1} \otimes R_{N_0}) \\
 &  \cong \varphi^{*}(R_{N_{01} \times N_{12} \times N_{20}}),
\end{align*}
where we have omitted the pullbacks from the notation. It's easy to check that this map satisfies the required property.
\end{proof}

\begin{lem}\label{lem:ConvOrien}
Given $X_0, X_1$ oriented $(-1)$-symplectic derived stacks and oriented Lagrangians 
$g:N\longrightarrow X_{0}$ and $(f_0,f_1):M \longrightarrow X_{0}^{-} \times X_1$ the Lagrangian $c_{f}(g): N \times_{X_0} M\longrightarrow X_{1}$ from Proposition \ref{prop:Lag2Lag} has an induced orientation 
\end{lem}
\begin{proof} 
We have isomorphisms $\beta_{N}:K_{N} \to g^{*}R_{X_0}$ and $\beta_{M}:K_{M} \to (f_{0}^{*} R_{X_0})\otimes (f_{1}^{*} R_{X_1})$ and $\gamma_{X_0}:R_{X_0}^{\otimes 2} \to K_{X_0}$ . Define $\beta_{N\times_{X_0}M}$ as the composition 
\[K_{N\times_{X_0}M}\cong (K_{N} \otimes g^{*}K_{X_0}^{-1}) \boxtimes K_{M}
\xrightarrow{(\beta_{N} \otimes \text{id})\boxtimes \beta_{M}} g^{*}(R_{X_0} \otimes K_{X_0}^{-1})\boxtimes ((f_{0}^{*} R_{X_0})\otimes (f_{1}^{*} R_{X_1})) \cong c_{f}(g)^{*}R_{X_1}
\]
where in the last isomorphism we use $\gamma_{X_0}$. It is easy to check that $\gamma_{X_1} \circ \beta^{\otimes 2}_{N\times_{X_0}M} = \alpha_{N\times_{X_0}M}.$
\end{proof}

As in the unoriented case, we will use the notation $C_M(N)$ for the oriented Lagrangian constructed in the above lemma. 

\begin{defn}\label{defn:OrientedLagrangeomorphism} 
Let $S$ be a $0$-symplectic derived stack with a line bundle $E$. An oriented Lagrangeomorphism between $E$-oriented Lagrangians $(X_0, R_{X_0}, \gamma_{X_0})$ and $(X_1, R_{X_1}, \gamma_{X_1})$ is a pair consisting of a Lagrangeomorphism $\rho:X_0 \to X_1$ and an isomorphism $\zeta: \rho^{*} R_{X_1} \to R_{X_0}$.

Let $Y$ be an oriented $(-1)$-symplectic derived stack and $(f_0:N_0\To Y, \beta_{N_0})$, $(f_1:N_1 \To Y, \beta_{N_1})$ be oriented $(-1)$-Lagrangians. An oriented Lagrangeomorphism between $N_0$ and $N_1$ consists of a Lagrangeomorphism $\psi: N_0 \To N_1$ such that $\beta_{N_0}$ equals
\[ K_{N_0}\cong \psi^*(K_{N_1})\xrightarrow{\psi^* \beta_{N_1}} \psi^*(f_1^*(R_Y))\cong f_0^*(R_Y).
\]
\end{defn}

\begin{rem}
Notice that the condition of oriented Lagrangeomorphism between $0$-Lagrangians gives the following isomorphism of line bundles 
\[K_{X_0} \cong R_{X_{0}}^{\otimes 2}\otimes f_{0}^{*} (E^{-1}) \cong R_{X_{0}} \otimes \rho^{*} (R_{X_{1}}) \otimes f_{0}^{*} (E^{-1}) \cong \Gamma_{\rho}^{*}(R_{X_{01}}).
\] 
We can easily see that this determines an orientation on the Lagrangian $\Gamma_{\rho}:X_{0} \to X_{01}$, and in fact it is equivalent to it.
\end{rem}

From now on we will use $\vdim M$ to denote the virtual dimension of a derived Artin stack, that is the Euler characteristic of its tangent complex $\mathbb{T}_M$.

The operation $C_M(\cdot)$ satisfies similar properties to the unoriented one, stated in Proposition \ref{prop:cXncYequiv}. We collect the most important ones in the following lemma whose proof is elementary. 

\begin{lem}\label{lem:comporientation}
Let $X_0$, $X_1$ and $X_2$ be oriented $(-1)$-symplectic derived stacks and let $M_0, M_0' \longrightarrow X_0^{-} \times Y_0$, $M_1 \longrightarrow X_1^{-} \times Y_1$ and $N_0 \longrightarrow Y_0^{-} \times Z_0$ be oriented Lagrangian correspondences and consider oriented Lagrangians $U_1, U_1' \longrightarrow X_0$ and $U_1\longrightarrow X_1$. We have the following:
\begin{itemize}
\item[\textbf{(a)}] If $M_0$ is oriented Lagrangeomorphic to $M_0'$ and $U_0$ is oriented Lagrangeomorphic to $U_0'$ then $C_{M_0}(U_0)$ and $C_{M_0'}(U_0')$ are oriented Lagrangeomorphic. 
\item[\textbf{(b)}] We have an oriented Lagrangeomorphism
\[C_{N_0}(C_{M_0}(U_0)) \cong C_{N_0 \bullet M_0}(U_0),\]
where $N_0 \bullet M_0$ is the oriented Lagrangian constructed in Lemma \ref{lem:OrienConv}.
 \item[\textbf{(c)}] We have an oriented Lagrangeomorphism
\[C_{M_0\times M_1}(U_0\times U_1)\cong (-1)^{m_0 u_1} C_{M_0}(U_0) \times C_{M_1}(U_1),
\]
where $m_0=\vdim M_0$ and $u_1=\vdim U_1$.
\end{itemize}
\end{lem}

Now we have the tools to carry out all the constructions of Section 4 in the oriented setting. This gives the following

\begin{thm}\label{thm:sympor}
There exists a symmetric monoidal bicategory $\fkSymp^{or}$  enriched over $\textbf{gr-Inv}$. The objects are pairs consisting of a $0$-symplectic derived stack $S$ and a line bundle $E$ on $S$.  The $1$-morphisms in $\fkSymp^{or}_1((S_0, E_0),(S_1,E_1))$ consist of $(E^{-1}_{0}\boxtimes E_{1})$-oriented Lagrangians in $S_{0}^{-} \times S_{1}$ and the $2$-morphisms $\fkSymp^{or}_2(X_0, X_1)$ are oriented Lagrangeomorphism classes of oriented Lagrangians in $X_0 \times_{S} X_1$. 

Moreover there is a symmetric monoidal homomorphism $\fkSymp^{or} \to \fkSymp^{0}$ which forgets the orientation data. 
\end{thm}
\begin{proof}
We first discuss the enrichment over $\textbf{gr-Inv}$. Given $M \in\fkSymp^{or}_2(X_0, X_1)$ we define its degree as $|M|= n-\vdim M$, where $n=\vdim X_0$. The involution sends $M$ to the same Lagrangian with the reverse orientation. 

The composition of $1$-morphism is defined as Corollary \ref{cor:Symp0}, using Lemma \ref{lem:OrienConv} to define the orientations. Using Lemmas \ref{lem:OrienProds} and \ref{lem:ConvOrien} we define the compositions of $2$-morphisms as follows. Consider $X_0, X_1, X_2 \in \fkSymp^{or}_1(S_0, S_1)$ and $Y_0, Y_1 \in \fkSymp^{or}_1(S_1, S_2)$. Given $M \in \fkSymp^{or}_2(X_0, X_1)$, $N \in \fkSymp^{or}_2(X_1, X_2)$ and $P \in \fkSymp^{or}_2(Y_0, Y_1)$ we take
\[N \odot M = (-1)^{\vdim N (n_0+n_1)}C_{X_{012}}(M\times N) \ \ \textrm{and} \ \ 
P * M = (-1)^{(n_0+n_1)(n_1+n_2)}C_{Z_0 \times_{S_0\times S_1 \times S_2} Z_1}(M \times P),
\]
where $n_i = (\vdim S_i)/2$. Recall that the Lagrangian $Z_0 \times_{S_0\times S_1 \times S_2} Z_1$ can be described as a triple intersection of oriented $0$-Lagrangians and hence Lemma \ref{lem:OrienProds} (b) assigns it an orientation. 

Using Lemma \ref{lem:comporientation} we can easily show that we have
\[
M_2 \odot (M_1 \odot M_0) = (-1)^{\vdim M_1 (n_0 + n_1)} C_{X_{023}\bullet (X_{012} \times \Delta_{X_{23}})}(M_0 \times M_1 \times M_2)
\]
and
\[(M_2 \odot M_1)\odot M_0 =(-1)^{\vdim M_1 + n_0 +n_1} C_{X_{013}\bullet (\Delta_{X_{01}} \times X_{123})}(M_0 \times M_1 \times M_2).
\]
Now following our conventions for orientations we can easily check that there is an oriented Lagrangeomorphism
$X_{023}\bullet (X_{012} \times \Delta_{X_{23}}) \cong (-1)^{n_0 +n_1}X_{013}\bullet (\Delta_{X_{01}} \times X_{123})$. Putting these facts together we conclude that $\odot$ is associative.
We proceed similarly and compute
\begin{align}\label{comporient1}
(N_2 *&M_2)\odot(N_1*M_1)= \\
&=(-1)^{(n_0+n_2)(\vdim M_2 +\vdim N_2)}C_{Z_{012}\bullet\big((Z_0 \times_{S_{012}} Z_1)\times(Z_1 \times_{S_{012}} Z_2)\big)}(M_1 \times N_1 \times M_2 \times N_2).\nonumber 
\end{align}
Next using the symplectomorphism $\rho$ as in the proof of Lemma \ref{lem:Compat} and using our conventions for orientations we have the oriented Lagrangeomorphism
\[C_{\Gamma_\rho}(M_1 \times M_2 \times N_1 \times N_2) \cong (-1)^{\vdim N_1 \vdim M_2} M_1 \times N_1 \times M_2 \times N_2.
\]
Therefore we have
\begin{equation}\label{comporient2}
(N_2 \odot N_1)*(M_2 \odot M_1)=(-1)^{\epsilon}C_{(Z_0 \times_{S_{012}} Z_2)\bullet (X_{012}\times Y_{012})\bullet \Gamma_\rho}(M_1 \times N_1 \times M_2 \times N_2),
\end{equation}
where $\epsilon=(n_0+n_1)(n_1 +n_2 +\vdim N_1 +\vdim N_2)+(n_0+n_2)(\vdim M_2 +\vdim N_2)+ \vdim N_1 \vdim M_2$. Tracing back through our conventions for orientations we can check that the Lagrangeomorphism
\[Z_{012}\bullet\big((Z_0 \times_{S_{012}} Z_1)\times(Z_1 \times_{S_{012}} Z_2)\big) \cong
(Z_0 \times_{S_{012}} Z_2)\bullet (X_{012}\times Y_{012})\bullet \Gamma_\rho,
\]
constructed in Lemma \ref{lem:Compat} is in fact an oriented Lagrangeomorphism. Therefore we conclude that (\ref{comporient1}) and (\ref{comporient2}) differ by 
\[(-1)^{(n_0+n_1)(n_1+n_2)+\vdim M_2(n_1 +n_2)+\vdim N_1(n_1 +n_2)+ \vdim N_1 \vdim M_2}=(-1)^{|M_2||N_1|},
\]
which finishes the proof of the compatibility of vertical and horizontal composition of $2$-morphisms.

The rest of the proof of the theorem doesn't differ from Corollary \ref{cor:Symp0}. Recall that the identity $1$-morphisms in $\fkSymp^{0}$ are given by $\Delta:S \to S^{-} \times S$.  This has the canonical orientation $R_{S}= \mathcal{O}_{S}$ since there is a canonical isomorphism $K_{S} \otimes_{\mathcal{O}_{S}} \Delta^{*} (E^{-1} \boxtimes E) \cong \mathcal{O}_{S} \otimes_{\mathcal{O}_{S}} \mathcal{O}_{S} \cong  \mathcal{O}_{S}$. The identity $2$-morphisms, associators and unitors in $\fkSymp^{0}$ were all described as the graphs of certain Lagrangeomorphisms. To assign them orientations in the sense of Definition \ref{defn:OrientedLagrangeomorphism} is simply a matter of choosing the obvious $\zeta$-morphisms needed in that definition.

The monoidal structure can be constructed as in Corollary \ref{cor:Symp0} using Lemma \ref{lem:OrientedProduct} to define the necessary orientations. Finally the existence of the forgetful homomorphism is obvious.
\end{proof}

\subsection{Constructible sheaves}\label{subsec:sheaves}\ 

In the remainder of this chapter and in the next one, we take the ground field $k$ to be $\mathbb{C}$ for simplicity. We will work in the context of algebraically constructible sheaves of $\mathbb{F}$-vector spaces on higher algebraic Artin stacks over $\mathbb{C}$. This theory itself has two approaches. The first is the theory of constructible sheaves on the lisse-\'{e}tale topos of an algebraic Artin stack for which $\mathbb{F}$ can be taken to be $\mathbb{Z}/l\mathbb{Z}$, for some prime $l$ or closely related categories using $l$-adic coefficients of various types with some slight technical difficulties (for example restriction on the values of $l$).  The second (see for instance \cite{Sun}) is the theory of constructible sheaves on the Lisse-analytic topos of the analytification of an Artin stack over $\mathbb{C}$. In this case, we can take $\mathbb{F}=\mathbb{Z}/l\mathbb{Z}$ again but now also $\mathbb{Q}$ or $\mathbb{C}$ or any Noetherian ring is fine. Rigid (or Berkovich or Huber) geometry allows us to consider analytifications of stacks defined over fields other than $\mathbb{C}$ if the field is equipped with a valuation and there are also theories of algebraically constructible sheaves on those analytifications but we do not pursue this here. For a derived algebraic Artin stack $X$ over $\mathbb{C}$ we use the notation $D_{c}(X)$, $D^{+}_{c}(X)$, $D^{-}_{c}(X)$, $D^{b}_{c}(X)$ to denote the (triangulated) categories of algebraically constructible sheaves of $\mathbb{F}$-modules on the underlying Artin stack of $X$, and its bounded below, above, and bounded versions. Categories of constructible sheaves on stacks in the (algebraic) \'{e}tale topology are defined in the work \cite{LZ} of Y. Liu and W. Zheng, following Lurie and the work of Laszlo and Olsson \cite{LO1}, \cite{LO2}.   Categories of algebraically constructible sheaves on analytic stacks in the classical analytic topology are discussed in \cite{Sun} (see also  \cite{Pau}).  In the case of rings $\mathbb{F}$ where both theories make sense, such as $\mathbb{F}=\mathbb{Z}/l\mathbb{Z}$ there is no ambiguity in this notation as shown in the comparison theorem proven in \cite{Sun}. 

We consider only morphisms between derived Artin stacks which are locally of finite type, quasi-compact and quasi-separated. For all types of pullback and pushforward functors for morphisms of derived Artin stacks, we work with the associated morphisms on the reduced Artin stacks.  For a morphism $f:X \to Y$ of derived Artin stacks we have a functor $f_{*}:D^{+}_{c}(X) \to D^{+}_{c}(Y)$ with left adjoint the restriction of a functor $f^{*}:D_{c}(Y) \to D_{c}(X)$ to $D^{+}_{c}(X)$.  There is also a duality functor $D_{X}: D_{c}(X) \to D_{c}(X)$. We also sometimes use the pushforward with proper support, $f_{!}=D_{Y}\circ f_{*} \circ D_{X} :D^{-}_{c}(X) \to D^{-}_{c}(Y)$, which is used in the definition of the perverse sheaf of vanishing cycles.  We denote by $f^{!}:D^{-}_{c}(Y) \to D^{-}_{c}(X) $ the right adjoint of $f_{!}$ which is actually the restriction of a functor $D_{c}(Y) \to D_{c}(X) $. See Lemma 6.3.3 and Proposition 6.3.4 of \cite{LZ} for the existence of these functors and their adjointness properties. 

We define the graded vector space
\[\Gamma^{\bullet}(\mS, \mT)= \bigoplus^{\infty}_{i=-\infty}\Hom_{D^{b}_{c}(X)}(\mS, \mT[i]),\] 
graded by $i$.
%$\Hom^{i}(\mS, \mT)= \Hom_{D^{b}_{c}(X)}(\mS, \mT[i])$. 
We will denote by $\mathbb{H}^{\bullet}(X,\mS)$ the graded vector space      $\Gamma^{\bullet}(\mathbb{F}_{X}, \mS)$ and call it the hypercohomology of $\mS$. 

Consider derived Artin stacks $X,Y$ together with morphisms $\pi_{X}$ from $X$ to a point and $\pi_{Y}$ from $Y$ to a point. Then for any $S \in D_{c}(X)$ and $T\in D_{c}(Y)$ we know by the K\"unneth Formula (Proposition 6.1.3 of \cite{LZ}) and by the compatibility of the duality functor with the derived box-product (see Proposition 5.6.4 of \cite{LO1}) that \[(\pi_{X} \times \pi_{Y})_{!}(D_{X\times Y} (S \boxtimes T))  \cong (\pi_{X} \times \pi_{Y})_{!}(D_{X} S \boxtimes D_{Y}T)) \cong (\pi_{X!}D_{X}S) \otimes (\pi_{Y!}D_{Y}T)
\] 
and therefore, 
\[(\pi_{X} \times \pi_{Y})_{*} (S \boxtimes T) \cong  (\pi_{X*}S) \otimes (\pi_{Y*}T)
\]
and so we have 
\begin{equation}\label{eqn:Kunneth}\mathbb{H}^{\bullet}(X \times Y, S \boxtimes T) \cong \mathbb{H}^{\bullet}(X , S)\otimes \mathbb{H}^{\bullet}(Y , T).
\end{equation}

\begin{lem} \label{lem:FunFacts}
If $f:X \to Y$ is a proper morphism of derived Artin stacks then $f_{*}= f_{!}$ as a functor $D^{b}_{c}(X) \to D^{b}_{c}(Y)$.  If $f$ is a closed embedding of derived schemes then $f^{*}f_{*}= \text{id}=f^{!}f_{!}$. Given any Cartesian diagram, 
\begin{equation}\xymatrix{ X\times_{S} Y \ar[r]^{\pi_{Y}} \ar[d]_{\pi_{X}} & Y \ar[d]^{g}   \\
X  \ar[r]_{f}&S }
\end{equation}
of derived Artin stacks there are base change natural isomorphisms 
\[f^{*}g_{!} \cong (\pi_{X})_{!} \pi_{Y}^{*} \  \ \text{and} \ \ g^{*}f_{!} \cong (\pi_{Y})_{!}\pi_{X}^{*}.\] 
Additionally there is a natural transformation
\[c_{f,g}: \pi_{X}^{*}f^{!} \Longrightarrow \pi_{Y}^{!}g^{*}
\]
\end{lem}
\begin{proof}
The first statement is trivial. The second can be found on page 12 of \cite{Mas} in the complex analytic context. The third statement can be found in Proposition 3.2 of \cite{Nic} in the complex analytic context or in Proposition 6.1.1 of \cite{LZ} in the algebraic \'{e}tale context. For the last statement see Proposition 3.1.9 of \cite{KS} in the complex analytic context or in either context simply notice that 
\[\Hom(\pi_{X}^{*}f^{!}\mathcal{S},\pi_{Y}^{!}g^{*}\mathcal{S}) \cong \Hom(\pi_{Y!}\pi_{X}^{*}f^{!}\mathcal{S},g^{*}\mathcal{S}) \cong \Hom(g^{*}f_{!}f^{!}\mathcal{S},g^{*}\mathcal{S}) 
\]
and the right hand side has a canonical element corresponding to the pullback by $g$ of the canonical morphism $f_{!}f^{!}\mathcal{S}\to \mathcal{S}$.
\end{proof}

We will now review the construction, and some properties, of the perverse sheaf of vanishing cycles of a regular function.
Recall that given a regular function $f$ on a variety $U$ over $\mathbb{C}$, we can define a sheaf of nearby cycles of $\mathcal{F} \in D^{b}_{c}(U)$ in $D^{b}_{c}(U^{0})$ where $U^{0}= f^{-1}(0)$. It is defined using the commutative diagram 
\begin{equation}\label{eqn:VanCyc}
\xymatrix{ U_{0} \ar[r]^{\iota} \ar[d] & U \ar[d]_{f}  & T(U_{0})\setminus U_{0}\ar[d] \ar[l]^(.6){j} &E \ar[d] \ar[l]^(.3){\pi} \\
\{0\}  \ar[r]& \mathbb{C}& D_{\epsilon}\setminus \{0\} \ar[l] & \widetilde{D}_{\epsilon}\ar[l] }
\end{equation}
where $D_{\epsilon}$ is a disk of radius $\epsilon$ at $0$ in $\mathbb{C}$, $\widetilde{D_{\epsilon}}$ the universal cover of $D_{\epsilon}\setminus \{0\} $, $T(U_{0})= f^{-1}(D_{ \epsilon})$ and each square is Cartesian. Then the sheaf of nearby cycles of $\mathcal{F} \in D^{b}_{c}(U)$ is defined as 
\[\psi_{f} \mathcal{F} = \iota^{*}(j \circ \pi)_{*}(j \circ \pi)^{*}\mathcal{F}.
\]
The sheaf of vanishing cycles of $\mathcal{F} \in D^{b}_{c}(U)$ is defined as the object $\phi_{f}\mathcal{F}$ making 
\begin{equation}\label{eqn:NearVan} \iota^{*} \mathcal{F} \to \psi_{f} \mathcal{F} \to \phi_{f}\mathcal{F} \to
\end{equation}
into an exact triangle where we have used the natural morphism $\mathcal{F} \to (j \circ \pi)_{*}(j \circ \pi)^{*}\mathcal{F}$. 

We will be interested on the sheaf of vanishing cycles not only of $f$ but also all translations $f-c$ where $c$ is a critical value of $f$, so we introduce the following perverse sheaf
\begin{equation}
\PV_{U,f}=\bigoplus_{c \ \in \ f(Crit(f))} \phi_{f-c}(\mathbb{F}_U [\dim U])|_{Crit(f)}
\end{equation}
on the critical locus of $f$ which we will simply call the sheaf of vanishing cycles of $f$.

Suppose now that we are given a morphism $\Phi:V \to U.$ We can consider the morphism
\begin{equation}\label{eqn:compareVanCyc}
\xymatrix{ V_{0} \ar[d]_{\Phi_{0}} \ar[r]^{\iota_{V}} & V \ar[d]_{\Phi}  & T(V_{0})\setminus V_{0}\ar[d] \ar[l]_(.6){j_{V}} &E_{V}  \ar[d]_{\tilde{\varphi}} \ar[l]_(.3){\pi_{V}}  
\\
U_{0} \ar[r]^{\iota_{U}} & U   & T(U_{0})\setminus U_{0}\ar[l]_(.6){j_{U}} &E_{U}  \ar[l]_(.3){\pi_{U}} }
\end{equation}
of diagrams living over the bottom row of (\ref{eqn:VanCyc}).  Each square in this diagram is Cartesian. Then we have natural equivalences $\tilde{\Phi}_{!}(j_{V}\circ \pi_{V})^{*}=(j_{U}\circ \pi_{U})^{*} \Phi_{!}$ and $\Phi_{0!}i_{V}^{*}= i_{U}^{*}\Phi_{!}$ using base change. Also there is natural morphism $\Phi_{!}(j_{V} \circ \pi_{V})_* \to (j_{U} \circ \pi_{U})_{*} \tilde{\Phi}_{!}$ (which is a natural equivalence if $\Phi$ is proper). This morphism is constructed in (2.5.7) of Proposition 2.5.11 of \cite{KS}. Let $g = f \circ \Phi$. Putting these all together we get for any $\mathcal{F} \in D^{b}_{c}(V)$ morphisms 

\begin{equation*}
\varphi_{0!}\psi_{g} \mathcal{F}= \Phi_{0!}\iota_{V}^{*}(j_{V} \circ \pi_{V})_{*}(j_{V} \circ \pi_{V})^{*}\mathcal{F}= i_{U}^{*}\Phi_{!}(j_{V} \circ \pi_{V})_{*}(j_{V} \circ \pi_{V})^{*}\mathcal{F}\to  i_{U}^{*}(j_{U} \circ \pi_{U})_{*} \tilde{\Phi}_{!}(j_{V} \circ \pi_{V})^{*}\mathcal{F}
\end{equation*}
and 
\begin{equation} i_{U}^{*}(j_{U} \circ \pi_{U})_{*} \tilde{\Phi}_{!}(j_{V} \circ \pi_{V})^{*}\mathcal{F}= i_{U}^{*}(j_{U} \circ \pi_{U})_{*} (j_{U}\circ \pi_{U})^{*} \Phi_{!}\mathcal{F}= \psi_{f} \Phi_{!} \mathcal{F}.
\end{equation}
So for any $\mathcal{F} \in D^{b}_{c}(V)$ we have a canonical morphism
\[\Phi_{0!}(\psi_{f \circ \Phi} \mathcal{F})\to \psi_{f} (\Phi_{!} \mathcal{F}) \]
and hence by (\ref{eqn:NearVan}) we also have a canonical morphism 
\begin{equation}\label{eqn:nonProppush}
\Phi_{0!}(\phi_{f \circ \Phi} \mathcal{F})\to \phi_{f} (\Phi_{!} \mathcal{F}).
\end{equation}
If $\Phi$ is proper these are isomorphisms.

In the next lemma, we construct a ``pull-push" map between the hypercohomology of sheaves and prove some useful properties.
 
\begin{lem}\label{lem:XWdiag}
Consider a diagram 
\[\xymatrix@C=60pt{ X_0 & W \ar[l]_(0.4){ \Phi_0} \ar[r]^{\Phi_1} & X_1 }\] 
of Artin stacks and suppose that $\Phi_1$ is proper and we have objects $\mS_0 \in D^{b}_{c}(X_0)$, and $\mS_1 \in D^{b}_{c}(X_1)$. The following holds:
\begin{itemize}
\item[{\bf (a)}] A morphism $\mu \in \Hom_{D^{b}_{c}(W)}(\Phi_0^{*}\mS_0, \Phi_{1}^{!}\mS_1)$ induces a map 
\[\mu_{*}:\mathbb{H}^{\bullet}(X_0,\mS_0) \to \mathbb{H}^{\bullet}(X_1,\mS_1).
\]
\item[{\bf (b)}] Given another diagram \[\xymatrix@C=60pt{ X_0 & U \ar[l]_(0.4){ \Tau_0}
\ar[r]^{\Tau_1} & X_1 }\] and an equivalence $\Theta:W \to U$ such that $\Tau_i \circ \Theta$ is equivalent to $\Phi_i$, along with morphisms $\mu:\Phi_{0}^{*}\mS_{0} \to \Phi_{1}^{!}\mS_{1}$ and $\eta:\Tau_{0}^{*}\mS_{0} \to \Tau_{1}^{!}\mS_{1}$ such that $\Theta^{*}\eta= \mu$ then $\mu_{*} = \eta_{*}$.
\item[{\bf (c)}] The maps $\mu_*$ compose correctly: given a diagram 
\begin{equation}\label{eqn:CompositionPullPush}
\xymatrix{& & W \times_{X_1} V \ar[dl]_{\pi_W} \ar[dr]^{\pi_V}& & \\
& W \ar[dl]_{\Phi_0} \ar[dr]^{\Phi_1}& & V\ar[dl]_{\Psi_1} \ar[dr]^{\Psi_2}& \\
X_0 && X_1 && X_2 }
\end{equation}
with $\Phi_1$ and $\Psi_2$ proper and morphisms $\mu:\Phi_{0}^{*}\mS_0\to \Phi_{1}^{!}\mS_1$ and $\eta:\Psi_{1}^{*}\mS_1\to \Psi_{2}^{!}\mS_2$ then 
\begin{equation}\label{eqn:SheafComp}(\pi_{V}^{!}(\eta)\circ c \circ \pi_{W}^{*}(\mu))_{*}= \eta_{*} \circ \mu_{*}
\end{equation}
where $c$ comes from the natural transformation $\pi_{W}^{*}\phi_{1}^{!} \Longrightarrow \pi_{V}^{!}\Psi^{*}_{1}$ discussed in Lemma \ref{lem:FunFacts}.
\item[{\bf (d)}] When $X_0=X=X_1$ and $W$ is the diagonal $X \to X \times X$ and $\mS_0=\mS=\mS_1$ and $\Phi_0 = \id_{X} = \Phi_1$ we have $\Phi_{0}^{*}\mS= \mS= \Phi_{1}^{!}\mS$, $\mu=\textnormal{id}_{\mS}$ and using these identifications, $\mu_{*}= \textnormal{id}_{\mathbb{H}^{\bullet}(X,\mS)}$. 
\item[{\bf (e)}] 
Given another diagram \[\xymatrix@C=60pt{ Y_0 & V \ar[l]_(0.4){ \Psi_0}
\ar[r]^{\Psi_1} & Y_1 }\] morphisms $\mu:\Phi_{0}^{*}\mS_0\to \Phi_{1}^{!}\mS_1$ and $\eta:\Psi_{0}^{*}\mT_0\to \Psi_{1}^{!}\mT_1$, consider 
\[\mu\boxtimes \eta:(\Phi_0 \times \Psi_0)^{*}(\mS_0 \boxtimes \mT_0)\to (\Phi_1 \times \Psi_1)^{!}(\mS_1 \boxtimes \mT_1).
\]
Then we have $(\mu\boxtimes \eta)_* = \mu_{*} \otimes \eta_{*}$ via the natural isomorphism $\mathbb{H}^{\bullet}(X_i \times Y_i, \mS_{i} \boxtimes \mT_{i}) \cong \mathbb{H}^{\bullet}(X_i, \mS_{i}) \otimes \mathbb{H}^{\bullet}(Y_i, \mT_{i})$ for $i=0,1$ from equation (\ref{eqn:Kunneth}).
\end{itemize}
\end{lem}
\begin{proof} We define the morphism $\mu_{*}:\Gamma^{\bullet}(\mathbb{F}_{X_0}, \mS_0) \to  \Gamma^{\bullet}(\mathbb{F}_{X_1},\mS_1)$ to be the following composition:
\begin{equation}\label{eqn:firstTransform}
\Gamma^{\bullet}(\mathbb{F}_{X_0}, \mS_0) \to \Gamma^{\bullet}(\mathbb{F}_{W}, \Phi_{0}^{*}\mS_0)\to \Gamma^{\bullet}(\mathbb{F}_{W}, \Phi_{1}^{!}\mS_1) \cong \Gamma^{\bullet}(\Phi_{1*}\mathbb{F}_{W}, \mS_1) \to \Gamma^{\bullet}(\mathbb{F}_{X_1},\mS_1)
\end{equation}
Here the first map is simply $\Phi_{0}^{*}$ and the second is post-composition with $\mu$. The isomorphism is adjunction plus the fact that $(\Phi_1)_{*}=(\Phi_1)_{!}$, since $\Phi_1$ is proper, and the last is pre-composition with the canonical morphism $\mathbb{F}_{X_1} \to \Phi_{1*}\mathbb{F}_{W}$.

In order to prove (b) consider the commutative diagram 
\[\xymatrix{\Gamma^{\bullet}(\mathbb{F}_{X_0}, \mS_{0}) \ar[r]^{\Phi_{0}^{*}}  \ar[dr]_{\Psi_{0}^{*}}& \Gamma^{\bullet}(\mathbb{F}_{W}, \Phi_{0}^{*}\mS_0)\ar[r] & \Gamma^{\bullet}(\mathbb{F}_{W}, \Phi_{1}^{!}\mS_1) \ar[r]^{\cong}&  \Gamma^{\bullet}(\Phi_{1*}\mathbb{F}_{W}, \mS_1) \ar[r] \ar[d]& \Gamma^{\bullet}(\mathbb{F}_{X_1},\mS_1) \\   & \Gamma^{\bullet}(\mathbb{F}_{V}, \Psi_{0}^{*}\mS_0)\ar[u]^{\Theta^{*}}\ar[r] & \Gamma^{\bullet}(\mathbb{F}_{V}, \Psi_{1}^{!}\mS_1) \ar[u]^{\Theta^{*}}\ar[r]&  \Gamma^{\bullet}(\Psi_{1*}\mathbb{F}_{V}, \mS_1) \ar[ur] & 
}
\]
where the map from $\Gamma^{\bullet}(\Phi_{1*}\mathbb{F}_{W}, \mS_1) $ to $\Gamma^{\bullet}(\Psi_{1*}\mathbb{F}_{V}, \mS_1) $ is given by pre-composing by the pushforward by $\Psi_{1}$ by the canonical map $\mathbb{F}_{V} \to \Theta_{*}\mathbb{F}_{W}$. Since each square and triangle commutes the two paths from $\Gamma^{\bullet}(\mathbb{F}_{X_0}, \mS_{0})$ to $\Gamma^{\bullet}(\mathbb{F}_{X_1}, \mS_{1})$ agree. 

We now prove (c).
Let $U= W \times_{X_1}V$. A straightforward but tedious check shows that every sub-diagram of the following three diagrams commutes. Each arrow is some combination of a canonical adjunction, base change, pullback, and the morphisms $c$, $\mu$ and $\eta$.
\begin{equation}\xymatrix{\Gamma^{\bullet}(\mathbb{F}_{X_0}, \mS_0) \ar[r]\ar[dddr] & \Gamma^{\bullet}(\mathbb{F}_{W}, \Phi_{0}^{*}\mS_1)\ar[r] \ar[ddd] & \Gamma^{\bullet}(\mathbb{F}_{W}, \Phi_{1}^{!}\mS_1) \ar[ddd] \ar[r] & \Gamma^{\bullet}(\Phi_{1*}\mathbb{F}_{W},\mS_1)\ar[d]  \\ & & & \Gamma^{\bullet}(\Psi_{1}^{*}\Phi_{1*}\mathbb{F}_{W}, \Psi_{1}^{*}\mS_{1}) \ar[d] \\ & & & \Gamma^{\bullet}(\pi_{V*}\mathbb{F}_{U}, \Psi_{1}^{*}\mS_{1})  \\ 
& \Gamma^{\bullet}(\mathbb{F}_{U}, \pi_{W}^{*}\Phi^{*}_{0}\mS_{1}) \ar[r] & \Gamma^{\bullet}(\mathbb{F}_{U}, \pi_{W}^{*}\Phi^{!}_{1}\mS_{1}) \ar[r] & \Gamma^{\bullet}(\mathbb{F}_{U}, \pi_{V}^{!}\Psi^{*}_{1}\mS_{1})\ar[u]   }
\end{equation}

\begin{equation}
\xymatrix{\Gamma^{\bullet}(\Phi_{1*}\mathbb{F}_{W},\mS_1)\ar[d]  \ar[r] & \Gamma^{\bullet}(\mathbb{F}_{X_1}, \mS_{1}) \ar[r]  & \Gamma^{\bullet}(\mathbb{F}_{V}, \Psi_1^{*}\mS_1) \ar[r]  & \Gamma^{\bullet}(\mathbb{F}_{V}, \Psi_{2}^{!}\mS_2)  \\
\Gamma^{\bullet}(\Psi_{1}^{*}\Phi_{1*}\mathbb{F}_{W}, \Psi_{1}^{*}\mS_{1}) \ar[rru]\ar[d] \\
\Gamma^{\bullet}(\pi_{V*}\mathbb{F}_{U}, \Psi_{1}^{*}\mS_{1})  \ar[rrr] & & & \Gamma^{\bullet}(\pi_{V*}\mathbb{F}_{U}, \Psi_{2}^{!}\mS_{2}) \ar[uu] \\
\Gamma^{\bullet}(\mathbb{F}_{U}, \pi_{V}^{!}\Psi^{*}_{1}\mS_{1}) \ar[u]\ar[rrr] & & & \Gamma^{\bullet}(\mathbb{F}_{U}, \pi_{V}^{!}\Psi_{2}^{!}\mS_{2})\ar[u]
}
\end{equation}
\begin{equation}
\xymatrix{\Gamma^{\bullet}(\mathbb{F}_{V}, \Psi_{2}^{!}\mS_2) \ar[r] & \Gamma^{\bullet}(\Psi_{2*}\mathbb{F}_{V}, \mS_2) \ar[r] & \Gamma^{\bullet}(\mathbb{F}_{X_2}, \mS_2) \\
\\
\Gamma^{\bullet}(\pi_{V*}\mathbb{F}_{U}, \Psi_{2}^{!}\mS_{2}) \ar[uu] && \\
\Gamma^{\bullet}(\mathbb{F}_{U}, \pi_{V}^{!}\Psi_{2}^{!}\mS_{2}) \ar[u] \ar[r] & \Gamma^{\bullet}((\Psi_{2} \circ \pi_{V})_{*}\mathbb{F}_{U}, \mS_{2}) \ar[uuu] \ar[uuur]
}
\end{equation}
Putting together the above three diagrams we have the proof of (c). 
Item (d) is obvious since in this situation, all the maps in equation (\ref{eqn:firstTransform}) are the identity. 

In order to prove (e), notice that $(\mu \boxtimes \eta)_{*}$ can be decomposed into tensor product morphisms 
\begin{equation*}\begin{split}\Gamma^{\bullet}(\mathbb{F}_{X_0}, \mathcal{S}_{0}) \otimes \Gamma^{\bullet}(\mathbb{F}_{Y_0}, \mathcal{T}_{0})\to \Gamma^{\bullet}(\mathbb{F}_{W}, \Phi_{0}^{*}\mathcal{S}_{0})\otimes \Gamma^{\bullet}(\mathbb{F}_{V}, \Psi_{0}^{*}\mathcal{T}_{0}) \to \Gamma^{\bullet}(\mathbb{F}_{W}, \Phi_{1}^{!}\mathcal{S}_{0})\otimes \Gamma^{\bullet}(\mathbb{F}_{V}, \Psi_{1}^{!}\mathcal{T}_{0}) 
\end{split}
\end{equation*} 
followed by 
\begin{equation*}\begin{split}\Gamma^{\bullet}(\mathbb{F}_{W}, \Phi_{1}^{!}\mathcal{S}_{0})\otimes \Gamma^{\bullet}(\mathbb{F}_{V}, \Psi_{1}^{!}\mathcal{T}_{0}) \to \Gamma^{\bullet}(\Phi_{1*}\mathbb{F}_{W}, \mathcal{S}_{0})\otimes \Gamma^{\bullet}(\Psi_{1*}\mathbb{F}_{V}, \mathcal{T}_{0}) \to \Gamma^{\bullet}(\mathbb{F}_{X_1}, \mathcal{S}_{1}) \otimes \Gamma^{\bullet}(\mathbb{F}_{Y_1}, \mathcal{T}_{1}).
\end{split}
\end{equation*} 
\end{proof}

\subsection{Joyce's conjecture} \

The starting point for the linearization comes from the fact that on an oriented $(-1)$-symplectic derived stack $(X, \omega)$ there is (\cite{BBDJS}, \cite{BBBJ}) a perverse sheaf $\mathcal{P}_{(X,\omega)}$ on its underlying reduced Artin stack which is locally modelled on the perverse sheaf of vanishing cycles of a certain algebraic function appearing in the local Darboux model \cite{BBJ}. As usual, we continue to assume that $X$ is defined over $\mathbb{C}$ for simplicity, but as emphasized in \cite{BBBJ}, this perverse sheaf can be constructed in other contexts including the algebraic \'{e}tale context when $X$ is defined over a general field $k$. For simplicity however, we stick to the complex analytic context where techniques of classical topology are used. The following theorem is a rephrasing of a theorem which appeared in \cite{BBBJ}.

\begin{thm}\label{thm:Existence} 
Let\/ $(X,\omega)$ be a $(-1)$-symplectic derived stack with orientation $(S_X, \mu_{X})$. Then  we may define a perverse sheaf\/
$ \PV_{X,\om}$ on $X$ uniquely up to canonical isomorphism. It is characterized in the following way. The Darboux theorem implies the existence of  local models
\begin{equation*}
V \stackrel{(i, \varphi)}\longrightarrow Crit(f)^{-} \times X
\end{equation*}
where $U$ is a smooth scheme, $f \in \mathcal{O}(U)$, and $V$ is a derived scheme, $\varphi$ is smooth of dimension $n,$ and the morphism $(i, \varphi)$ is an oriented Lagrangian. Here we use the canonical orientation on $Crit(f)$ and the product orientation on $Crit(f)^{-} \times X$. 
The perverse sheaf $ \PV_{X,\om}$ satisfies the following condition: $\vp^*(
\PV_{X,\om})[n],$ is canonically isomorphic to $i^*(\PV_{U,f})$ where $\PV_{U,f}$ is the perverse sheaf of vanishing cycles of $f$.  
\end{thm}

\begin{rem}
In \cite{BBBJ} it was written that $V$ is coisotropic but the fact that it is an oriented Lagrangian was not mentioned. Instead, two other properties were given. For the reader familiar with \cite{BBBJ} we now explain why our statement is equivalent to the one in \cite{BBBJ}. The Lagrangian condition is equivalent\footnote{Thank you to Chris Brav for verifying this suspicion.} to the fact  that $\mathbb{L}_{V/Crit(f)}  \cong \mathbb{T}_{V/X}[2]$. This is the first of two conditions on $V$ given in \cite{BBBJ}. Indeed we have a pair of exact triangles and a morphism between them 
\begin{equation}
\xymatrix{\mathbb{T}_{V/X}\ar[r] \ar[d] & \mathbb{T}_{V} \ar[r] \ar[d] & \mathbb{T}_{X} \ar[d]  \\\mathbb{L}_{V/Crit(f)}[-2] \ar[r]  & \mathbb{L}_{(i,\varphi)}[-2]  \ar[r] &\mathbb{L}_{X}[-1] .
}
\end{equation}
Since the rightmost and center downward arrows give isomorphisms in the homotopy category we can conclude that the leftmost downward arrow is also an isomorphism in the homotopy category. 
The orientation is a isomorphism 
\[\det (\mathbb{L}_{V})\to (i, \varphi)^{*}S_{Crit(f)^{-} \times X}
\]
inducing the canonical isomorphism $\det (\mathbb{L}_{V})^{\otimes 2} \cong \det(\mathbb{L}_{Crit(f)^{-} \times X})$. However, we can rewrite this as an isomorphism
\[K_{V/X}\otimes \varphi^{*} K_{X} \to i^{*}S_{Crit(f)} \otimes \varphi^{*}S_{X}\cong i^{*}K_{U}\otimes \varphi^{*}S_{X}
\]
or using $S_{X}^{\otimes 2} \cong K_{X}$ an  isomorphism 
\[K_{V/X}\otimes \varphi^{*} S_{X} \to i^{*}S_{Crit(f)} \cong i^{*}K_{U}
\]
or \/ $\varphi^*(S_X)\!\cong\!
i^*(K_U)\!\ot\!\La^n\bT_{V/X}$ which is the second condition of the two conditions given in \cite{BBBJ}.
\end{rem}

The theorem we are citing from \cite{BBBJ} was shown in the case of derived schemes in \cite{BBJ}, \cite{BBDJS}. In that case, one can take $\varphi$ to be smooth of dimension $0$ (in fact a Zariski open) and $i$ to be an isomorphism, which gives the following

\begin{cor}
Let\/ $(X,\omega)$ be a $(-1)$-symplectic derived scheme with orientation $(S_X, \mu_{X})$. 
For each closed point $p$ in $X$, there is an open neighbourhood $W$ symplectomorphic to $Crit(f)$ where $f$ is a regular function on a smooth scheme $U$.
Then the restriction of $ \PV_{X,\om}$ to $W$ is isomorphic to the pullback of the sheaf of vanishing cycles $\PV_{U,f}$.
\end{cor}

Joyce conjectured that there should exist a natural way to assign cycles in the cohomology of the perverse sheaf $\mP_{X}$ to Lagrangians in $X$. He made the following conjecture.

\begin{con}\label{con:Joyce}
Let $(X, \omega)$ be an oriented $(-1)$-symplectic derived stack and $i:L \longrightarrow X$ a proper oriented Lagrangian.  Let $\mathcal{P}_{(X,\omega)}$ be the perverse sheaf described in Theorem \ref{thm:Existence}. Then there is a morphism 
\[\mu_{L}: \mathbb{F}_{L}[\vdim  L]\longrightarrow i^{!}\mathcal{P}_{(X,\omega)}
\]
of constructible sheaves on $L$ with given local models in the Darboux charts.
\end{con}

In order to give some evidence for this conjecture, we first explain the construction of the map $\mu$ in a simple local model. 

\begin{example}
Let $U$ be a smooth variety over $\mathbb{C}$ equipped with a regular function $f$.  Consider the derived critical locus $X = Crit(f)$. It is equipped with the shifted symplectic structure coming from writing $Crit(f)=  U\times_{T^{*}U}U$ given by the pair of Lagrangians $df$ and $0$.  Let $\Psi:W \longrightarrow U$ be a smooth subvariety such that $f \circ \Psi = 0$. Consider the total space $N^{*}(W/U) \subset T^{*}U$ of the conormal bundle, the dual of $N_{W/U}$. Then $df \circ \Psi$ can be thought of as a section of $N^{*}(W/U) \to W$. Let $M$ be the derived zero locus of $df \circ \Psi$. Notice that $\vdim(L) = \dim W - (\dim U - \dim W) = 2 \dim W - \dim U$. 

Let us apply Corollary \ref{cor:Comp} taking $S= T^{*}U$, the three Lagrangians $X_0 = U \xrightarrow{df} T^{*}U$ and  $X_2=U \xrightarrow{0} T^{*}U$ and $X_1 = N^{*}(W/U)\xrightarrow{k} T^{*}U$ and $N_1 = W\xrightarrow{(\Psi,df\circ \Psi)} U \times_{df, T^{*}U, k} N^{*}(W/U) $ and  $N_2=W \xrightarrow{(0,\Psi)}  N^{*}(W/U) \times_{k, T^{*}U,0}U$. These are simply graphs of $df \circ \Psi$ and $0$ interpreted as sections of the shifted cotangent bundle $T^{*}[-1]W$. We conclude that natural morphism $\varphi$ from the derived zero locus 
\[L= (df \circ \Psi)^{-1}(0) = W \times_{df\circ \Psi, N^{*}(W/U), 0} W \To U \times_{df, T^{*}U,0} U = Crit(f)=X\] is Lagrangian. 

Consider the morphism 
\[j:Crit(f)\to f^{-1}(0)
\]
induced by $(j_1,j_2)$.
Let $U_{0} = f^{-1}(0)$. We have the diagram 
\begin{equation}\label{eqn:SquareVan}
\xymatrix{X \ar[r]^{j} & U_{0} \ar[r] & U \ar[r]^{f} & \mathbb{C} \\
L \ar[u]^{\varphi} \ar[r]^{i} & W \ar[u]^{\Psi} \ar[ur]_{\Psi}}
\end{equation}
Given a smooth algebraic variety $Z$, over $\mathbb{C}$, we let $\mathfrak{or}_{Z}= \mathbb{F}_{Z}[2 \dim Z]$ denote the orientation complex of its underlying topological space . For simplicity we assume that $Crit(f)$ is contained in $f^{-1}(0)$. In this case the perverse sheaf $\mathcal{P}_X$ is the pullback of the sheaf of vanishing cycles 
\[\mathcal{P}_{X} = PV_{f} = j^{*}\phi_{f}(\mathbb{F}_{U}[\dim U])
%= j_{i}^{*}\phi^{p}_{f_i}(\omega_{U_i}[-\dim U_{i}])
\]
where $j:Crit(f) \to f^{-1}(0)=U_{0}$ is the inclusion. In the general case we would sum $PV_{f-c}$ over all the critical values $c$ of $f$. 

We now construct the desired map, for this consider the canonical map $\delta_\Psi \in \Hom(\Psi_{!}\mathfrak{or}_{W}, \mathfrak{or}_{U})$. Applying the functor $\phi_{f}$ and pulling back by $j$ we obtain a morphism 
\[j^{*} \phi_{f}(\delta_\Psi): j^{*} \phi_{f} \Psi_{!}\mathfrak{or}_{W} \to  j^{*} \phi_{f} \mathfrak{or}_{U}.
\]
Noticing that $ j^{*} \phi_{f} \mathfrak{or}_{U} = \mathcal{P}_{X}[\dim U]$ and using Theorem 2.10 of \cite{BBDJS} we have since $\Psi$ is proper that
\[\phi_{f} \Psi_{!}\mathfrak{or}_{W} \cong \Psi_{!}\phi_{f\circ \Psi}(\mathfrak{or}_{W}) \cong \Psi_{!}\phi_{0}(\mathfrak{or}_{W}) \cong \Psi_{!} \mathfrak{or}_{W}.
\]
Hence, we can consider $j^{*}\phi_{f}(\delta_\Psi)$ as a morphism $j^{*}\Psi_{!}\mathfrak{or}_{W} \to \mathcal{P}_{X}[\dim U]$. Since the square (\ref{eqn:SquareVan}) is Cartesian, we have 
\[j^{*}\Psi_{!}(\mathfrak{or}_W) \cong  \varphi_{!}i^{*}(\mathfrak{or}_W)= \varphi_{!}(\mathbb{F}_{L}[2\dim W]).
\]
So we get a morphism $\varphi_{!}(\mathbb{F}_{L}) [2\dim W - \dim U]) \to \mathcal{P}_{X}$ or,
by adjunction, a morphism 
\begin{equation}\label{eqn:preMuM}
\mu_{L}: \mathbb{F}_{L}[\vdim L] \to \varphi^{!} \mathcal{P}_{X}.
\end{equation}
\end{example}

The preceding was a kind of warm-up to the general situation of $(-1)$-Lagrangians which we now consider. We will restrict ourselves to the case of derived schemes, that is both $X$ and $L$ will be derived schemes. In this situation the paper \cite{JS} provides a local description for $L$.

\begin{prop}\label{lem:general} 
Assume we have a Darboux chart $(U,f)$ for the $(-1)$-symplectic derived scheme $X$, that is, $X$ is locally equivalent to $Crit(f)$. Then \cite[Example 3.6]{JS} shows that any $(-1)$-Lagrangian $L$ in $Crit(f)$ is locally determined by the following data: a submersion $\Psi:V \to U$ of smooth varieties, a (trivial) vector bundle $E$ on $V$ equipped with an algebraic quadratic form $q$ which is non-degenerate on each fiber and a section $s$ of $E$ such that $q\circ s = f \circ \Psi$. The classical truncation of $L$ is locally isomorphic to $s^{-1}(0)$.

An orientation on $L$ determines a trivialization $\det E \cong \mathcal{O}_{V}$ and a morphism 
\[\mu_{L}: \mathbb{F}_{L}[\vdim L] \to \varphi^{!} \mathcal{P}_{X}\]
\end{prop}
\begin{proof}
Recall from Example 3.6 in \cite{JS} that the cotangent complex of $L$ has the form 
\[\mathbb{L}_{L} = {[T_{V/U} \to E^{\vee} \to T_{V/U}^{\vee}\oplus T^{\vee}_{U}]|}_{L}
\]
living in degrees $-2,-1,0$. In particular the virtual dimension of $L$ is \[\vdim(L) = 2 \dim V- \dim U -rk E.\] The cotangent complex of the derived critical locus $X$ of $f$ looks like
\[{[T_{U} \to T_{U}^{\vee}]|}_{X}
\]
in degrees $-1,0$.

These cotangent complexes give us natural isomorphisms $\det(\mathbb{L}_{L})\cong \det(T^{\vee}_{U})|_{L}\otimes \det(E)|_{L}$ and $\det(\mathbb{L}_{X}) \cong \det(T^{\vee}_{U})^{\otimes 2}|_{X}$. The morphism $\alpha: \det (\mathbb{L}_{L})^{\otimes 2} \to (\det \mathbb{L}_{X})|_{L}$ determined by the the Lagrangian structure  can be thought of therefore as a morphism $\det(T^{\vee}_{U})^{\otimes 2}|_{L} \otimes (\det E)|_{L}^{\otimes 2} \to \det(T^{\vee}_{U})^{\otimes 2}|_{L}$, the orientation $\beta$ is therefore a trivialization of $\det(E)$ along $L$ which comes from restricting the given isomorphism $\det(E) \cong \mathcal{O}_{V}$ to $L$.

Now we construct the morphism $\mu$. As in the example we will assume, for notational simplicity, that $Crit(f)$ is contained in $f^{-1}(0)$. We have the following diagram
\begin{equation}\label{eqn:SquareVect}
\xymatrix{X  \ar[r]^{i} & f^{-1}(0)\ar[r]^{} & U \ar[r]^{f} &  \mathbb{C} \\
L  \ar[u]^{\varphi} \ar[r]^{j} & (q \circ s)^{-1}(0) \ar[u]^{\Psi_{0}} \ar[r] & V \ar[u]^{\Psi} \ar[r]^{s}& E\ar[u]^{q}}
\end{equation}
Let $s_0$ denote the composition $(q \circ s)^{-1}(0) \to V \to E$. Since $L=s^{-1}(0)$ and $i$ is a proper closed embedding we get the following chain of isomorphisms
\[\varphi_{!}\mathbb{F}_{L} \cong i^{*}i_{!}\varphi_{!}\mathbb{F}_{L} \cong i^{*}\Psi_{0!}j_{!}\mathbb{F}_{L} \cong i^{*}\Psi_{0!} s_0^{*} \mathbb{F}_{0_E},
\]
(using Lemma \ref{lem:FunFacts}) where $\mathbb{F}_{0_E}$ is the pushforward of $\mathbb{F}$  from the zero section of $E$ to $E$.
Now the well known description of the sheaf of vanishing cycles of a non-degenerate quadratic function tells us that $\mathbb{F}_{0_E}\cong \phi_{q} \mathbb{F}_{E}[rk E]$. Composing this with the previous chain of isomorphisms we obtain:
\begin{equation}\label{eq:mu1}
\varphi_{!}\mathbb{F}_{L}[-rk E] \cong i^{*}\Psi_{0!} s_0^{*} \phi_{q } \mathbb{F}_{E}
\end{equation}

Now recall there is a canonical map $\mathbb{F}_{E} \to s_{*} \mathbb{F}_{V}$. Applying the functor $i^{*}\Psi_{0!} s_0^{*} \phi_{q }$ to this map we get a morphism
\begin{equation}\label{eq:mu2}
i^{*}\Psi_{0!} s_0^{*} \phi_{q } \mathbb{F}_{E} \to i^{*}\Psi_{0!} s_0^{*} \phi_{q }s_{*} \mathbb{F}_{V}   \cong  i^{*}\Psi_{0!} s_0^{*} s_{0*}\phi_{q \circ s}\mathbb{F}_{V} \cong  i^{*}\Psi_{0!} \phi_{q \circ s}\mathbb{F}_{V} .
\end{equation}
Here we have used (\ref{eqn:nonProppush}) and the fact that $s_0$ is proper.

As in the previous example, there is a canonical morphism $\delta_\Psi: \Psi_{!} \mathbb{F}_{V}[2\dim V] \cong \Psi_{!} \mathfrak{or}_{V} \to \mathfrak{or}_{U}\cong \mathbb{F}_{U}[2\dim U]$. We will apply the functor $\phi_{f}$ to this morphism and then pull back to $X$ via $i$. Then precomposing this map with (\ref{eqn:nonProppush}) we get the map
\begin{equation}\label{eq:mu3}
i^{*}\Psi_{0!} \phi_{q \circ s}\mathbb{F}_{V}=i^{*}\Psi_{0!} \phi_{f \circ \Psi}\mathbb{F}_{V} \to   i^{*} \phi_{f}\Psi_{!}\mathbb{F}_{V}  \to i^{*} \phi_{f} \mathbb{F}_{U}[2\dim U -2 \dim V],
\end{equation} 
where the equality follows from the assumption $q \circ s = f\circ \Psi$. 
Finally we compose (\ref{eq:mu1}), (\ref{eq:mu2}) and (\ref{eq:mu3}) and get 
\[\varphi_{!}\mathbb{F}_{L}[-rk E] \To i^{*} \phi_{f} \mathbb{F}_{U}[2\dim U -2 \dim V]=\mathcal{P}_{X}[\dim U -2 \dim V],
\]
and the equality follows from the definition of $\mathcal{P}_X$.
By adjunction, this corresponds to a morphism $\mu_{L} : \mathbb{F}_{L}[\vdim L] \to \varphi^{!} \mathcal{P}_{X}$.
\end{proof}

\begin{rem}
The previous proposition proves Joyce's conjecture locally (for derived schemes). The main difficulty in giving a complete proof of the conjecture is to glue these maps along a cover of $L$ by Darboux charts. Note that these are general maps in $D^{b}_{c}(X)$, not necessarily perverse and hence do not glue like sheaves.
\end{rem}

We now formulate a more detailed version of Joyce's conjecture which implies the phrasing in Conjecture \ref{con:Joyce} and includes the behavior of the map $\mu$ under composition of Lagrangian correspondences. In the next section we will use this conjecture to construct a bicategory which is a linear version of $\fkSymp^{or}$.

\begin{con}\label{con:OurConj}
Let $(X_0,\omega_0)$ and $(X_1, \omega_1)$ be oriented $(-1)$-symplectic derived stacks. Let $\varphi=\varphi_0 \times \varphi_1 : M \to X_{0}^{-} \times X_1$ be an oriented Lagrangian, in the sense of Definition \ref{defn:OrLag} such that $\varphi_1$ is proper.  Then
\begin{enumerate}
\item[{\bf (a)}]
There is a map of constructible sheaves
\[\mu_{M}: \varphi_{0}^{*} \mP_{X_0}[v] \To \varphi_{1}^{!}\mP_{X_1}
\] where $v= \vdim(M)$  with given local models in Darboux charts. Moreover, if we reverse the orientation of $M$ we change the map by $-1$.

\item[{\bf (b)}] Consider oriented Lagrangians $\varphi: M \to X_{0}^{-} \times X_1$ and $\psi: N \to X_{0}^{-} \times X_1$ and let $\rho:M \to N$ be an oriented Lagrangeomorphism. Then the morphism 

\[\varphi_{0}^{*} \mP_{X_0}[v] \cong \rho^{*}\psi_{0}^{*} \mP_{X_0}[v] \xrightarrow{\rho^{*}\mu_{N}} \rho^{*}\psi_{1}^{!} \mathcal{P}_{X_1} =  \rho^{!}\psi_{1}^{!} \mathcal{P}_{X_1} \cong \varphi_{1}^{!} \mP_{X_1}
\]
agrees with $\mu_{M}$ in $D^{b}_{c}(M)$.
\item[{\bf (c)}] Given oriented Lagrangians $\varphi: M \to  X_{0}^{-} \times X_1$ and $\psi: N \to  X_{1}^{-} \times X_2$ with $v_{M} = \vdim(M)$ and $v_{N}=\vdim(N)$ equip the Lagrangian $N \bullet M$  with the orientation constructed in Lemma \ref{lem:OrienConv}. Then the maps 
\[\mu_{N \bullet M}: (\varphi_{0} \circ \pi_{M})^{*}\mP_{X_0}[v_M + v_N]\longrightarrow (\psi_{2} \circ \pi_{N})^{!}\mP_{X_2}
\]
and the following composition,
\[\pi_{M}^{*}\varphi^{*}_{0}\mP_{X_0}[v_M+v_N]\xrightarrow{\pi_{M}^{*}\mu_{M}[v_N]} \pi_{M}^{*}\varphi_{1}^{!}\mP_{X_1}[v_N]\xrightarrow{c_{\varphi_1,\psi_1}} \pi_{N}^{!}\psi_{1}^{*}\mP_{X_1}[v_N] \xrightarrow{\pi_{N}^{!}\mu_{N}} \pi_{N}^{!}\psi_{2}^{!}\mP_{X_2}
\]
agree in $D^{b}_{c}(M \times_{X_1}N)$. Note that this statement makes sense because $\vdim(X_1)=0$ and so $\vdim(M \times_{X_1}N)= \vdim(M) + \vdim(N) $.
\item[{\bf (d)}] If $\varphi$ is the diagonal  $\Delta: X\to X^{-} \times X$ then since $\vdim X = 0 $ the resulting morphism $\mu_{X}: \mP_{X} \to \mP_{X}$ is the identity.
\item[{\bf (e)}] If we are given Lagrangians $\varphi: M \to X_{0}^{-} \times X_{1}$ and  $\psi: N \to Y_{0}^{-} \times Y_{1}$ the morphism 
\[\mP_{X_0 \times Y_0} [v_M+v_N ] \cong 
(\mP_{X_0} [v_M])\boxtimes ( \mP_{Y_0}[v_N]) \xrightarrow{\mu_{M} \boxtimes \mu_{N}} \mP_{X_1} \boxtimes  \mP_{Y_1} \cong \mP_{X_1 \times Y_1}
\]
agrees with the morphism $\mu_{M \times N}$ corresponding to the Lagrangian 
\[\phi \times \psi: M \times N\longrightarrow (X_0 \times Y_0)^{-} \times (X_1 \times Y_1).\]
where $v_M= \vdim M$ and $v_N= \vdim N$.
\end{enumerate}
\end{con}

\begin{rem}\label{rem:productP}
We observe that the statements of Conjecture \ref{con:Joyce} and Conjecture \ref{con:OurConj}\/(a) are equivalent using the fact that for a product of oriented $(-1)$-symplectic derived stacks we have $\mathcal{P}_{X_0 \times X_1} \cong \mathcal{P}_{X_0}\boxtimes \mathcal{P}_{X_1}$ when we take the product orientation on $X_0 \times X_1$. 

To check this is true first note the isomorphism can be checked locally as they are perverse sheaves. Examining the Darboux theorem in \cite{BBBJ} we can see that if, using the notation from Theorem \ref{thm:Existence}, $X_i$ has local Darboux data $V_i, U_i, f_i$ then for the product $X_1\times X_2$ we can take $V=V_1\times V_2\To Crit(f_1\boxplus f_2)^- \times X_1\times X_2$, with product morphisms. Then the claim follows from the Thom-Sebastiani isomorphism (see Theorem 2.1.3 of \cite{BBDJS}) for the perverse sheaf of vanishing cycles of $f_1\boxplus f_2$. 

%In the case that $X_0$ and $X_1$ are derived schemes, then f locally $X_i$ is given by $Crit(f_i)$ then $X_0 \times X_1$ is equivalent to $Crit(f_0\boxplus f_1)$. Finally the Thom-Sebastiani isomorphism (see Theorem 2.1.3 of \cite{BBDJS}) for the perverse sheaf of vanishing cycles implies the claim. In the case that the $X_i$ are derived Artin stacks, chose smooth atlases $u_{i}:A_i \to X_i$ of dimension $d_i$. Then $A = A_0 \times A_1 \to X_0 \times X_1=X$ is a smooth atlass. Now $\mathcal{P}_{X}$ is determined its pullback $\mathcal{P}_{A}[d_1+d_2]$ and the isomorphism of its two pullbacks to $A \times_{X}A$. The perverse sheaf  $\mathcal{P}_{A}$ is determined by the structure on $A$ of a $d$-critical locus.  Similarly, the perverse sheaves $\mathcal{P}_{A_i}$ are determined by considering on $X_i$ the natural structure of $d$-critical stacks and considering $A_i$ with their structure as $d$-critical loci, see Theorem 4.8 of \cite{BBBJ}.  We have $\mathcal{P}_{A}[d_1+d_2] \cong (\mathcal{P}_{A_1}[d_1])\boxtimes (\mathcal{P}_{A_2}[d_2])$ as can be seen by looking in $d$-critical charts on $A$ which are the product of charts for $A_1$ and $A_2$ and again applying Thom-Sebastiani. Similarly, the descent cocycles on $A \times_{X}A \cong (A_{0} \times_{X_0} A_{0})\times (A_{1} \times_{X_1} A_{1})$ agree and so the descent data which determines the sheaf $\mathcal{P}_{X_1}\boxtimes \mathcal{P}_{X_2}$, agrees with that which determines the sheaf $\mathcal{P}_{X}$ and so the natural morphism $\mathcal{P}_{X} \to \mathcal{P}_{X_1}\boxtimes \mathcal{P}_{X_2}$ is an isomorphism.
\end{rem}

\section{A linearization of $\fkSymp^0$ and $\fkLag(S)$}

In this section, we will first construct a bicategory $\fkLSymp$, whose objects and $1$-morphisms agree with those of $\fkSymp^{or}$, but it is linear at the level of $2$-morphisms. We will also construct a  homomorphism of bicategories $\fkSymp^{or}_c \To \fkLSymp$. In both cases we will use Conjecture \ref{con:OurConj}. In the last subsection we will see how the same construction gives, for any $1$-symplectic derived stack $S$, an oriented version $\fkLag^{or}(S)$ and a linearization $\fkLLag(S)$ of $\fkLag(S)$.

\subsection{A linearized $2$-category of 0-symplectic derived stacks}\label{Ssec:PL}\

Here we will define the bicategory $\fkLSymp$. Before we define the objects and morphism in this bicategory we make some useful observations. If $S$ is a $0$-symplectic derived stack then it has even virtual dimension and if $X$ is a Lagrangian in $S$, it follows from the definition of Lagrangian that $\vdim(X)=\frac{1}{2}\vdim(S)$.

\begin{defn}
The objects of $\fkLSymp$ are the same as the objects of $\fkSymp^{or}$, namely $0$-symplectic derived stacks $(S, \omega)$ together with a line bundle $E$ on $S$. The $1$-morphisms are the same as in $\fkSymp^{or}$, so we have $\fkLSymp_{1}(S_0, S_1):=\fkSymp^{or}_{1}(S_0, S_1)$. 

If $X_0$ and $X_1$ are $1$-morphisms, then by Lemma \ref{lem:OrienProds}, the $(-1)$-symplectic derived stack $X_{01}=X_0\times_{S_0\times S_1} X_1$ has an induced orientation. Theorem \ref{thm:Existence} then constructs a perverse sheaf $\mathcal{P}_{X_{01}}$. We define the graded vector space of $2$-morphisms as
\[\fkLSymp_{2}(X_0,X_1):=\mathbb{H}^{\bullet}(X_{01}, \mathcal{P}_{X_{01}}[-n_0-n_1]),
\]
where $n_i= \frac{1}{2} \vdim (S_i)$.
\end{defn}

We now define the different compositions and identities in $\fkLSymp$.

\begin{defn}
Composition of $1$-morphisms in $\fkLSymp$ is defined in the same way as in $\fkSymp^{or}$. Similarly, the identity $1$-morphisms $\text{id}_{X}$ are defined to be same as the ones in $\fkSymp^{or}$. 

Let $S_i$ be objects in $\fkLSymp$ with $\vdim(S_i)=2n_i$. Given $X_0, X_1, X_2 \in \fkLSymp_{1}(S_0, S_1)$, Lemma \ref{lem:OrienProds} implies that the inclusion 
\[\varphi:X_{012} \to ( X_{12} \times X_{01})^{-} \times X_{02}\] 
is a $(-1)$-Lagrangian of virtual dimension $-n_0-n_1$ equipped with an induced orientation. Here we reverse the order of the first two factors to respect the usual convention for compositions in a category.
Since $\varphi$ is proper, Conjecture \ref{con:OurConj}(a) combined with Lemma \ref{lem:XWdiag}(a) gives a morphism 
\[(\mu_{X_{012}})_{*}: \mathbb{H}^{\bullet}(\mathcal{P}_{X_{12} \times X_{01}}[-n_0-n_1]) \to \mathbb{H}^{\bullet}(\mathcal{P}_{X_{02}}).
\]
Applying the shift $[-n_0-n_1]$ and using the isomorphism $\mathcal{P}_{ X_{12} \times X_{01}} \cong  \mathcal{P}_{X_{12}}\boxtimes \mathcal{P}_{X_{01}}$, explained in Remark \ref{rem:productP},
 and the K\"unneth isomorphism we obtain a map
\[ (\mu_{X_{012}})_*[-n_0-n_1]: \mathbb{H}^{\bullet} (\mathcal{P}_{X_{12}}[-n_0-n_1]) \otimes \mathbb{H}^{\bullet}(\mathcal{P}_{X_{01}}[-n_0-n_1])  \to \mathbb{H}^{\bullet}(\mathcal{P}_{X_{02}}[-n_0-n_1]).
\]
We define the vertical composition of $2$-morphisms as 
\[ a_2 \odot a_1 = (-1)^{(|a_2|-n_0-n_1)(n_0+n_1)}(\mu_{X_{012}})_*[-n_0-n_1](a_2 \otimes a_1).\] 

Consider $Y_{0}, Y_1 \in \fkLSymp_{1}(S_1,S_2)$ and denote $Z_{i} = Y_{i} \circ X_{i} \in  \fkLSymp_{1}(S_0,S_2)$. By Lemma \ref{lem:LagUniqP01} there is a natural Lagrangian
\[Z_{0} \times_{S_0 \times S_1 \times S_2} Z_1 \to (Y_{01} \times X_{01})^{-} \times Z_{01}.
\]
As explained in the proof of Proposition \ref{prop:6} this Lagrangian can be described as a triple intersection of oriented $0$-Lagrangians and hence Lemma \ref{lem:OrienProds} (b) assigns it an orientation. As above, since this Lagrangian is proper, we obtain a map 
\[(\mu_{Z_{0} \times_{S_0 \times S_1 \times S_2} Z_1})_*: \mathbb{H}^{\bullet} (\mathcal{P}_{Y_{01} \times X_{01}}[v]) \to \mathbb{H}^{\bullet}(\mathcal{P}_{Z_{01}})
\]
where $v= \vdim(Z_{0} \times_{S_{0} \times S_{1} \times S_{2}} Z_{1}) = -2 n_1$. If we apply the shift $[-n_0-n_2]$ we obtain a map
\[(\mu_{Z_{0} \times_{S_0 \times S_1 \times S_2} Z_1})_*[-n_0 - n_2]:  \mathbb{H}^{\bullet}(\mathcal{P}_{Y_{01}}[-n_1 - n_2]) \otimes \mathbb{H}^{\bullet}(\mathcal{P}_{X_{01}}[-n_0 - n_1])  \to \mathbb{H}^{\bullet}(\mathcal{P}_{Z_{01}}[-n_0 - n_2]).
\]
We define the horizontal composition of $2$-morphisms as
\[b* a= (-1)^{(n_0+n_1)(n_1+n_2)}(\mu_{Z_{0} \times_{S_0 \times S_1 \times S_2} Z_1})_*[-n_0 - n_2](b \otimes a)
\]
\end{defn}

In order to define the identity $2$-morphisms, associators and unitors we need the following  

\begin{lem}\label{lem:eMors}
Let $X_0$, $X_1$ be $1$-morphisms. An oriented Lagrangeomorphism $\rho:X_0 \to X_1$ determines a $2$-morphism $e_{\rho} \in \fkLSymp_{2}(X_0, X_1) = \mathbb{H}^{\bullet}(\mathcal{P}_{X_{01}}[-n_0-n_1])$, sometimes we denote it by $e_M$, with $M=\Gamma_\rho$.

If $\rho= \text{id}_{X}$ then $e_{\rho}$ is an identity for the operation $ \odot $. Moreover for any $\rho$, the $2$-morphism $e_{\rho}$ is invertible with respect to $ \odot $.
\end{lem}
\begin{proof} 
By the definition of Lagrangeomorphism, its graph $\Gamma_{\rho}:X_{0} \to X_{01}$ is an oriented $(-1)$-Lagrangian of virtual dimension $n_0+n_1$. This can be thought of as a Lagrangian correspondence from a point to $X_{01}$, that is a Lagrangian $X_0 \to (\bullet_{(-1)})^-\times X_{01}$. Since it is proper, we can apply Conjecture \ref{con:OurConj}(a)  and Lemma \ref{lem:XWdiag} (a) to this Lagrangian and obtain a map 
\[(\mu_{\Gamma_{\rho}})_{*}:\mathbb{H}^{\bullet}(\mathcal{P}_{\bullet_{(-1)}}[n_0+n_1])\to \mathbb{H}^{\bullet}(\mathcal{P}_{X_{01}}).
\]
Applying the shift $[-n_0-n_1]$ and using the fact that $\mathbb{H}^{\bullet}(\mathcal{P}_{\bullet_{(-1)}}) \cong \mathbb{F}$, we obtain the map 
\[(\mu_{\Gamma_{\rho}})_{*}[-n_0-n_1]:\mathbb{F}\to \mathbb{H}^{\bullet}(\mathcal{P}_{X_{01}}[-n_0-n_1]) \cong \fkLSymp_{2}(X_0,X_1).
\]
We define $e_{\rho}= (\mu_{\Gamma_{\rho}})_{*}[-n_0-n_1](1)\in \fkLSymp_2(X_0,X_1)$. 

Let $\rho_{1}:X_{0} \to X_{1}$ and $\rho_{2}:X_{1} \to X_{2}$ be Lagrangeomorphisms, recall from the proof of Proposition \ref{prop:IdOriginal} that $\rho_2\circ\rho_1$ is also a Lagrangeomorphism. Moreover we have an oriented Lagrangeomorphism $\Gamma_{\rho_2} \odot \Gamma_{\rho_1}\cong \Gamma_{\rho_2\circ\rho_1}$, where $\odot$ is the vertical composition in the category $\fkSymp^{or}$. We claim that
\begin{equation}\label{eqn:RhoComp}
e_{\rho_2} \odot e_{\rho_1} = e_{\rho_2 \circ \rho_1}
\end{equation}
In order to prove this, we compute 
\begin{equation}
\begin{split}e_{\rho_2} \odot e_{\rho_1} & = (-1)^{n_0+n_1}(\mu_{X_{012}})_{*}[-n_0-n_1]\big((\mu_{\Gamma_{\rho_2}})_{*}[-n_0-n_1](1)\otimes (\mu_{\Gamma_{\rho_1}})_{*}[-n_0-n_1](1)\big) \\
&= (-1)^{n_0+n_1}(\mu_{X_{012}})_{*}[-n_0-n_1]\big((\mu_{\Gamma_{\rho_2} \times \Gamma_{\rho_1}})_{*}[-2n_0-2n_1](1)\big)
\\
&=(-1)^{n_0+n_1}(\mu_{X_{012} \bullet (\Gamma_{\rho_2} \times \Gamma_{\rho_1})})_{*}[-n_0-n_1](1)
\\
&=(\mu_{ \Gamma_{\rho_2} \odot\Gamma_{\rho_1}})_{*}[-n_0-n_1](1)
\\
&= (\mu_{ \Gamma_{\rho_2 \circ \rho_1}})_{*}[-n_0-n_1](1)=e_{{\rho_2} \odot {\rho_1}}.
\end{split}
\end{equation} 
Here the first, fourth and last equalities follow from the definitions, the second from Conjecture \ref{con:OurConj}(e) combined with Lemma \ref{lem:XWdiag}(e), the third equality follows from combining Conjecture \ref{con:OurConj}(c) with Lemma \ref{lem:XWdiag}(c) and finally the fifth equality follows Conjecture \ref{con:OurConj}(b) and Lemma \ref{lem:XWdiag}(b).

Now equation (\ref{eqn:RhoComp}) together with Proposition \ref{prop:IdOriginal} immediately implies the second half of the statement, namely: 
\[ e_{\rho}\odot e_{\text{id}_{X_0}}  = e_{\rho} = e_{\text{id}_{X_1}} \odot e_{\rho}
\]
and
\[e_{\rho} \odot e_{\tilde{\rho}}= e_{\text{id}_{X_0}},\ \  e_{\tilde{\rho}}\odot e_{\rho}= e_{\text{id}_{X_1}}
\]
where $\tilde{\rho}$ is a homotopy inverse of $\rho$. 
\end{proof}

The previous lemma allows us to make the following definition

\begin{defn} 
Let $S_0,S_1, S_2$ and $S_3$ be objects in $\fkLSymp$ and consider $X_{i} \in \fkLSymp_{1}(S_{i-1}, S_{i})$ for $i=1,2,3$.
The identity $2$-morphism of $X_1$ is defined as $1_{X_1} = e_{\text{id}_{X_1}}$. 
Let $W_{X_3 X_2 X_1}$ be the associator in $\fkSymp^{or}$, we define the associator in $\fkLSymp$, still denoted $W_{X_3 X_2 X_1}$, as $e_{W_{X_3 X_2 X_1}}$.
 
Next, we define the unitors in $\fkLSymp_{2}(I_{S_0}\circ X_1, X_1)$ and 
$\fkLSymp_{2}(X_1 \circ I_{S_1}, X_1)$ as $e_{l_{X_1}}$ and $e_{r_{X_1}}$ where $l_{X_1}$ and $r_{X_1}$ are the unitors in $\fkSymp^{or}$. 

By Lemma \ref{lem:eMors} all of these $2$-morphisms are $2$-isomorphisms as required. 
\end{defn}

We now have all the data needed to define a bicategory, we will now check the axioms.

\begin{lem}
The vertical composition $\odot$ is associative and the $2$-morphisms $1_X$ are identities.
\end{lem}
\begin{proof} 
Consider $X_0, X_1, X_2, X_3 \in \fkLSymp_1(S_0, S_1)$ and take $a_i \in \fkLSymp_1(X_{i-1}, X_i)$. We denote $n=\vdim X_i$ and compute
\begin{equation}
\begin{split}
a_3 \odot ( a_2 \odot a_1)&= (-1)^{(|a_2|+|a_3|)n}(\mu_{X_{023}})_{*}[-n]( (\mu_{\Delta_{X_{23}}})_*(a_3) \otimes (\mu_{X_{123}})_{*}[-n] (a_2 \otimes a_1) )\\
& =(-1)^{|a_2|n +n} (\mu_{X_{023}})_{*}[-n]((\mu_{\Delta_{X_{23}}\times X_{123}})_*[-n](a_3 \otimes a_2 \otimes a_1)) \\
& = (-1)^{|a_2|n +n} \big(\mu_{X_{023} \bullet (\Delta_{X_{23}} \times X_{012})}\big)_{*}[-n](a_3 \otimes a_2 \otimes a_1))
\end{split}
\end{equation}
where the first equality follows from the definitions, together with Conjecture \ref{con:OurConj}(d) and Lemma \ref{lem:XWdiag}(d); the second one follows from Conjecture \ref{con:OurConj}(e) and Lemma \ref{lem:XWdiag}(e) and the fact that $(\mu_{X_{123}})_{*}$ has degree $n$. Finally the third equality follows from Conjecture \ref{con:OurConj}(c) and Lemma \ref{lem:XWdiag}(c).

Similarly, 
\[(a_3 \odot a_2) \odot a_1= (-1)^{|a_2|n} \big(\mu_{X_{013} \bullet (X_{123}\times \Delta_{X_{01}})}\big)_{*}[-n](a_3 \otimes a_2 \otimes a_1),
\]
hence associativity follows from applying Conjecture \ref{con:OurConj}(b) to the Lagrangeomorphism 
$$X_{023} \bullet (\Delta_{X_{23}}\times X_{012}) \cong (-1)^n X_{013}\bullet ( X_{123}\times \Delta_{X_{01}} )$$ 
proven in Corollary \ref{cor:NotationX}, without considering the orientations, but that is elementary.
The statement about the identity $2$-morphisms follows from the second part of Lemma \ref{lem:eMors}.
\end{proof}

\begin{lem}
Consider $X_i \in\fkLSymp_{1}(S_0, S_1)$ and $Y_i \in\fkLSymp_{1}(S_1, S_2)$ for $i=0,1,2$ and denote $Z_i=Y_i\circ X_i$. For $a_{i} \in \fkLSymp_{2}(X_{i-1}, X_{i})$ and $b_{i} \in \fkLSymp_{2}(Y_{i-1}, Y_{i})$ for $i=1,2$, we have
\[(b_{2} * a_{2}) \odot (b_{1} * a_{1}) = (-1)^{|b_1||a_2|}(b_{2} \odot b_{1}) * (a_{2} \odot a_{1})
\]
\end{lem}
\begin{proof}
The proof follows the proof of the same statement for the bicategory $\fkSymp^{or}$ in Theorem \ref{thm:sympor}, using Conjecture \ref{con:OurConj} and Lemma \ref{lem:XWdiag}, instead of properties $C_{-}$. We have
\[(b_{2} * a_{2}) \odot (b_{1} * a_{1}) =(-1)^{\epsilon_1}\big(\mu_{Z_{012}\bullet(Z_1 \times_{S_{012}}Z_2 \times Z_0\times_{S_{012}} Z_1)}\big)_*[-n_0-n_2](b_2\otimes a_2\otimes b_1\otimes a_1)
\]
\[(b_{2} \odot b_{1}) * (a_{2} \odot a_{1})=(-1)^{\epsilon_2}\big(\mu_{(Z_0\times_{S_{012}}Z_2)\bullet(Y_{012} \times X_{012})\bullet \Gamma_\rho}\big)_*[-n_0-n_2](b_2\otimes a_2\otimes b_1\otimes a_1)
\]
where $\epsilon_1=(|b_2|+|a_2|+n_0+n_2)(n_0+n_2)$ and $\epsilon_2=(n_0+n_1)(n_1+n_2)+(|b_2|+n_1+n_2)(n_1+n_2)+(|a_2|+n_0+n_1)(n_0+n_1)+(|b_1|+n_1+n_2)(|a_2|+n_1+n_2)$. Here we have used the fact that 
$$(\mu_{\Gamma_\rho})_*(b_2\otimes b_1\otimes a_2\otimes\ a_1)=(-1)^{(|b_1|+n_1+n_2)(|a_2|+n_1+n_2)}b_2\otimes a_2\otimes b_1\otimes a_1,$$
where the sign corresponds to the unshifted degrees in $\mathbb{H}^{\bullet}(\mathcal{P}_{X_{12}})$ and $\mathbb{H}^{\bullet}(\mathcal{P}_{Y_{01}})$. 

Note that $\epsilon_1+\epsilon_2=|b_1||a_2| \pmod 2$, therefore the statement follows from the existence of the following oriented Lagrangeomorphism
\[Z_{012}\bullet(Z_1 \times_{S_{012}}Z_2 \times Z_0\times_{S_{012}} Z_1) \cong (Z_0\times_{S_{012}}Z_2)\bullet(Y_{012} \times X_{012})\bullet \Gamma_\rho,
\]
which is analogous to the one constructed in the proof of Theorem \ref{thm:sympor}.
\end{proof}

\begin{lem}
The associator satisfies the pentagon axiom and the unitors satisfy the triangle axiom.
\end{lem}
\begin{proof}
Consider $X_i \in \fkLSymp_1(S_{i-1}, S_i)=\fkSymp^{or}(S_{i-1}, S_i)$, for $i=1,..4$, the pentagon axiom in $\fkSymp^{or}$ states that the following oriented $(-1)$-Lagrangians are Lagrangeomorphic
\[W_{43(21)} \odot W_{(43)21} \cong (1_{X_4} * W_{321}) \odot W_{4(32)1} \odot (W_{432}*1_{X_1}) 
\]
By definition and Lemma \ref{lem:eMors} , we have 
\[W_{43(21)} \odot W_{(43)21} = e_{W_{43(21)}} \odot e_{W_{(43)21}} = e_{W_{43(21)} \odot W_{(43)21} }. 
\]
Similarly we have
\begin{equation}
\begin{split} (1_{X_4} * W_{321}) \odot W_{4(32)1} \odot (W_{432}*1_{X_1}) & = e_{1_{X_4} * W_{321}} \odot e_{W_{4(32)1}} \odot e_{W_{432}*1_{X_1}} \\
& = e_{(1_{X_4} * W_{321}) \odot W_{4(32)1} \odot (W_{432}*1_{X_1})},
\end{split}
\end{equation}
an therefore the pentagon axiom in $\fkLSymp$ follows from the pentagon axiom in $\fkSymp^{or}$ combined with Conjecture \ref{con:OurConj}(b) and Lemma \ref{lem:XWdiag}(b). By an analogous argument we can see that the triangle axiom in $\fkSymp^{or}$ implies the triangle axiom in $\fkLSymp$.
\end{proof}

The proof of the next lemma is very similar to others in this section so we omit it.

\begin{lem}
The associator  and the unitors are natural, meaning that given $2$-morphisms $b_i \in \fkLSymp_2(X_{i}, Y_{i})$, for $i=1,2,3$, we have
\[(b_{3}*(b_2 *b_1 )) \odot W_{X_3 X_2 X_1} =  W_{Y_3 Y_2 Y_1} \odot ((b_3 * b_2) * b_1),
\]
and 
\[b_0 \odot r_{X_0}= r_{Y_0} \odot (b_0* 1_{I_{S_0}}) \ \ \textrm{and} \ \ b_0 \odot l_{X_0}= l_{Y_0} \odot (1_{I_{S_1}}* b_0). 
\]
\end{lem}

Summarizing the results from this subsection we have the following

\begin{thm}
The definitions and lemmas above define a bicategory $\fkLSymp$ enriched over graded vector spaces. Moreover it has a symmetric monoidal structure.
\end{thm}
\begin{proof}
%The monoidal structure on $\fkLSymp$ is defined on objects by 
%\[((S_0, \omega_0), E_0) \otimes ((S_1,\omega_1), E_1) = ((S_0 \times S_1, \omega_0 \boxplus  %omega_1), E_{0} \boxtimes E_{1}).
%\]
%The point serves as the units. We can now describe the monoidal structure on $1$-morphisms. Given $X %\in \fkLSymp_{1}((S_0, E_0), (S_1, E_1))$ and $Y \in \fkLSymp_{1}((S_0, E_0), (S_1, E_1))$ then $X 
%\times Y$ is $(E_{0} \boxtimes F_{0})^{-1} \boxtimes (E_{1} \boxtimes F_{1})$-oriented in the way %discussed in Lemma \ref{lem:OrientedProduct}. 

The only point left to discuss is the symmetric monoidal structure. At the level of objects and $1$-morphisms it is the same as $\fkSymp^{or}$. In order to define the monoidal structure on $2$-morphisms
we use the following canonical isomorphisms
%say we are given 
%\[\alpha \in \fkLSymp_{2}(X_0, X_1) = \mathbb{H}^{\bullet}(\mathcal{P}_{X_{01}}[-n_X])
%\]
% and 
%\[\beta \in \fkLSymp_{2}(Y_0, Y_1) = \mathbb{H}^{\bullet}(\mathcal{P}_{X_{01}}[-n_X]).
%\] 
%Their product comes from the canonical isomorphisms 
\begin{equation}
\begin{split}
\fkLSymp_{2}(X_{0} \times Y_{0}, X_{1} \times Y_{1}) & \cong \mathbb{H}^{\bullet}(\mathcal{P}_{(X_0 \times Y_0)\times_{S_0 \times S_1 \times T_0 \times T_1} (X_1 \times Y_1)}[-n_X - n_Y]) \\ & \cong  \mathbb{H}^{\bullet}(\mathcal{P}_{X_{01} \times Y_{01}}[-n_X - n_Y]) \\ & \cong  \mathbb{H}^{\bullet}(\mathcal{P}_{X_{01}}[-n_X]) \otimes \mathbb{H}^{\bullet}(\mathcal{P}_{Y_{01}})[- n_Y]),
\end{split}
\end{equation}
where $n_X=\vdim X_0 \times X_1$ and $n_Y=\vdim Y_0 \times Y_1$. The structure of symmetric monoidal structure can then be constructed in a straightforward way.
\end{proof}

\subsection{The linearization functor}\label{subsec:linfunct}\ 

In the previous subsection we used Conjecture \ref{con:OurConj} to construct the  $2$-category $\fkLSymp$. In this subsection, again using Conjecture \ref{con:OurConj} we would like to construct a linearization functor, that is a (symmetric monoidal) homomorphism of bicategories $\fkSymp^{or}\to \fkLSymp$. This is not possible since in order to apply Conjecture \ref{con:OurConj} we need proper $(-1)$-Lagrangians. Because of this we will introduce a slightly modified version of $\fkSymp^{or}$.

\begin{prop}
There is a symmetric monoidal bicategory $\fkSymp^{or}_c$ defined as the subcategory of $\fkSymp^{or}$ with the same objects and $1$-morphisms and $2$-morphisms $(\fkSymp^{or}_c)_2(X_0,X_1)$ are (equivalence classes) of oriented Lagrangians $\psi: M \To X_{01}$, such that $\psi$ is a proper map.
\end{prop}
\begin{proof}
It easily follows from the definitions that being proper is preserved by both horizontal and vertical composition in $\fkSymp^{or}$. All the other data required in the definition of a symmetric monoidal bicategory (identities, associators, unitors,..) are defined as the graph of some Lagrangeomorphism which is necessarily proper.
\end{proof}

We can now state the main result of this subsection.

\begin{thm}\label{thm:prequant}
There is a symmetric monoidal homomorphism of bicategories, in the sense of Definition 2.2 of \cite{S-P}, 
\[F: \fkSymp^{or}_c \to \fkLSymp
\]
which is the identity on objects and $1$-morphisms.
\end{thm}
\begin{proof} 
By definition we have that $F(Y\circ X)=F(Y)\circ(X)$ and $F(I_S)=I_{F(S)}$ for any object $S$ and $1$-morphisms $X,Y$. Hence we take the $2$-isomorphisms $F_{g,f}$ and $F_S$, in Definition \ref{def:functor}, to be the identity.

Now consider a $2$-morphism $N\in (\fkSymp^{or}_c)_{2}(X_0, X_1)$, this is a proper oriented $(-1)$-Lagrangian $N \to (\bullet_{(-1)})^-\times X_{01}$. Applying Conjecture \ref{con:OurConj}(a)  and Lemma \ref{lem:XWdiag}(a) to this Lagrangian and shifting we obtain a map 
\[(\mu_{N})_{*}[-n_0-n_1]:  \mathbb{H}^{\bullet}(\mathcal{P}_{\bullet_{-1}}[\vdim N-n_0-n_1]) \to \mathbb{H}^{\bullet}( \mathcal{P}_{X_{01}}[-n_0-n_1]).
\]
Since $\mathbb{F}\cong\mathbb{H}^{\bullet}(\mathcal{P}_{\bullet_{-1}})$, we define 
\[F(N) = (\mu_{N})_{*}[-n_0-n_1](1) \in \fkLSymp(X_0,X_1).
\]
%Note that $F(N)$ has degree $n_0+n_1-\vdim N$. 
This well defined since, by Conjecture \ref{con:OurConj} (b) together with Lemma \ref{lem:XWdiag} (b), if $N'$ is Lagrangeomorphic to $N$ then $(\mu_{N})_{*}=(\mu_{N'})_{*}$.

We observe that, by definition we have 
\[F(1_X)=1_{F(X)}, \ F(W_{X_{3}X_{2}X_{1}}) = W_{F(X_{3})F(X_{2})F(X_{1})}, \ F(r_{X})=r_{F(X)} \ \textrm{and} \  F(l_{X})=l_{F(X)}.
\]
The only conditions left to check are the following
\[ F(N \odot M)=  F(N) \odot F(M), \ \ F(N * M)=  F(N) * F(M).
\]
Since both can be proved in the same way, we check only the first one. We denote $\epsilon = \vdim N (n_0+n_1)$ and compute
\begin{align*}
F(N \odot M) & = (\mu_{N \odot M})_*[-n_0-n_1](1)\\
   &= (\mu_{(-1)^\epsilon X_{012}\bullet(M\times N)})_*[-n_0-n_1](1)\\
   &= (-1)^\epsilon\big((\mu_{X_{012}})_*\circ ((\mu_{M\times N})_*[-n_0-n_1])\big)[-n_0-n_1](1)\\
   &= (-1)^\epsilon(\mu_{X_{012}})_*[-n_0-n_1]((\mu_{M\times N})_*[-2n_0-2n_1](1))\\
   &= (-1)^\epsilon(\mu_{X_{012}})_*[-n_0-n_1]\big((\mu_{N})_*[-n_0-n_1](1)\otimes(\mu_{M})_*[-n_0-n_1](1)\big)\\
   &= (-1)^\epsilon(\mu_{X_{012}})_*[-n_0-n_1](F(N) \otimes F(M))= F(N) \odot F(M),
\end{align*}
where the third equality follows from Conjecture \ref{con:OurConj}(c) combined with Lemma \ref{lem:XWdiag}(c), the fifth equality follows from Conjecture \ref{con:OurConj}(e) together with Lemma \ref{lem:XWdiag}(e) and the other follow from the definitions. Finally, in the last equality, we used the fact that $|F(N)|= n_0+n_1-\vdim N$.
\end{proof}

\subsection{A linearization of $\fkLag(S)$}\

In this subsection we will see that the constructions in Sections 5 and 6 can be adapted, almost with no change, to linearize the bicategory $\fkLag(S)$ whenever $S$ is a $1$-symplectic derived stack. The previous constructions will then be the special case when $S$ is a point.

For the rest of the section fix a $1$-symplectic derived stack $S$. We start by first introducing $\fkLag^{or}(S)$, the oriented version of $\fkLag(S)$. Its objects will be pairs consisting of a Lagrangian $f: X \to S$ in $S$ and a line bundle $E$ on $X$. A $1$-morphism in  $\fkLag^{or}(S)_1((X_0, E_0),(X_1,E_1))$ is defined to be a Lagrangian $\varphi: N \to X_{01}$ in the $0$-symplectic derived stack $X_{01}:=X_0 \times_S X_1$ together with a $(K_{X_0}\otimes E_0)^{-1} \boxtimes E_1$ orientation. One can easily prove the following  

\begin{lem}
For $i=0,1,2$, consider Lagrangians $f_i: X_i \to S$ equipped with line bundles $E_i$. If $\varphi_0: N_0 \to X_{01}$ is a $(K_{X_0}\otimes E_0)^{-1} \boxtimes E_1$-oriented Lagrangian and $\varphi_1: N_{1} \to X_{12}$ is a $(K_{X_1}\otimes E_1)^{-1} \boxtimes E_2$-oriented Lagrangian, then the Lagrangian $N_0 \times_{X_1} N_1 \to X_{02}$, constructed in Corollary \ref{cor:Comp}, has an induced $(K_{X_0}\otimes E_0)^{-1} \boxtimes E_2$ orientation.

The diagonal Lagrangian $\Delta_{X_0}\to X_{00}$ has a natural $(K_{X_0}\otimes E_0)^{-1} \boxtimes E_0$ orientation.
\end{lem} 

This lemma defines the composition of $1$-morphisms and also gives the identity $1$-morphisms in $\fkLag^{or}(S)$. 

The rest of the construction of $\fkLag^{or}(S)$ is the same as $\fkSymp^{or}$. Note that as in Lemma \ref{lem:OrienProds}, given $N_0$ and $N_1$ in $\fkLag^{or}(S)_1((X_0, E_0),(X_1,E_1))$, the $(-1)$-symplectic derived stack $N_{01}= N_0 \times_{X_{01}} N_1$ is naturally oriented. So we define $\fkLag^{or}(S)_2(N_0,N_1)$ as the set of oriented Lagrangeomorphism classes of oriented Lagrangians in $N_{01}$. Now we can proceed as in Section \ref{Ssec:Or} and obtain the following theorem which is analogous to Theorem \ref{thm:sympor}.

\begin{thm}\label{thm:Lagor}
Let $S$ be a $1$-symplectic derived stack. There exists a bicategory $\fkLag^{or}(S)$  enriched over $\textbf{gr-Inv}$. The objects are pairs consisting of a Lagrangian $f: X \to S$ and a line bundle $E$ on $X$.  The $1$-morphisms $\fkLag^{or}(S)_1((X_0, E_0),(X_1,E_1))$ are $(K_{X_0}\otimes E_0)^{-1} \boxtimes E_1$-oriented Lagrangians and the $2$-morphisms $\fkLag^{or}(S)_2(N_0, N_1)$ are oriented Lagrangeomorphism classes of oriented Lagrangians in $N_0 \times_{S} N_1$. 

Moreover there is a homomorphism $\fkLag^{or}(S) \to \fkLag(S)$ which forgets the orientation data. 
\end{thm}

Next we linearize $\fkLag^{or}(S)$ using Conjecture \ref{con:OurConj} and the results in Sections \ref{subsec:sheaves}, \ref{Ssec:PL} and \ref{subsec:linfunct}. Since the details are no different from the case of $\fkLSymp$ we do not repeat them and simply state the end result.

\begin{thm}\label{thm:LLag}
Let $S$ be a $1$-symplectic derived stack. There exists a bicategory $\fkLLag(S)$ enriched over graded vector spaces, whose objects and $1$-morphisms are the same as $\fkLag^{or}(S)$. The space of two morphisms $\fkLLag(S)_2(N_0, N_1)$ is the hypercohomology $\mathbb{H}^\bullet(\mathcal{P}_{N_{01}}[-\vdim N_0])$.

Consider $\fkLag^{or}_c(S)$, the subcategory of $\fkLag^{or}(S)$ where the $2$-morphisms are (equivalence classes of) proper $(-1)$-Lagrangians. Then there is a homomorphism
\[\fkLag^{or}_c(S) \To \fkLLag(S),
\]
which is the identity on objects and $1$-morphisms.
\end{thm}

\section{Categories of Fillings and mapping stacks}

One of the main results in \cite{PTVV} states that under certain conditions, the mapping stack $\bMap(X, S)$ is a symplectic derived stack if $S$ is also symplectic. The main condition is that the stack $X$ possess a $d$-orientation, rather informally this can be thought of as a volume form that allows us to ``integrate functions" on $X$. Calaque \cite{Cal} defined a relative version of orientation and proved that the functor $\bMap(-, S)$ sends relative orientations to Lagrangians.
In this section we will build a bicategory of derived stacks with relative orientations, in analogous but dual way to how we constructed $\fkLag$. Then we will show that under certain conditions, $\bMap(-, S)$ can be promoted to a homomorphism of bicategories. We end by showing that this functor gives extended topological field theories with values in $\fkSymp^m$.

\subsection{Categories of Fillings}\label{subsec:lsymp}\

In this section, contrary to the rest of the paper, we will not assume that our derived stacks are Artin or locally of finite presentation. Instead we will require that the derived stacks be $\mathcal{O}$-compact. 

A derived stack is $X$ is $\mathcal{O}$-compact according to \cite[Definition 2.1]{PTVV} when for any affine derived scheme $Spec(A)$ we have that $\mathcal{O}_{X\times Spec(A)}$ is a compact object of $D_{qcoh}(X\times Spec(A))$ and for any perfect complex $E$ on $X\times Spec(A)$, the $A$-dg-module $\mathbb{R}\underline{Hom}(\mathcal{O}_{X\times Spec(A)}, E)$ is perfect. 
For a derived stack $X$ and an affine derived scheme $Spec \ A$, we use $X_{A}$ to denote $X \times Spec \ A$.

\begin{lem}\label{lem:PushOcompact} 
Given a diagram 
\[W_{1} \stackrel{j_1}\longleftarrow X \stackrel{j_2}\longrightarrow W_{2}\] 
of $\mathcal{O}$-compact derived stacks, a homotopy pushout, in the category of derived stacks, is $\mathcal{O}$-compact. 
\end{lem}
\begin{proof}
%First note that the homotopy pushout of this diagram in the category of derived stacks or classical stacks agrees, because the inclusion functor has a R/L adjoint. 
Let us denote the homotopy pushout by $Y$. Then $Y_A$ is a homotopy pushout of 
\[(W_{1})_A \xleftarrow{j_{1,A}} X_A  \xrightarrow{j_{2,A}} (W_{2})_A.
\]
Consider the resulting canonical maps $i_{1,A}:(W_1)_A \to Y_A$, $i_{2,A} :(W_2)_A \to Y_A$, and $i_{A}: X_A \to Y_A$.  We can write the (stable $\infty$-) categories of quasi-coherent sheaves and perfect complexes on $Y_{A}$ as a homotopy limit of the corresponding categories on $(W_1)_A$, $(W_2)_A$, and $X_A$. This means that given an object on $Y_A$ it is determined by objects on  $(W_1)_A$, $(W_2)_A$, and $X_A$ related by the appropriate pullbacks. Since homotopy colimits and finite homotopy limits commute in the stable context, this correspondence is preserved by homotopy filtered colimits. In particular, working in the derived categories of quasi-coherent sheaves, for any $E \in D_{qcoh}(Y_A)$ the set $\Hom(\mathcal{O}_{Y_A}, E)$ is the limit of the diagram 
\[\Hom(\mathcal{O}_{(W_{1})_A}, E|_{(W_1)_A}) \xrightarrow{j^{*}_{1,A}} \Hom(\mathcal{O}_{X_A}, E|_{X_A}) \xleftarrow{j^{*}_{2,A}} \Hom(\mathcal{O}_{(W_{2})_A}, E|_{(W_2)_A}).\]
Notice that these pullbacks commute with homotopy filtered colimits in the $E$ variable,  that $(W_1)_A$, $(W_2)_A$, and $X_A$ are $\mathcal{O}$-compact and that finite limits and filtered colimits in the category of sets commute. Putting this all together, this diagram commutes with homotopy colimits in the $E$ variable and so $\mathcal{O}_{Y_A}$ is compact.
In a similar way, considering the functor $\mathbb{R}\underline{Hom}(\mathcal{O}_{Y_A},- ) $  we obtain the exact triangle 
\[ \mathbb{R}\underline{Hom}(\mathcal{O}_{Y_A}, E ) \to  \mathbb{R}\underline{Hom}(\mathcal{O}_{(W_1)_A}, E|_{(W_1)_A} )\oplus   \mathbb{R}\underline{Hom}(\mathcal{O}_{(W_2)_A},E|_{(W_2)_A}) \to   \mathbb{R}\underline{Hom}(\mathcal{O}_{X_A},E|_{X_A}).
\] 
Since the restrictions of $E$ are perfect, and because  $(W_1)_A$, $(W_2)_A$, and $X_A$ are $\mathcal{O}$-compact we can conclude that  $\mathbb{R}\underline{Hom}(\mathcal{O}_{(W_1)_A}, E|_{(W_1)_A} )$,    $\mathbb{R}\underline{Hom}(\mathcal{O}_{(W_2)_A},E|_{(W_2)_A})$ and $\mathbb{R}\underline{Hom}(\mathcal{O}_{X_A},E|_{X_A})$ are all perfect.  This implies that $\mathbb{R}\underline{Hom}(\mathcal{O}_{Y_A}, E )$ is perfect, which completes the proof.
\end{proof}

We now review the definition of orientation following \cite{PTVV}. From now on we use the notation $C(X, E)=\mathbb{R}\underline{Hom}(\mathcal{O}_{X}, E)$, for a complex $E$ on $X$.
Let $\eta:C(X, \mathcal{O}_{X})\to k[-d]$ be a morphism in the derived category $D(k)$, this defines, for any perfect complex $E$ on $X_{A}$, a morphism 
\begin{equation}\label{eqn:AcapEta}  
{(- \cap \eta)}_{A}:C(X_A, E) \to C(X_A, E^{\vee})^{\vee}[-d]
\end{equation}
corresponding to the composition 
\[C(X_A, E) \otimes C(X_A, E^{\vee}) \to C(X_A, E \otimes E^{\vee}) \to C(X_A, \mathcal{O}_{X_A})\cong C(X, \mathcal{O}_{X})\otimes A \to A[-d]
\]
where the first map is the cup product, the second is the trace and the last is $\eta\otimes \text{id}_{A}$. 

\begin{defn}
Let $X$ be a $\mathcal{O}$-compact derived stack, an $\mathcal{O}$-\emph{orientation} of degree $d$ (usually abbreviated to $d$-orientation) consists of a morphism $[X] : C(X, \mathcal{O}_{X})\to k[-d]$ such that for any $A \in cdga_{k}^{\leq 0}$ and any perfect complex $E$ on $X_{A}$, the morphism 
\begin{equation}\label{eqn:Acap} 
{( - \cap [X])}_{A}:C(X_A, E) \to C(X_A, E^{\vee})^{\vee}[-d]
\end{equation}
is a quasi-isomorphism of $A$-dg-modules.
\end{defn}
From now on, however, we will suppress this notation from our calculations, and just use the notation 
\[- \cap [X]:C(X, E) \to C(X, E^{\vee})^{\vee}[-d]
\]
for the entire family of morphisms in (\ref{eqn:Acap}) for all possible choices of $A$ and $E$.

\begin{rem} If $X$ is equipped with an orientation $[X]$ which is understood, we sometimes use $\overline{X}$ to denote $(X, -[X])$.
\end{rem}

We now recall the definitions of boundary structure and relative orientation (which we will call a filling) from \cite{Cal}.
Let $X$ be a $d$-oriented derived stack and $f:X \to W$ be a morphism of derived stacks. Denote by $f_{*}[X]$ be the composition 
\[C(W, \mathcal{O}_{W}) \to C(X, \mathcal{O}_{X}) \xrightarrow{[X]} k[-d],\]
where the first morphism is pullback. Note that we can rewrite $- \cap f_{*}[X]$ as the composition 
\begin{equation}\label{capfX}
C(W,E) \to  C(X,f^{*}E)  \to  C(X,f^{*}E^{\vee})^{\vee}[-d]  \to C(W,E^{\vee})^{\vee}[-d] 
\end{equation}
given by pullback, cap with $[X]$ and finally the shifted dual of pullback. 
\begin{defn}
Let $(X, [X])$ be a $\mathcal{O}$-compact derived stack with a $d$-orientation. A \emph{ boundary structure} \cite{Cal} on a morphism $f:X \to W$ is a path $\gamma$ from $f_{*}[X]$ to $0$ in the space $Map(C(W, \mathcal{O}_{W}), k[-d])$.
\end{defn}

Suppose we have a morphism $f:X \to Y$ of derived stacks and an object $E \in \Perf(Y)$. We define $C(f,E)$ by the exact triangle 
\[C(Y,E) \longrightarrow C(X, f^{*}E) \longrightarrow C(f,E) \longrightarrow.
\]
Notice that for a pair of morphisms of derived stacks $X \stackrel{f}\longrightarrow Y \stackrel{g}\longrightarrow Z$, $C(f,E)$ and $C(g,E)$ are related by the following exact triangle
\[C(g,E) \longrightarrow C(g\circ f, E) \longrightarrow C(f, g^{*}E) \longrightarrow,
\]
in the derived category $D(k)$.

A boundary structure $\gamma$ induces the following diagram
\begin{equation}\label{eqn:BdStr2}
\xymatrix{&C(W,E) \ar[ddl]_{\Theta_{\gamma}} \ar[d]\ar[ddr]^{0}& & &  \\& C(X,f^{*}E)   \ar[d]_{-\cap [X]} \\C(f, E^{\vee})^{\vee}[-d] \ar[r] & C(X,f^{*}E^{\vee})^{\vee}[-d] \ar[r]& C(W,E^{\vee})^{\vee}[-d] \ultwocell\omit.
}
\end{equation}
This is because $\gamma$ determines a homotopy between $0$ and the composition (\ref{capfX}) which, since the bottom row is exact, determines the lift $\Theta_{\gamma}$.
A boundary structure $\gamma$ is called {\it non-degenerate} if the associated morphism $\Theta_{\gamma}$ is a quasi-isomorphism. 

\begin{defn}
Consider a $\mathcal{O}$-compact derived stack $X$ with a $d$-orientation $[X]$.
If $\gamma$ is a non-degenerate boundary structure on $f:X\to W$, we call the pair $(f, \gamma)$ a \emph{filling} of $X$. Denote the set of fillings for a fixed morphism $f$ by $\mathcal{F}ill(f,[X])$. 
\end{defn}

%From now on in this subsection, we will assume that all the derived stacks are in fact classical stacks and all the morphism are closed immersions.
We will see that fillings have many of the same formal properties as Lagrangians, just dualized.
The following is the analogue of Example \ref{point} and Proposition \ref{product}.

\begin{prop}\label{prop:EasyFill}
Let $X$ and $W$ be $\mathcal{O}$-compact derived stacks and $f:X \To W$ a morphism of derived stacks. We have the following:
\begin{itemize}
\item[\textbf{(a)}] Consider the empty set $\empty$ as a $d$-oriented derived stack. A $d$-filling of the morphism $\emptyset \To X$ is equivalent to a $(d+1)$-orientation on $X$.  
\item[\textbf{(b)}] If $(X,[X])$ is $d$-oriented, there is a bijection between $\mathcal{F}ill(f,[X])$ and $\mathcal{F}ill(f,-[X])$.
\item[\textbf{(c)}] Let $X_{1}$ and $X_{2}$ be $d$-oriented derived stacks and suppose we have fillings $f_1: X_{1} \To W_{1}$ and $f_2: X_{2}\To W_{2}$. Then $X_{1} \coprod X_{2}$ has an induced $d$-orientation and the morphism $f_1\coprod f_2: X_{1} \coprod X_{2} \To W_{1} \coprod W_{2}$ is a filling.
\end{itemize}
\end{prop}
\begin{proof}  In order to prove (a), notice that a boundary structure on $i: \emptyset \to X$ is just a loop $\gamma$ at $0$ in $Map(C(X, \mathcal{O}_{X}),k[-d])$. This is the same as a point in $Map(C(X, \mathcal{O}_{X}),k[-(d+1)])$. The associated morphism 
\[- \cap [X]:C(X, E) \to C(X, E^{\vee})^{\vee}[-(d+1)]
\]
is equivalent to 
\[C(X,E) \stackrel{\Theta_{\gamma}}\To C(i, E^{\vee})^{\vee}[-d] \cong  C(X, E^{\vee})[-1]^{\vee}[-d] \cong C(X, E^{\vee})^{\vee}[-(d+1)]\]
and so each is non-degenerate if and only if the other is. Points (b) and (c) are obvious. 
\end{proof}

The following is an analogue of Proposition \ref{prop:Lag2Lag}
\begin{prop}\label{prop:BigBlittleB}
Suppose that $X_0$ and $X_1$ are $d$-oriented derived stacks and we are given a filling $f= (f_0, f_1): \overline{X_0} \coprod X_1 \to W$. For a morphism of derived stacks $g:X_0 \to U$, consider the associated morphism $b_{f}(g): X_{1} \to  U\coprod_{X_0} W$.
Then there is a map 
\[B_{f}: \mathcal{F}ill(g, [X_0]) \to \mathcal{F}ill(b_{f}(g), [X_1]).
\]
\end{prop}
\begin{proof}
First note that Lemma \ref{lem:PushOcompact} guarantees that $U\coprod_{X_0} W$ is $\mathcal{O}$-compact. Let us denote by $\gamma$ the boundary structure for the map $f$, that is a path from $-f_{0*}[X_0]+ f_{1*}[X_1]$ to $0$ in the mapping space $Map((C(W, \mathcal{O}_{W}), k[-d])$, or equivalently, a path $\gamma$ from $f_{1*}[X_1]$ to $f_{0*}[X_0]$.  Let $\delta$ be a filling of $g$, that is a path from $g_{*}[X_0]$ to $0$. 
Consider the homotopy commutative diagram 
\begin{equation}
\xymatrix{ & U \coprod_{X_0} W && \\
U \ar[ur]^{i_U} && W  \ar[ul]_{i_W}& \\
& & X_0 \coprod X_1 \ar[u]_{f} \\
& X_0 \ar[uur]^{f_0} \ar[uul]_{g} \ar[ur]_{j_0} && X_1 \ar[ul]^{j_1} \ar[uul]_{f_1}
}
\end{equation}
The fact that the square on the left is homotopy commutative determines a path $c$ from $i_{W*}f_{0*}[X_0]$ to $i_{U*}g_{*}[X_0]$. Consider the concatenation \[B_{f}(\delta)= (i_{W*}\gamma)\bullet c \bullet (i_{U*}\delta) .\] It is a path from $(i_{W}f_{1})_{*}[X_1]$ to $0$. Since $i_{W}f_{1}= b_{f}(g)$ we have produced a boundary structure on $b_{f}(g)$. We will now prove that it is non-degenerate if $\delta$ is non-degenerate. 
Let $T=U \coprod_{X_1} W$. 
Given $F \in \Perf(U \coprod_{X_1} W)$ we have an exact triangle
\[F \to i_{U*}i_{U}^{*}F \oplus i_{W*}i_{W}^{*}F \to (i_W \circ f_0)_{*} (i_W \circ f_0)^{*}F \to.
\]
applying this to $E$ and $E^{\vee}$ and taking derived global sections over $T$ we get a commutative diagram 
\begin{equation}\label{eqn:boundaryND}
\xymatrix{
C(T,E) \ar[d] \ar[r]  &C(U, i^{*}_{U}E) \oplus C(W, i_{W}^{*}E)\ar[r] \ar[d]^{(\Theta_{\delta}, \Theta_{\gamma})}&  C(X_{1}, f_{0}^{*}i_{W}^{*} E)  \ar[d]^{(-)\cap [X_1]}    \\
C(i_{W} \circ f_{1},E^{\vee})^{\vee}[-d] \ar[r]  &C(g, i_{U}^{*}E^{\vee})^{\vee}[-d] \oplus C(f, i_{W}^{*} E^{\vee})^{\vee}[-d] \ar[r] &  C(X_{1}, f_{0}^{*}i_{W}^{*} E^{\vee})^{\vee}[-d]  
}
\end{equation}
Consider 
\[X_{1} \stackrel{j_1}\longrightarrow X_{0} \coprod X_1 \cong X_{0} \coprod_{X_0} (X_{0} \coprod X_{1}) \xrightarrow{(g,f)} U \coprod_{X_0} W,
\]
because $(g,f)\circ j_2 \cong i_{2} \circ f_{1}$ we get an exact triangle 
\[C(j_2, (g,f)^{*}E^{\vee})^{\vee} \longrightarrow 
C(i_2 \circ f_1 , E^{\vee})^{\vee} \longrightarrow C((g,f),E^{\vee})^{\vee} \longrightarrow.
\]
We also have 
\[C(X_1,j_{1}^{*}F)\oplus C(X_2, j_{2}^{*}F) = C(X_1 \coprod X_2, F) \longrightarrow C(X_2, j_{2}^{*}F) \longrightarrow C(j_2, F) \longrightarrow,
\]
and therefore, $C(j_2, F) \cong C(X_1, j_{1}^{*}F)[+1]$. And so we conclude
\begin{equation}\label{eqn:eqni}
C(X_1, j_{1}^{*}F)^{\vee}[-1] \cong C(j_2, F)^{\vee}
\end{equation}
and 
\begin{equation}\label{eqn:eqnii}
C((g,f), E^{\vee}) \cong C(g,i_{U}^{*} E^{\vee}) \oplus C(f, i_{W}^{*}E^{\vee}).
\end{equation}
Putting these together we obtain the exact triangle 
\[C(X_0, f_{0}^{*} i_{W}^{*} E^{\vee})^{\vee}[-1] \longrightarrow C(i_{W} \circ f_{0}, E^{\vee})^{\vee} \longrightarrow C(g, i^{*}_U E^{\vee})^{\vee} \oplus C(f, i^{*}_W E^{\vee})^{\vee} \longrightarrow.
\]
Shifting and rotating it we have the exact triangle 
\[C(i_{W} \circ f_{2}, E^{\vee})^{\vee}[-d] \longrightarrow C(g, i^{*}_U E^{\vee})^{\vee} [-d]\oplus C(f, i^{*}_W E^{\vee})^{\vee}[-d] \longrightarrow C(X_1, f_{0}^{*} i_{W}^{*} E^{\vee})^{\vee}[-d] \longrightarrow.
\]
Therefore, the bottom row in (\ref{eqn:boundaryND}) is an exact triangle. The top row is an exact triangle as well. Therefore the left vertical arrow in (\ref{eqn:boundaryND}) is an equivalence in the homotopy category. This agrees with the morphism $\Theta_{B_{f} (\delta)}$ which is induced by the path $B_{f}(\delta)$ by the mechanism explained in (\ref{eqn:BdStr2}).
\end{proof}

The following lemma is the analogue of Proposition \ref{prop:Diag} and Corollary \ref{cor:Const}.
\begin{lem}\label{lem:NatOrien}
Suppose that $X$ is a $d$-oriented derived stack. The natural morphism $ \nabla:\overline{X} \coprod X\to X$ has a canonical filling.  Moreover, given fillings $f_1:X \to W_1$ and  $f_2:X \to W_2$ of $X$ then the (homotopy) pushout $W_1 \coprod_{X} W_2$ has an induced $(d+1)$-orientation. 
\end{lem}
\begin{proof} 
The pushforward of $[\overline{X}]$ to $X$ is of course $-[X]$ and so we can use that to find a path to zero from the pushforward of $[X \coprod \overline{X}]$ to $X$. It induces a morphism $C(X, E) \to C( \nabla, E^{\vee})^{\vee}[-d]$ which we want to show is a quasi-isomorphism. If we consider the inclusion of the first factor $X \stackrel{j}\to  \overline{X}\coprod X$, then $\nabla \circ j$ is the identity and the exact triangle 
\[C(X , E^{\vee}) \to C( \overline{X}\coprod X, \nabla^{*}E^{\vee}) \to C(X , E^{\vee})
\]
shows that $C( \nabla, E^{\vee}) \cong C(X, E^{\vee})$ and in fact $C(X, E) \to C( \nabla, E^{\vee})^{\vee}[-d]$ actually agrees with the original morphism $(-)\cap [X]:C(X,E) \to C(X,E^{\vee})^{\vee}[d]$ itself so it is a quasi-isomorphism. 

The second statement is an immediate corollary of the first and Proposition \ref{prop:BigBlittleB}. Indeed by Proposition \ref{prop:EasyFill} (3) we have a filling $\overline{X} \coprod X \to W_1 \coprod W_2$. By the first statement, we have the filling $\nabla : (\overline{X} \coprod  X) \coprod \emptyset\to X $. By applying Proposition \ref{prop:BigBlittleB} we see that $\emptyset \to (W_1 \coprod W_2)\coprod_{\overline{X} \coprod X} X$ is a $d$-filling and so Proposition \ref{prop:EasyFill} (1) corresponds to a $(d+1)$-orientation on $(W_1 \coprod W_2)\coprod_{\overline{X} \coprod X} X \cong W_1 \coprod_{X} W_2$.
\end{proof}

The following proposition is an analogue of Theorem \ref{thm:ThreeLag}.
\begin{prop}\label{lem:ThreeFill} 
Let $X$ be a $d$-oriented derived stack. Suppose that we are given three fillings $X \to W_{i}$ for $i=0,1,2$. The natural morphism  
\[\phi: (W_{0}\coprod_{X} W_{1}) \coprod (W_{1}\coprod_{X} W_{2}) \coprod (W_{2}\coprod_{X} W_{0}) \to W_{0} \coprod_{X} W_{1} \coprod_{X} W_{2}
\]
has a canonical  filling.
\end{prop}
\begin{proof}
The derived stack $W_{0} \coprod_{X} W_{1} \coprod_{X} W_{2}$ is $\mathcal{O}$-compact by Lemma \ref{lem:PushOcompact}. The construction of the natural boundary structure is analogous to the construction of the isotropic structure in Theorem \ref{thm:ThreeLag} and Proposition \ref{prop:Characterization} and so is omitted. We will prove that this boundary structure is non-degenerate.

Denote by $W_{ij} = W_{i} \coprod_{X} W_{j}$ the $(d+1)$-oriented derived stacks constructed in Lemma \ref{lem:NatOrien}. Let $T=W_{0} \coprod_{X} W_{1} \coprod_{X} W_{2}$ and denote by $f_{ij}: W_{ij}\to T$ the induced maps. Notice that $T \cong W_{01} \coprod_{ W_{1}} W_{12}$. Consider $E \in \Perf(T)$, there is an exact triangle
\[C(T,E) \longrightarrow C(W_{01}, f_{01}^{*}E) \oplus C(W_{12}, f_{12}^{*}E) \longrightarrow C(W_{1},f_{1}^{*}E) \longrightarrow.
\]
Denote by $q$ the composition 
\[W_{20} \stackrel{i}\longrightarrow W_{01} \coprod W_{12} \coprod W_{02} \stackrel{\phi}\longrightarrow T
.\]
This gives an exact triangle 
\[C(i, \phi^{*}E)^{\vee} \longrightarrow C(q,E)^{\vee}  \longrightarrow C(\phi,E)^{\vee}   \longrightarrow.
\]
Notice also that $C(i, \phi^{*}E) \cong C(W_{01} \coprod W_{12} , \phi^{*}E)[1]$ and we have a co-Cartesian square 
\begin{equation}
\xymatrix{X \ar[d]_{f_1} \ar[r] & W_{20} \ar[d]^{q} \\
W_{1} \ar[r]_{\pi} & T
}
\end{equation}
and therefore $C(q,E) \cong C(f_1, \pi^{*}T)$ for all $E \in \Perf(T)$. Combining the above we get an exact triangle 
\[C(W_{01},f_{01}^{*}E)^{\vee}[-1] \oplus C(W_{12}, f_{12}^{*}E)^{\vee}[-1] \longrightarrow C(f_{1}, E)^{\vee} \longrightarrow C(\phi, E)^{\vee} \longrightarrow 
\]
and by shifting and rotating, an exact triangle 
\[C(\phi, E)^{\vee}[-d-1] \longrightarrow C(W_{01},f_{01}^{*}E)^{\vee}[-d-1] \oplus C(W_{12}, f_{12}^{*}E)^{\vee}[-d-1] \longrightarrow C(f_{1}, E)^{\vee}[-d].
\]
In conclusion, we get a diagram with exact triangles as rows
\begin{equation*}
\resizebox{1 \textwidth}{!}{
\xymatrix{
C(T,E) \ar@{.>}[d]^{\Theta_{\gamma_{012}}} \ar[r]  &C(W_{01}, f_{01}^{*}E) \oplus C(W_{12}, f_{12}^{*}E)\ar[r] \ar[d]^{((-)\cap[W_{01}], (-)\cap[W_{12}])}&  C(W_{1}, f_{1}^{*}E) \ar[r] \ar[d]^{\Theta_{\gamma_1}} &   \\
C(\phi,E^{\vee})^{\vee}[-d-1] \ar[r]  &C(W_{01}, f_{01}^{*}E^{\vee})^{\vee}[-d-1] \oplus C(W_{12},  f_{12}^{*}E^{\vee})^{\vee}[-d-1] \ar[r] &  C(f_1,  E^{\vee})^{\vee}[-d] \ar[r]& 
}}
\end{equation*}
Because the middle and final vertical arrow are quasi-isomorphisms, the first is as well.
\end{proof}

\begin{defn}\label{defn:Filleo}
Let $X$ be a $d$-oriented derived stack and let $i_0:X\to W_0$ and $i_1:X \to W_1$ be fillings of $X$. A \emph{filleomorphism} between $W_0$ and $W_1$ is a triple consisting of an equivalence of derived stacks $g: W_{0} \to W_{1}$, a homotopy between $g \circ i_0$ and $i_1$ and a filling of the induced morphism 
\[g \coprod_{X} \text{id}_{W_1}: W_{0} \coprod_{X} W_{1} \longrightarrow W_{1}.
\]
\end{defn}

Using Proposition \ref{prop:EasyFill}, Proposition \ref{prop:BigBlittleB} and Proposition \ref{lem:ThreeFill} we can redo the entirety of sections 2, 3, 4 of this article in this ``dual" picture where symplectic structures are replaced with $\mathcal{O}$-orientations, Lagrangians are replaced by fillings,  Lagrangeomorphisms are replaced by filleomorphisms and all the morphisms go in the opposite direction. For example composition of $1$-morphisms and vertical composition of $2$-morphisms are defined using the morphisms
\[\mathcal{F}ill(W_1 \coprod_{X} W_2) \times \mathcal{F}ill(W_2 \coprod_{X} W_3) \to \mathcal{F}ill(W_1 \coprod_{X} W_3),
\]
constructed by combining Proposition\ref{prop:EasyFill} and Proposition \ref{lem:ThreeFill}, as in 
Corollary \ref{cor:Comp}.

We spare the reader the details and summarize the result in the following:
% letting $W_{ij} = W_{i}\coprod_{X} W_{j}$
%\[\mathcal{F}ill(P_{1}\coprod_{W_{12}} P_{2}) \times \mathcal{F}ill(Q_{1}\coprod_{W_{23}} Q_{2}) \to %\mathcal{F}ill(R_{1}\coprod_{W_{13}} R_{2})
%\]
%where $R_{i} = P_{i}\coprod_{W_2} Q_{i}$. 

\begin{thm}
Let $(X, [X])$ be a $d$-oriented stack. There exists a bicategory $\fkFill(X, [X])$ whose objects are fillings $(f:X \to W, \gamma)$, $1$-morphisms between two fillings $(f_1:X \to W_1, \gamma_1)$ and  $(f_2:X \to W_2, \gamma_2)$ are the fillings of $W_1 \coprod_{X} W_2$, equipped with the orientation defined in Lemma \ref{lem:NatOrien}. The $2$-morphisms between two such fillings $(W_1 \coprod_{X} W_2 \to Q_1, \tau_1)$ and $(W_1 \coprod_{X} W_2 \to Q_2, \tau_2)$ are fillings of $Q_1 \coprod_{(W_1 \coprod_{X} W_2)} Q_2$ up to filleomorphism.
\end{thm}

\begin{defn}\label{defn:Ord}
In the special case of the $(d-1)$-oriented derived stack $X=\emptyset$, this theorem constructs a bicategory $\fkFill(\emptyset)_{d-1}$, which we denote by $\fkOr^d$, whose objects are $d$-oriented derived stacks. Analogous to the symplectic case, it has a symmetric monoidal structure.
\end{defn}
\begin{thm}
The bicategory $\fkOr^d$ is a symmetric monoidal bicategory. The monoidal structure 
\[
\fkOr^d \times \fkOr^d \to \fkOr^d,
\]
at the level of objects, sends $((X_1, [X_1]), (X_2, [X_2]))$ to $(X_1 \coprod X_2, [X_1 \coprod X_2] )$ and has the empty set $\emptyset$ as the unit.
\end{thm} 
\begin{proof}
We define the monoidal structure on morphisms by the coproduct of fillings, as defined in Proposition \ref{prop:EasyFill}\/(c). Together with some natural isomorphisms which we do not write down, this defines a symmetric monoidal bicategory.
\end{proof}

\subsection{From fillings to Lagrangians}\

Let $X$ be a $d$-oriented derived stack and $S$ be a $n$-symplectic derived stack. Consider a subcategory of $\fkFill(X, [X])$, such that the restriction of the functor $\bMap(- ,S)$ has image in the category of derived Artin stacks. We will see that the functor $\bMap(- ,S)$ defines a homomorphism of bicategories to $\fkLag(\bMap(X,S))$. The material here is a modest elaboration on the ideas in \cite{Cal}. We start by reviewing how an orientation determines an ``integral".
Recall from \cite{PTVV} that for if $X$ is $\mathcal{O}$-compact, then for any derived stack $Z$ there is a natural map
\[\mathbf{DR}(X\times Z) \To C(X,\mathcal{O}_X)\otimes \mathbf{DR}(Z).
\]
If we are given a morphism $ \eta:C(X, \mathcal{O}_X) \To k[-d]$, we can compose it with the previous map and get a map
\[\mathbf{DR}(X\times Z) \To \mathbf{DR}(Z)[-d],
\]
which, in particular, induces a map 
\[\int_{\eta}: \mathcal{A}^{2, cl}(X\times Z,n) \To \mathcal{A}^{2,cl}(Z,n-d).
\]
We collect a few useful properties of this construction.

\begin{lem}\label{lem:MsPullPush}
The assignment $\eta \mapsto \int_{\eta}(-)$, determines a continuous map
\[Map(C(X, \mathcal{O}_X), k[-d]) \To Map (\mathcal{A}^{2, cl}(X\times Z,n), \mathcal{A}^{2,cl}(Z,n-d)).
\]

Let $f: X \to Y$ and $g:Z_0 \to Z_1$ be morphisms of $\mathcal{O}$-compact derived stacks. If $[X]: C(X, \mathcal{O}_X) \to k[-d]$, the following holds:
\[\int_{f_{*}[X]} (-) = \int_{[X]} (f \times \text{id})^{*}(-) \ \ \ \textrm{and} \ \ \
g^{*}\big(\int_{[X]} (-) \big)= \int_{[X]}  (\text{id}\times g)^{*}(-) .
\]
\end{lem}

\begin{thm}[\cite{PTVV}, Theorem 2.6]
Let $(X,[X])$ be a $d$-oriented derived stack, $(S,\omega)$ be a $n$-symplectic derived stack and assume that $\bMap(X, S)$ is a derived Artin stack. Denote by $ev: X\times\bMap(X, S)\to S $ the evaluation map. Then $\int_{[X]}ev^{*}\omega$ is an $(n-d)$-shifted symplectic structure on $\bMap(X, S)$
\end{thm}

\begin{thm}[\cite{Cal}, Theorem 2.11]\label{lem:Fill2Lag} 
Let $f:X \To W$ be a filling and assume that $\bMap(W, S)$ and $\bMap(X, S)$ are derived Artin stacks. The induced morphism $\bMap(f, S): \bMap(W, S) \to \bMap(X, S)$ has an induced Lagrangian structure. 
\end{thm}
\begin{proof}
We will explain how a boundary structure in $f$ determines an isotropic structure on $\bMap(f, S)$ and refer the reader to \cite{Cal} for a proof that this assignment preserves non-degeneracy.

For simplicity of notation, we denote $\mathcal{M}_{S}(X)=\bMap(X,S)$ and $\mathcal{M}_{S}(W)= \bMap(W,S)$ and $\mathcal{M}_{S}(f)$ for the morphism $\bMap(W,S)\to \bMap(X,S)$ the morphism induced by $f:X\to W$. A boundary structure on $f:X \to W$ consists of a path from $f_{*}[X]$ to $0$. By the first part of Lemma \ref{lem:MsPullPush}, this induces a path from $\int_{f_{*}[X]} \pi^{*} \omega$ to $0$.
% Consider the map $(f\times id):X \times \bMap(X, S) 
%\to W \times \bMap(X, S)$
Again using Lemma \ref{lem:MsPullPush} we have
\[\mathcal{M}_{S}(f)^{*}\int_{[X]}\pi^{*}\omega = \int_{[X]}(\text{id}\times \mathcal{M}_{S}(f))^{*}\pi^{*}\omega = \int_{f_{*}[X]}\pi^{*}\omega
\]
so we have a path from $\mathcal{M}_{S}(f)^{*}\int_{[X]}\pi^{*}\omega$ to $0$, in other words an isotropic structure on $\mathcal{M}_{S}(f)$.
\end{proof}

This proof points the way to some helpful notation. 
If $f:X \to W$ has a boundary structure, that is a path $\gamma_{f}$ from $0$ to $f_{*}[X]$ then we define $\mathcal{M}_{S}(\gamma_{f})$ to be the corresponding path in the space of closed $2$-forms on $\bMap(W,S)$ from $\mathcal{M}_{S}(f)^{*}\int_{[X]}\pi^{*}\omega$ to $0$.

\begin{prop}\label{prop:filltolag}
Let $X_0$, $X_1$ be $d$-oriented derived stacks, let $f= (f_0, f_1): \overline{X_0} \coprod X_1 \to W$ and $g:X_0 \to U$ be fillings and consider the filling $b_{f}(g): X_{1} \to U  \coprod_{X_0} W$, constructed in Proposition \ref{prop:BigBlittleB}. Assuming that the following mapping stacks are Artin, then an equivalence (determined by the universal property)  
\[\bMap(U \coprod_{X_0} W, S) \cong \bMap(U, S) \times_{\bMap(X_0, S)} \bMap(W,S)
\]
can be upgraded to a Lagrangeomorphism of Lagrangians in $\bMap(X_1, S)$. Here the Lagrangian structure on the right side is constructed by applying Proposition \ref{prop:Lag2Lag} 
%to the Lagrangians $\bMap(W, S) \to \bMap(X_0, S)^{-}\times \bMap(X_1, S)$ and $\bMap(U, S) \to \bMap(X_0,S)$. 
and the Lagrangian structure on the left hand side comes from applying Theorem \ref{lem:Fill2Lag} to $b_{f}(g)$.
\end{prop}
\begin{proof}
Let $\gamma$ be the path from $f_{0*}[X_0]$ to $f_{1*}[X_1]$. It gives rise to a path from $\mathcal{M}_{S}(f_1)^{*}\int_{[X_1]}\pi_1^{*}\omega$ to $\mathcal{M}_{S}(f_0)^{*}\int_{[X_0]}\pi_0^{*}\omega$ in the space of closed $2$-forms on $\bMap(W,S)$. Similarly, we have $\mathcal{M}_{S}(\gamma_g)$, a path from $\mathcal{M}_{S}(g)^{*}\int_{[X_0]}\pi_g^{*}\omega$ to $0$ in the space of closed $2$-forms on $\bMap(U,S)$. The canonical path connecting the pullbacks of $\mathcal{M}_{S}(f_0)^{*}\int_{[X_0]}\pi_0^{*}\omega$ and $\mathcal{M}_{S}(g)^{*}\int_{[X_0]}\pi_g^{*}\omega$ in the space of forms on the right hand side is induced using from the canonical path from the two different pushforwards of $[X_0]$ to the space $Map(C(U \coprod_{X_0} W),k[-d])$. Therefore $\mathcal{M}_{S}((i_{W*}\gamma) \bullet c \bullet (i_{U*}\gamma_g))$ is homotopy equivalent to the path made by connecting the pullbacks of $\mathcal{M}_{S}(i_{W})^{*}(\mathcal{M}_{S}(\gamma))$ and  $\mathcal{M}_{S}(i_{U})^{*}(\mathcal{M}_{S}(\gamma_g))$. We can now complete the proof by appealing to Corollary \ref{cor:HIsomLag}.
 \end{proof}
 
By taking $X_1$ to be a point, we obtain the following corollary, which can be found in \cite{Cal}
\begin{cor}\label{lem:Bd2Lag} 
Given two fillings $X\to W_1$ and $X\to W_2$, the equivalence of derived stacks  
\[ \bMap(W_{1} \coprod_{X} W_{2}, S) \to \bMap(W_1,S) \times_{\bMap(X,S)} \bMap(W_2,S)
\]
is a symplectomorphism, assuming these are derived Artin stacks.
\end{cor}

We now have all the ingredients necessary to show that $\bMap(-, S)$ defines a homomorphism from $\fkFill(X)$ to $\fkLag(\bMap(X,S)$, modulo the question of the required mapping stacks being derived Artin stacks. We fix this problem by restricting the domain of the homomorphism.

\begin{defn}\label{defn:Cdef}
Let $S$ be a derived Artin stack. Fix a subcategory $\fkC$ of the category of $\mathcal{O}$-compact derived stacks, closed under pushouts, and containing the initial object $\emptyset$ such that for any $X$ in $\fkC$, the mapping stack $\bMap(X, S)$ is a derived Artin stack.

Let $(X,[X])$ be a $d$-oriented derived stack, such that $X$ is an object of $\fkC$. We define the bicategory $\fkFill_{\fkC}(X)$ as the subcategory of $\fkFill(X)$, where all the fillings are objects and morphisms in $\fkC$. Note this defines a subcategory since $\fkC$ is closed under pushouts. Let $\fkOr_{\fkC}^d=\fkFill_{\fkC}(\emptyset_{d})$. As with $\fkOr$, this is a symmetric monoidal bicategory.
\end{defn}

We are aware of two examples of categories $\fkC$ which fulfil the conditions of Definition \ref{defn:Cdef}. It would be interesting to identify other examples.

\begin{example}\label{exam:Const}
Let $S$ be an arbitrary derived Artin stack. We can take $\fkC$ to be the category of ``constant stacks", that is to say, stacks whose value on any cdga is the same topological space (which has the homotopy type of a finite CW complex) and whose value on any morphism is the identity. The homotopy pushout of a diagram of such constant stacks is just the constant stack with value the homotopy pushout of the corresponding topological spaces. Moreover, as explained in \cite{PTVV}, for any such stack $X$, $\bMap(X, S)$ is a derived Artin stack. We discuss this example in more detail in subsection \ref{Ssec:ETFT}.
\end{example}

\begin{example}
Assume that $S$ is a smooth quasi-projective variety, or a classifying stack $BG$. Take $\fkC$ to be the category whose objects are finite homotopy colimits (in the category of derived stacks) of diagrams of smooth proper Deligne-Mumford stacks with morphisms closed immersions. As explained in \cite{PTVV}, if $X$ is a smooth proper Deligne-Mumford stack then $\bMap(X, S)$ is a derived Artin stack. Therefore for any object $Y$ in $\fkC$, then $\bMap(Y, S)$ is a derived Artin stack since it is a finite homotopy limit of derived Artin stacks.
\end{example}

\begin{thm}
Let $S$ be an $n$-symplectic derived stack and $\fkC$ a category as in Definition \ref{defn:Cdef}. If $(X, [X])$ is a $d$-oriented derived stack belonging to $\fkC$ then there is a homomorphism of bicategories 
\[\mathcal{M}_{S}:\fkFill_{\fkC}(X)\To \fkLag(\bMap(X,S))
\]
where $\bMap(X,S)$ is equipped with $(n-d)$-shifted symplectic structure $\int_{[X]}\pi^{*}\omega_{S}$ discussed above.
\end{thm}
\begin{proof}
The definition of this homomorphism on objects was explained in the proof of Lemma \ref{lem:Fill2Lag}
% \ref{lem:Bd2Lag}
. Recall that the $1$-morphisms and $2$-morphisms in the category $\fkFill_{\fkC}(X)$ are given by fillings, and similarly the $1$-morphisms and $2$-morphisms in the category $\fkLag(\bMap(X,S))$ are given by Lagrangians. The definition of this homomorphism on $1$-morphisms and $2$-morphisms looks the same as the definition on objects, it is again given by taking the mapping stack into $S$ and using the Lagrangian structure we have explained. The main thing to check is the compatibility of this assignment with the composition of $1$-morphisms and with the composition of $2$-morphisms. This works because compositions in $\fkFill_{\fkC}(X)$ are given by pushouts and compositions in $\fkLag(\bMap(X,S))$ are given by fiber products. Therefore, checking that this is a homomorphism boils down to a repeated use of Proposition \ref{prop:filltolag} and Corollary \ref{lem:Bd2Lag}, for which it is shown that the relevant structures on the pushouts correspond to the correct structures on fiber products. 
\end{proof}

As a special case of this, taking $X=\emptyset$ we have the following 

\begin{cor}\label{cor:functor_ms}
Let $S$ be an $n$-symplectic derived stack $\fkC$ a category satisfied the conditions of Definition \ref{defn:Cdef}. There is a homomorphism of symmetric monoidal bicategories
\[\mathcal{M}_{S}:\fkOr^d_{\fkC} \To \fkSymp^{n-d}
\]
which at the level of object sends a $d$-oriented derived stack $X$ to the $(n-d)$-symplectic derived Artin stack $\bMap(X,S)$.
\end{cor}

The claim that this respects the monoidal structure is an easy consequence of the fact that $\bMap(-,S)$ sends coproducts to products.
  
\subsection{Extended Topological Field Theories} \label{Ssec:ETFT} \

We now sketch the existence of (symmetric monoidal) homomorphism of bicategories from the oriented  cobordism category of $d$-dimensional manifolds to the category $\fkOr_{\fkC}^d$ discussed in \ref{defn:Cdef} where $\fkC$ is the category of constant stacks from Example \ref{exam:Const}. First we explain what we mean by the oriented cobordism bicategory $\fkCob^{or}_d$ following the constructions in \cite{S-P}. Our indexing is shifted from that in  \cite{S-P}, what we are describing here is the oriented version of what in his notation is $\fkCob_{d+2}$ 

\begin{defn}
The objects in $\fkCob^{or}_d$ are pairs $(M, \mu)$ where $M$ is a closed $d$-dimensional manifold and $\mu$ is an orientation on $M$. The $1$-morphisms in $\fkCob^{or}_d$ between two objects $(M_0, \mu_0)$ and $(M_1, \mu_1)$ are triples $(N, \nu, f)$ where $N$ is a compact $(d+1)$-dimensional manifold with boundary, $\nu$ is an orientation on $N$ and $f$ is an isomorphism of oriented manifolds $f: (\partial N, \nu|_{N}) \to (M_0, -\mu_0) \coprod  (M_1, \mu_1) $. The $2$-morphisms between triples $(N_0, \nu_0, f_0)$ and $(N_1, \nu_1, f_1)$ are equivalence classes of quadruples $(P, \rho, a, b)$ where $P$ is a compact $(d+2)$-dimensional manifold (with corners) with orientation $\rho$ and boundary components $(\partial_0 P, \rho|_{\partial_0 P})$ and $(\partial_1 P, \rho|_{\partial_1 P})$,  \[a:(\partial_0 P, \rho|_{\partial_0 P}) \to (N_0, -\nu_0) \coprod  (N_1, \nu_1)\] is an orientation preserving diffeomorphism and 
\[b: (\partial_1 P, \rho|_{\partial_1 P}) \to (M_0 \times I, -\mu_0) \coprod (M_1 \times I, \mu_1)
\]
is an orientation preserving diffeomorphism such that \[b \circ a^{-1}|_{\partial N_0}: \partial N_{0} \to (M_{0} \times \{0\}) \coprod (M_{1} \times \{0\} )\] and \[b \circ a^{-1}|_{\partial N_1}: \partial N_{1} \to (M_{0} \times \{1\}) \coprod (M_{1} \times \{1\})\] agree with $f_0$ and $f_1$. Two such quadruples $(P_0, \rho_0, a_0, b_0)$ and $(P_1, \rho_1, a_1, b_1)$ are equivalent if there is an orientation preserving diffeomorphism $(P_0, \rho_0)\to (P_1, \rho_1)$ which takes $a_0^{-1}(N_0)$ to  $a_1^{-1}(N_0)$ and $a_0^{-1}(N_1)$ to  $a_1^{-1}(N_1)$,  $b_0^{-1}(M_0 \times I)$ to  $b_1^{-1}(M_0 \times I)$ and $b_0^{-1}(M_1 \times I)$ to  $b_1^{-1}(M_1 \times I)$.
\end{defn}

This is a symmetric monoidal bicategory where the monoidal structure is given by disjoint union (along with the induced orientation). We end this section by explaining that the Betti stack construction gives a symmetric monoidal homomorphism
\[\fkCob^{or}_d \To \fkOr^d_{\fkC}
\]
where $\fkC$ is the category of constant stacks discussed in Example \ref{exam:Const}.  The interest in such a homomorphism is that given an $n$-symplectic derived Artin stack $S$, we can postcompose such a homomorphism with the homomorphism $\mathcal{M}_{S}$ described in Corollary \ref{cor:functor_ms}
to get a symmetric monoidal homomorphism 
 \[\mathcal{Z}_S: \fkCob^{or}_d \To \fkSymp^{n-d}
\]
given by $\mathcal{Z}_S(-)= \mathcal{M}_S((-)_{B})$. 

\begin{prop}
Let $\fkC$ be the category of constant stacks discussed in Example \ref{exam:Const}.  Then the Betti stack construction gives a symmetric monoidal homomorphism of bicategories 
\[\fkCob^{or}_d \To \fkOr^d_{\fkC}.
\]
\end{prop}
\begin{proof}
To each compact oriented manifold with corners $(T, \sigma)$ we assign to the (constant) Betti stack $T_{B}$ (which is $\mathcal{O}$-compact). In the case that $S$ has no boundary, it has a $\mathcal{O}$-orientation $[T_{B}]= \sigma_{B}$. We will sometimes use the notation $(T,\sigma)_{B} = (T_{B}, \sigma_{B})$. The functor $(-)_{B}$ preserves finite homotopy pushouts and homotopy equivalences.  If $(M_0, \mu_0)$ and  $(M_1, \mu_1)$ are closed oriented manifolds such that $M_0 \coprod M_1$, is the boundary of a manifold with boundary $(N,\eta)$, then the inclusion $M_0 \coprod M_1 \to N$ induces a filling $(M_0, -\mu_0)_{B} \coprod (M_{1}, \mu_1)_{B}\to (N,\eta)_{B}$ of $(M_0, -\mu_0)_{B} \coprod (M_{1}, \mu_1)_{B}$, which is a $1$-morphism in $\fkOr^d$ from $(M_0, \mu_0)_{B}$ to $(M_{1}, \mu_1)_{B}$.  Given two such oriented manifolds with boundary $(N_0,\eta_0,f_0)$ and $(N_1,\eta_1,f_1)$ a cobordism $(P,\rho, a, b)$ between them (a $2$-morphism in $\fkCob^{or}_d$) is homotopy equivalent to an oriented topological space whose boundary is the pushout 
$N_0 \coprod_{(M_0 \coprod M_1)} N_1.
$
Applying the Betti stack construction, one gets a filling of
\[(N_0,\eta_0)_{B} \coprod_{(M_0,-\mu_0 )\coprod (M_1,-\mu_1)} (N_1,\eta_1)_{B}
.\]
Any equivalence of cobordisms, gives an orientation preserving homeomorphism of the two topological spaces whose boundary is the pushout above and therefore induces an filleomorphism (see Definition \ref{defn:Filleo}) of the two corresponding fillings. 
\end{proof}

\end{document}